\documentclass[reqno,11pt]{amsart}

\usepackage[top=2cm,bottom=2cm,left=2.5cm,right=2.5cm]{geometry}

\usepackage{amsthm,amsmath,amssymb,dsfont}
\usepackage{mathrsfs,amsfonts,functan,extarrows,mathtools}
\usepackage[colorlinks]{hyperref}

\usepackage{marginnote}
\usepackage{xcolor}

\newtheorem{theo}{Theorem}[section]
\newtheorem{prop}[theo]{Proposition}

\newtheorem{rema}[theo]{Remark}
\newtheorem{lemm}[theo]{Lemma}

\numberwithin{equation}{section}
\allowdisplaybreaks

\makeatletter

\newcommand{\Rmnum}[1]{\expandafter\@slowromancap\romannumeral #1@}
\makeatother

\newcounter{wronumber}

\begin{document}

\title[Uniform Regularity and Invicid Limit for CNLC]
{Uniform Regularity and Vanishing Viscosity Limit for the Compressible
Nematic Liquid Crystal Flows in Three Dimensional Bounded Domain}

\author[J. Gao]{Jincheng Gao}
\address[Jincheng Gao]{\newline Institute of Applied Physics and Computational Mathematics,
100088, Beijing, P. R. China}
\email{gaojc1998@163.com}

\author[B.Guo]{Boling Guo}
\address[Boling Guo]{\newline Institute of Applied Physics and Computational Mathematics,
100088, Beijing, P. R. China}
\email{gbl@iapcm.ac.cn}

\author[Y. Liu]{Yaqing Liu}
\address[Yaqing Liu]{\newline Institute of Applied Physics and Computational Mathematics,
100088, Beijing, P. R. China}
\email{liuyaqing1981@163.com}


\begin{abstract}
In this paper, we study the uniform regularity and vanishing viscosity limit
for the compressible nematic liquid crystal flows in three dimensional bounded domain.
It is shown that there exists a unique strong solution for the compressible
nematic liquid crystal flows with boundary condition in a finite
time interval which is independent of the viscosity coefficient.
The solutions are uniform bounded in a conormal Sobolev space. Furthermore,
we prove that the density and velocity are uniform bounded in $W^{1, \infty}$,
and the director field is uniform bounded in $W^{3,\infty}$ respectively.
Based on these uniform estimates, one also obtains the convergence rate of
the viscous solutions to the inviscid ones with a rate of convergence.



\end{abstract}


\maketitle



\section{Introduction}

\quad In this paper, we investigate the motion of compressible nematic liquid crystal flows, which are governed by the following simplified version of the
Ericksen-Leslie equations as follows
\begin{equation}\label{eq1}
\left\{
\begin{aligned}
&\rho^\varepsilon_t+{\rm div} (\rho^\varepsilon u^{\varepsilon})=0,
& (x, t)\in \Omega \times (0, T),\\
&\rho^\varepsilon u^{\varepsilon}_t
+\rho^\varepsilon u^{\varepsilon}\cdot \nabla u^{\varepsilon}+\nabla p^{\varepsilon}
= \mu \varepsilon \Delta u^{\varepsilon}
+(\mu+\lambda)\varepsilon\nabla {\rm div}u^\varepsilon
-\nabla d^{\varepsilon} \cdot \Delta d^{\varepsilon},
&(x, t)\in \Omega \times (0, T),\\
&d^{\varepsilon}_t+u^{\varepsilon}\cdot \nabla d^{\varepsilon}
=\Delta d^{\varepsilon}+|\nabla d^{\varepsilon}|^2 d^{\varepsilon},
& (x, t)\in \Omega \times (0, T).
\end{aligned}
\right.
\end{equation}
Here $0 < T \le +\infty$ and $\Omega$ is a bounded smooth domain of $\mathbb{R}^3$.
The unknown functions $\rho^\varepsilon(x, t)$,
$u^\varepsilon(x, t)=(u^\varepsilon_1(x,t), u^\varepsilon_2(x,t), u^\varepsilon_3(x,t),)$
and $d^\varepsilon(x,t)=(d_1^\varepsilon(x,t), d_2^\varepsilon(x,t), d_3^\varepsilon(x,t))$
represent the density, velocity field of fluid and
the macroscopic average of the nematic liquid crystal orientation field respectively.
The scalar function $p^\varepsilon=p(\rho^\varepsilon)$
is the pressure function given by $\gamma-$law
$$
p(\rho)=\rho^\gamma \quad{\rm with}\quad \gamma >1.
$$
The viscous coefficients $\mu$ and $\lambda$ satisfy
the physical restrictions
\begin{equation*}
\mu>0, \quad 2\mu+3\lambda >0,
\end{equation*}
where the parameter $\varepsilon>0$ is the inverse of the Reynolds number.
For more results about the compressible Ericksen-Leslie system \eqref{eq1},
the readers can refer to \cite{{Huang-Wang-Wen1},{Huang-Wang-Wen2},{Hu-Wu},{Ding-Huang-Wen-Zi},{Jiang-Jiang-Wang},{Lin-Lai-Wawng},{Schade-Shibata},{Gao-Tao-Yao}} and references therein.
Corresponding to the system \eqref{eq1}, we impose the following Navier-slip type
and Neumann boundary
conditions:
\begin{equation}\label{bc1}
u^\varepsilon \cdot n=0, \quad ((Su^\varepsilon)n)_\tau=-(A u^\varepsilon)_\tau,
\quad {\rm and}
\quad \frac{ \partial d^\varepsilon}{\partial n}=0, \quad {\rm on}~\partial \Omega,
\end{equation}
where $A$ is a given smooth symmetric matrix(see \cite{Gie-Kelliher}),
$n$ is the outward unit vector normal to $\partial \Omega$,
$(A u^\varepsilon)_\tau$ represents the tangential component of $A u^\varepsilon$.
The strain tensor $Su^\varepsilon$ is defined by
\begin{equation*}
Su^\varepsilon=\frac{1}{2}\left((\nabla u^\varepsilon)+(\nabla u^\varepsilon)^t\right).
\end{equation*}
For any smooth solutions $v$, it is easy to check that
$$
(2S(v)n-(\nabla \times v)\times n)_\tau=-(2S(n)v)_\tau,
$$
see \cite{Xiao-Xin1} for detail. Then the boundary condition \eqref{bc1} can
be written in the form of the vorticity as
\begin{equation}\label{bc2}
u^{\varepsilon}\cdot n=0,
\quad n\times (\nabla \times u^{\varepsilon})=[Bu^{\varepsilon}]_\tau,
\quad {\rm and}\quad
\frac{\partial d^{\varepsilon}}{\partial n}=0,
\quad {\rm on}~\partial \Omega,
\end{equation}
where $B=2(A-S(n))$ is symmetric matrix.
Actually, it turns out that the form \eqref{bc2} will be more
convenient than \eqref{bc1} in the energy estimates,
see \cite{{Xiao-Xin},{Wang-Xin-Yong}}.

In this paper, we are interested in the existence of strong solution of \eqref{eq1}
with uniform bounds on an interval of time independent of viscosity
$\varepsilon\in (0, 1]$ and the vanishing
viscosity  limit to the corresponding
invicid nematic liquid crystal flows as $\varepsilon$ vanishes, i.e.
\begin{equation}\label{eq2}
\left\{
\begin{aligned}
&\rho_t+{\rm div} (\rho u)=0,
& (x, t)\in \Omega \times (0, T),\\
&\rho u_t+\rho u\cdot \nabla u+\nabla p= -\nabla d \cdot \Delta d,
&(x, t)\in \Omega \times (0, T),\\
&d_t+u\cdot \nabla d=\Delta d+|\nabla d|^2 d,
& (x, t)\in \Omega \times (0, T).
\end{aligned}
\right.
\end{equation}
with the boundary condition
\begin{equation}\label{bc3}
u\cdot n=0,
\quad {\rm and}\quad
\frac{\partial d}{\partial n}=0,
\quad {\rm on}~\partial \Omega.
\end{equation}

When the density and director field are constant scalar function and
constant vector field respectively, the systems \eqref{eq1} and \eqref{eq2}
are the well-known incompressible Navier-Stokes equations
and incompressible Euler equations respectively.
There is lots of literature on the uniform bounds and the vanishing viscosity
limit for the Navier-Stokes equations without boundaries,
for instance, \cite{{Constantin}, {Constantin-Foias},{Kato}, {Masmoudi}}.
The appearance of boundary gives arise to the time of existence
$T^\varepsilon$ depending on the viscosity coefficient,
and it is difficult to prove that it stays bounded away from zero.
Nevertheless, in
a domain with boundaries, for some special types of Navier-slip boundary conditions,
some uniform $H^3$ (or $W^{2,p}$, with $p$ large enough) estimates and
a uniform time of existence for Navier-Stokes when the viscosity goes to zero
have recently been obtained (see \cite{{Beira1},{Beira2},{Xiao-Xin}}).
It is easy to see that, for these special
boundary conditions, the main part of the boundary layer vanishes, which allows
this uniform control in some limited regularity Sobolev space.
Recently, Masmoudi and Rousset \cite{Masmoudi-Rousset}
established conormal uniform estimates for three-dimensional general
smooth domains with the Naiver-slip boundary condition
and obtained  convergence of the viscous solutions to the inviscid ones
by a compact argument.
Based on the uniform estimates in \cite{Gie-Kelliher},
better convergence with rates have been studied in \cite{Gie-Kelliher}
and \cite{Xiao-Xin2}. In particular, Xiao
and Xin \cite{Xiao-Xin2} have proved the convergence in $L^\infty(0, T ;H^1)$
with a rate of convergence.
Motivated by the work of \cite{Masmoudi-Rousset} and Xin \cite{Xiao-Xin2},
We \cite{Gao-Guo-Xi} investigated the vanishing viscosity limit of incompressible
nematic liquid crystal flows. More precisely, we proved that there exists a unique
strong solution for the incompressible nematic liquid crystal flows
in a finite time interval which is independent of the viscosity coefficient and
obtained the convergence rate of the viscous solutions to the inviscid
ones with a rate of convergence.

For the compressible Navier-Stokes equations, Paddick \cite{Paddick} obtained
uniform estimates for the solutions of the compressible isentropic Navier-Stokes
equations in the 3-D half-space with a Navier boundary condition, which was
improved by Wang et al. \cite{Wang-Xin-Yong} to generalized bounded domain.
Specially, Wang et al. \cite{Wang-Xin-Yong} shown that the boundary layers for
the density must be weaker than the one for the velocity and
established the convergence of the viscous solutions to the inviscid ones.
For more results about the inviscid limit for the compressible
Navier-Stokes equations, the readers can refer to
\cite{{Xin-Yanagisawa},{Wang-Williams}} and the references therein.
Motivated the work of \cite{Gao-Guo-Xi} and \cite{Wang-Xin-Yong}, we hope
the investigate the vanishing viscosity limit for the compressible
nematic liquid crystal flows \eqref{eq1}.

Before stating our main results, we first explain the notations and conventions
used throughout this paper. Similar to \cite{{Masmoudi-Rousset},{Wang-Xin-Yong}},
one assumes that the bounded domain $\Omega \subset \mathbb{R}^3$ has a covering that
\begin{equation}
\Omega \subset \Omega_0 \cup_{k=1}^n \Omega_k,
\end{equation}
where $\overline{\Omega}_0$, and in each $\Omega_k$ there exists a function
$\psi_k$ such that
\begin{equation*}
\begin{aligned}
&\Omega \cap \Omega_k =\{x=(x_1, x_2, x_3)| x_3 > \psi_k(x_1, x_2)\}\cap \Omega_k,\\
&\partial \Omega \cap \Omega_k=\{x_3=\psi_k(x_1, x_2)\}\cap \Omega_k.
\end{aligned}
\end{equation*}
Here, $\Omega$ is said to be $\mathcal{C}^m$ if the functions $\psi_k$
are a $\mathcal{C}^m-$function. To define the conormal Sobolev spaces,
one considers $(Z_k)_{1\le k \le N}$ to be a finite set of generators of vector fields
that are tangential to $\partial \Omega$, and sets
\begin{equation*}
H_{co}^m=\{f\in L^2(\Omega)|Z^I f\in L^2(\Omega), ~ {\rm for}~|I|\le m\},
\end{equation*}
where $I=(k_1,..., k_m)$. The following notations will be used
$$
\begin{aligned}
&\|u\|_m^2=\|u\|_{H^m_{co}}^2=\sum_{j=1}^3 \sum_{|I|\le m}\|Z^I u_j\|_{L^2}^2,\\
&\|u\|_{m, \infty}^2=\sum_{|I|\le m}\|Z^I u\|_{L^\infty}^2,
\end{aligned}
$$
and
$$
\|\nabla Z^m u\|^2=\sum_{|I|=m}\|\nabla Z^I u\|_{L^2}^2.
$$
Noting that by using the covering of $\Omega$, one can always assume that each
vector field $(p^\varepsilon, u^\varepsilon, d^\varepsilon)$
is supported in one of the $\Omega_i$, and moreover,
in $\Omega_0$ the norm $\|\cdot\|_m$ yields a control of the standard $H^m$ norm,
whereas if $\Omega_i \cap \partial \Omega\neq\emptyset$, there is no control of the
normal derivatives.

Since $\partial \Omega$ is given locally by $x_3=\psi(x_1, x_2)$
(we omit the subscript $j$ of notational convenience), it is convenient to
use the coordinates
$$
\Psi:~(y, z)\mapsto (y, \psi(y)+z)=x.
$$
A  basis is thus given by the vector fields $(e_{y^1}, e_{y^2}, e_z)$,
where $e_{y^1}=(1, 0, \partial_1 \psi)^t,~e_{y^2}=(0, 1, \partial_2 \psi)^t$,
and $e_z=(0, 0, -1)^t$. On the boundary, $e_{y^1}$ and $e_{y^2}$ are tangent to
$\partial \Omega$, and in general, $e_z$ is not a normal vector field.
By using this parametrization, one can take as suitable vector fields compactly
supported in $\Omega_j$ in the definition of the $\|\cdot \|_m$ norms
\begin{equation*}
Z_i=\partial_{y^i}=\partial_i+\partial_i \psi \partial_z,~i=1,2,~~
Z_3=\varphi(z)\partial_z,
\end{equation*}
where $\varphi(z)=\frac{z}{1+z}$ is smooth, supported in $\mathbb{R}_+$
with the property $\varphi(0)=0, \varphi'(0)>0, \varphi(z)>0$ for $z>0$.
It is easy to check that
\begin{equation}\label{pc}
Z_k Z_j=Z_j Z_k,~~j, k =1,2,3,
\end{equation}
and
$$
\partial_z Z_i=Z_i \partial_z,~i=1,2;\quad \partial_z Z_3\neq Z_3 \partial_z.
$$
We shall still denote by $\partial_j,~j=1,2,3,$ or
$\nabla$ the derivatives in the physical space. The coordinates of a vector field
$u$ in the basis $(e_{y^1}, e_{y^2}, e_z)$ will be denoted by $u^i$, and thus
$$
u=u^1 e_{y^1}+u^2 e_{y^2}+u^3 e_{z}.
$$
We shall denote by $u_j$ the coordinates in the standard basis of $\mathbb{R}^3$,
i.e, $u=u_1 e_1+u_2 e_2+u_3 e_3$. Denote by $n$ the unit outward normal in the
physical space which is given locally by
\begin{equation}
n(x)\equiv n(\Psi(y,z))=\frac{1}{\sqrt{1+|\nabla \psi(y)|^2}}
\left(
\begin{array}{c}
\partial_1 \psi(y)\\
\partial_2 \psi(y)\\
-1
\end{array}
\right)
\triangleq \frac{-N(y)}{\sqrt{1+|\nabla \psi(y)|^2}},
\end{equation}
and by $\Pi$ the orthogonal projection
\begin{equation}
\Pi u\equiv\Pi(\Psi(y,z))u=u-[u\cdot n(\Psi(y,z))]n(\Psi(y,z)),
\end{equation}
which gives the orthogonal projection on to the tangent space of the boundary.
Note that $n$ and $\Pi$ are defined in the whole $\Omega_k$ and do not depend on $z$.
For later use and notational convenience, set
\begin{equation*}
\mathcal{Z}^\alpha=\partial_t^{\alpha_0}Z^{\alpha_1}
=\partial_t^{\alpha_0}Z_1^{\alpha_{11}}Z_2^{\alpha_{12}}Z_3^{\alpha_{13}},
\end{equation*}
where $\alpha, \alpha_0$ and  $\alpha_1$ are the differential multi-indices with
$\alpha=(\alpha_0, \alpha_1), \alpha_1=(\alpha_{11}, \alpha_{12}, \alpha_{13})$
and we also use the notation
\begin{equation}\label{sdef1}
\|f(t)\|_{\mathcal{H}^{m}}^2
=\|f(t)\|_{\mathcal{H}^{m}}^2
=\sum_{|\alpha|\le m}\|\mathcal{Z}^\alpha f(t)\|_{L^2_x}^2,
\end{equation}
and
\begin{equation}\label{sdef2}
\|f(t)\|_{\mathcal{H}^{k,\infty}}
=\sum_{|\alpha|\le k}\|\mathcal{Z}^\alpha f(t)\|_{L^\infty_x}
\end{equation}
for smooth space-time function $f(x,t)$.
Throughout this paper, the positive generic constants that are independent
of $\varepsilon$ are denoted by $c, C$.
Denote by $C_k$ a positive constant independent of $\varepsilon \in (0, 1]$
which depends only on the $\mathcal{C}^k-$norm of the functions $\psi_j,~j=1,...,n$.
Here, $\|\cdot\|_{L^2}$ denotes the standard $L^2(\Omega; dx)$ norm,
and $\|\cdot\|_{H^m}(m=1,2,3,...)$ denotes the Sobolev $H^m(\Omega, dx)$ norm.
The notation $|\cdot|_{H^m}$ will be used for the standard Sobolev norm
of functions defined on $\partial \Omega$. Note that this norm involves only
tangential derivatives. $P(\cdot)$ denotes a polynomial function.

Since the boundary layers may appear in the presence of physical boundaries,
in order to obtain the uniform estimates for solutions to the nematic liquid
crystal flows with Navier-slip and Neumann boundary conditions, we needs to find a
suitable functional space.
In the spirit of Wang et al.\cite{Wang-Xin-Yong},
we also investigate the solutions of the nematic liquid crystal flows
in Conormal Sobolev space.
Hence, the functional space should include some information for the
direction field $d$. On the other hand, due to the nonlinear higher
order derivatives term $\nabla d\cdot \Delta d$, one should control this term
by using the dissipative term $\Delta d$ on the right hand side of the equation \eqref{eq1}$_2$
which involving the time derivatives term $d_t$. Hence, we also include some
information involving the time derivatives in the functional space.
Therefore, we define the functional space $X^\varepsilon_m(T)$ for a pair of function
$(u, p, d)(x, t)$ as follows
\begin{equation}
X^\varepsilon_m(T)=\{(p, u, d)\in L^\infty([0, T], L^2);
       ~\underset{0\le t\le T}{\rm esssup}\|(p, u, d)(t)\|_{X^\varepsilon_m}<+\infty\},
\end{equation}
where the norm $\|(\cdot, \cdot)\|_{X^\varepsilon_m}$ is given by
\begin{equation}
\begin{aligned}
\|(p, u, d)(t)\|_{X^\varepsilon_m}
=
&\|(u, p)(t)\|_{\mathcal{H}^{m}}^2+\|d(t)\|_{L^2}^2
+\|\nabla d(t)\|_{\mathcal{H}^{m}}^2+\|(\nabla u, \Delta d)(t)\|_{\mathcal{H}^{m-1}}^2\\
&+\|\nabla u(t)\|_{\mathcal{H}^{1,\infty}}^2
+\sum_{k=0}^{m-2}\|\partial_t^k \nabla p(t)\|_{m-1-k}^2
+\varepsilon\|\nabla \partial_t^{m-1} p(t)\|_{L^2}^2\\
&+\|\Delta p(t)\|_{\mathcal{H}^1}^2+\varepsilon\|\Delta p(t)\|_{\mathcal{H}^2}^2.
\end{aligned}
\end{equation}
In the present paper, we supplement the nematic liquid crystal flows system \eqref{eq1}
with initial data
\begin{equation}\label{ID}
(p^\varepsilon, u^\varepsilon, d^\varepsilon)(x, 0)
=(p^\varepsilon_0, u^\varepsilon_0, d^\varepsilon_0)(x),
\end{equation}
such that
\begin{equation}
0< \frac{1}{\hat{C}_0}\le \rho^\varepsilon_0\le \hat{C}_0<\infty,
\end{equation}
and
\begin{equation}\label{IDB}
\begin{aligned}
&\underset{0< \varepsilon \le 1}{\sup}
\|(p^\varepsilon_0, u^\varepsilon_0, d^\varepsilon_0)\|_{X^\varepsilon_m}\\
&=
\underset{0<\varepsilon \le 1}{\sup}\{\|(u_0^\varepsilon, p_0^\varepsilon)\|_{\mathcal{H}^{m}}^2
       +\|d_0^\varepsilon\|_{L^2}^2
       +\|\nabla d_0^\varepsilon\|_{\mathcal{H}^{m}}^2
       +\|\nabla u_0^\varepsilon\|_{\mathcal{H}^{m-1}}^2
       +\sum_{k=0}^{m-2}\|\partial_t^k \nabla p_0^\varepsilon\|_{m-1-k}^2\\
&\quad  \quad \quad \quad
+\varepsilon \|\nabla \partial_t^{m-1}p_0^\varepsilon\|_{L^2}^2
+\|\Delta p_0^\varepsilon\|_{\mathcal{H}^1}^2
+\varepsilon\|\Delta p_0^\varepsilon\|_{\mathcal{H}^2}^2
+\|\Delta d_0^\varepsilon\|_{\mathcal{H}^{m-1}}^2
       +\|\nabla u_0^\varepsilon\|_{\mathcal{H}^{1,\infty}}^2\}\le \widetilde{C}_0,
\end{aligned}
\end{equation}
where $\widetilde{C}_0$ is a positive constant independent of $\varepsilon \in (0,1]$,
and the time derivatives of initial data are defined through the equation \eqref{eq1}.
Thus, the initial data $(\rho_0^\varepsilon, u_0^\varepsilon, d_0^\varepsilon)$
is assumed to have a higher space regularity and compatibilities.
Notice that the a priori estimates in Theorem  \ref{Theoream3.1}
below are obtained in the case that the
approximate solution is sufficiently smooth up to the boundary, and therefore,
in order to obtain a selfconstained result, one needs to assume the approximated
initial data satisfies the boundary compatibilities condition \eqref{bc2}.
For the initial data $(\rho_0^\varepsilon, u^\varepsilon_0, d^\varepsilon_0)$ satisfying
\eqref{ID}, it is not clear if there exists an approximate sequences
$(\rho^{\varepsilon,\delta}_0, u^{\varepsilon,\delta}_0, d^{\varepsilon,\delta}_0)$
($\delta$ being a regularization parameter) which satisfy the boundary
compatibilities and $\|(p^{\varepsilon,\delta}_0-p^\varepsilon_0, u^{\varepsilon,\delta}_0-u^\varepsilon_0, d^{\varepsilon,\delta}_0-d^\varepsilon_0)\|_{X_m^\varepsilon}\rightarrow 0$ as
$\delta \rightarrow 0$. Therefore, we set
\begin{equation}
\begin{aligned}
X_{n,m}^{\varepsilon,ap}
=\left\{(p, u, d)\in [H^{4m}(\Omega)]^2\times H^{4(m+1)}(\Omega)\right.
&|\partial_t^k p, \partial_t^k u, \partial_t^k d,
k=1,...,m{\rm~ are~ defined }\\
&{\rm ~through ~the~ equations}~\eqref{eq1}~
{\rm and}~\partial_t^k u,\\
&~~\partial_t^k \nabla d, k=0,...,m-1,{\rm satisfy~the}\\
&{\rm ~boundary~compatibility~condition}\}.
\end{aligned}
\end{equation}
and
\begin{equation}
X_{n,m}^\varepsilon={\rm~ the ~closure~ of~}X_{n,m}^{\varepsilon, ap}~{\rm in~ the~ norm~}\|(\cdot,\cdot)\|_{X_m^\varepsilon}.
\end{equation}

Now, we state the first results concerning the uniform regularity for the nematic
liquid crystal flows \eqref{eq1}, \eqref{bc2} and \eqref{ID} as follows.

\begin{theo}[Uniform Regularity]\label{Theorem1.1}
Let $m$ be an integer satisfying $m \ge 6$, $\Omega$ be a $\mathcal{C}^{m+2}$ domain,
and $A\in C^{m+1}(\partial \Omega)$. Consider the initial data
$(p_0^\varepsilon, u_0^\varepsilon, d_0^\varepsilon)\in X_{n,m}^\varepsilon$
satisfy \eqref{IDB} and $|d^\varepsilon_0|=1$
in $\overline{\Omega}$.
Then, there exists a time $T_0>0$ and $\widetilde{C}_1>0$ independent of
$\varepsilon \in (0, 1]$, such that there exits a unique solution
of \eqref{eq1}, \eqref{bc2} and \eqref{ID} which is defined on $[0, T_0]$
and satisfies the estimates
\begin{equation}\label{111}
\begin{aligned}
&\underset{0\le t \le T_0}{\sup}
(\|d^\varepsilon(t)\|_{L^2}^2
+\|(u^\varepsilon, p^\varepsilon, \nabla d^\varepsilon)(t)\|_{\mathcal{H}^{m}}^2
+\|(\nabla u^\varepsilon, \Delta d^\varepsilon)(t)\|_{\mathcal{H}^{m-1}}^2
+\|\nabla u^\varepsilon(t)\|_{\mathcal{H}^{1,\infty}}^2)\\
&\quad
+\underset{0\le t \le T_0}{\sup}(
\sum_{k=0}^{m-2}\|\partial_t^k \nabla p^\varepsilon(t)\|_{m-1-k}^2
+\varepsilon \|\partial_t^{m-1}\nabla p^\varepsilon(t)\|_{L^2}^2
+\|\Delta p^\varepsilon (t)\|_{\mathcal{H}^1}^2
+\varepsilon\|\Delta p^\varepsilon (t)\|_{\mathcal{H}^2}^2)\\
&\quad +\int_0^{T_0}(\|\nabla \partial_t^{m-1}p^\varepsilon(t)\|_{L^2}^2
+\|\Delta p^\varepsilon(t)\|_{\mathcal{H}^2}^2)dt
+\varepsilon\int_0^{T_0} \|\nabla u^\varepsilon (t)\|_{\mathcal{H}^m}^2 dt\\
&\quad
+\varepsilon^2 \int_0^{T_0}\|\nabla^2 \partial_t^{m-1}u^\varepsilon(t)\|_{L^2}^2 dt
+\varepsilon\sum_{k=0}^{m-2}\int_0^{T_0}\|\nabla^2 \partial_t^{k}u^\varepsilon(t)\|_{m-k-1}^2dt\\
&\quad +\int_0^{T_0} \|\Delta d^\varepsilon\|_{\mathcal{H}^{m}}^2dt
+\int_0^{T_0} \|\nabla \Delta d^\varepsilon\|_{\mathcal{H}^{m-1}}^2 dt
\le \widetilde{C}_1,
\end{aligned}
\end{equation}
and
\begin{equation}
\frac{1}{2\widehat{C}_0}\le \rho^\varepsilon(t)\le 2\widehat{C}_0,
\quad t \in [0, T_0],
\end{equation}
where $\widetilde{C}_1$ depends only on $\widehat{C}_0, \widetilde{C}_0$ and $C_{m+2}$.
\end{theo}

\begin{rema}
For $(p_0^\varepsilon, u_0^\varepsilon, d_0^\varepsilon)\in X^\varepsilon_{n,m}$,
it must hold that $u_0^\varepsilon \cdot n|_{\partial \Omega}=0$,
$((Su_0^\varepsilon)n)_\tau|_{\partial \Omega}=-(A u_0^\varepsilon)_\tau|_{\partial \Omega}$,
and $n\cdot \nabla d_0^\varepsilon|_{\partial \Omega}=0$
in the trace sense for every fixed $\varepsilon \in (0, 1]$.
\end{rema}

The main steps of the proof of Theorem \ref{Theorem1.1} are the following.
First, we obtained a conormal energy estimates for
$(p^\varepsilon, u^\varepsilon, \nabla d^\epsilon)$
in $\mathcal{H}^{m}-$norm.
The second step is to give the estimate for $\|\partial_n u^\varepsilon\|_{\mathcal{H}^{m-1}}$.
In order to obtain this estimate by an energy method, $\partial_n u^\varepsilon$
is not a convenient quantity because it does not vanish on the boundary.
Similar to  Wang et al.\cite{Wang-Xin-Yong}, $\partial_n u^\varepsilon$ can be controlled
by $\partial_n u^\varepsilon \cdot n(~{\rm or~}{\rm div} u^\varepsilon)$ and
$(\partial_n u)_\tau$.
In order to give the estimate for $(\partial_n u^\varepsilon)_\tau$, one choose
the convenient quantity $\eta=w^\varepsilon \times n+(Bu^\varepsilon)_\tau$
with a homogeneous Dirichlet boundary conditions.
The third step is to give the estimates for $\Delta d^\varepsilon$
and ${\rm div}u^\varepsilon$.
Indeed, it is easy to obtain the estimate for the quantity $\Delta d^\varepsilon$
since there exists a dissipative
term $\Delta d^\varepsilon$ on the right-hand side of \eqref{eq1}$_3$.
In the spirit of Wang et al. \cite{Wang-Xin-Yong}, we obtain a control of
$\sum_{j=0}^{m-2}\|\partial_t^j({\rm div}u^\varepsilon, \nabla p^\varepsilon)\|_{m-1-j}^2$
at the cost that the term
$\int_0^t \|\nabla \mathcal{Z}^{m-2}{\rm div}u^\varepsilon\|_{L^2}^2 d\tau$
appears in the right-hand side of the inequality.
Following the idea as Wang et al. \cite{Wang-Xin-Yong}, we can
obtain the uniform estimates for
$\int_0^t \|\partial_t^{m-1}\nabla p^\varepsilon\|_{L^2}^2 d\tau$
and get a control of $\|\partial_t^{m-1}{\rm div}u^\varepsilon\|_{L^2}^2$
in terms of
$\sum_{j=0}^{m-2}\|\partial_t^j(\nabla u^\varepsilon, \nabla p^\varepsilon)\|_{m-1-j}^2$
and $\|(p^\varepsilon, u^\varepsilon)\|_{\mathcal{H}^m}^2$.
The fourth step is to estimate $\|\Delta d^\varepsilon\|_{W^{1,\infty}}$.
Indeed, this estimate is easy to obtain since there exists a dissipation term $\Delta d^\varepsilon$ on the right-hand side of \eqref{eq1}$_3$.
The fifth step is to estimate $\|\nabla u^\varepsilon\|_{\mathcal{H}^{1,\infty}}$.
In fact, it suffices to estimate $\|(\partial_n u^\varepsilon)_\tau\|_{\mathcal{H}^{1,\infty}}$
since the other terms can be estimated by the Sobolev embedding.
We choose an equivalent quantity such that it satisfies a homogeneous Dirichlet
condition and solves a convection-diffusion equation at the leading order.
The last step is to obtain the uniform estimate of
$\|\Delta p^\varepsilon\|_{\mathcal{H}^1}$, which gives a control of
$\|\nabla p^\varepsilon\|_{\mathcal{H}^{1,\infty}}$ from Proposition \ref{prop2.3}.
Then Theorem \ref{Theorem1.1} can be proved by these a priori estimates and
a classical iteration method.

Next, we hope to prove the vanishing viscosity limit with
rates of convergence, which can be stated as follows.

\begin{theo}[Inviscid Limit]\label{Theorem1.2}
Let $(\rho, u, d)(t)\in L^\infty(0, T_1; H^3\times H^3 \times H^4)$
be the smooth solution to the equation
\eqref{eq2} and boundary condition \eqref{bc2}
with initial data $(\rho_0, u_0, d_0)$ satisfying
\begin{equation}\label{121}
(\rho_0, u_0, d_0)\in (H^3 \times H^3 \times H^4)
\cap X_{n,m}^\varepsilon~{\rm with}~m \ge 6.
\end{equation}
Let $(\rho^\varepsilon, u^\varepsilon, d^\varepsilon)(t)$
be the solution to the initial boundary
value problem of the nematic liquid crystal flows \eqref{eq1}, \eqref{bc1}
with initial data $(\rho_0, u_0, d_0)$ satisfying \eqref{121}.
Then, there exists $T_2={\rm min}\{T_0, T_1\}>0$, which is independent of
$\varepsilon>0$, such that
\begin{equation}\label{122}
\|(\rho^\varepsilon-\rho, u^\varepsilon-u)(t)\|_{L^2}^2+\|(d^\varepsilon-d)(t)\|_{H^1}^2
\le C\varepsilon^{\frac{3}{2}}, \quad t\in [0, T_2]
\end{equation}
\begin{equation}\label{123}
\|(\rho^\varepsilon-\rho, u^\varepsilon-u)(t)\|_{H^1}^2
\le C\varepsilon^{\frac{1}{6}},\quad
\|(d^\varepsilon-d)(t)\|_{H^2}^2\le C\varepsilon^{\frac{1}{2}}, \quad t\in [0, T_2]
\end{equation}
and
\begin{equation}\label{124}
\|(\rho^\varepsilon-\rho, u^\varepsilon-u)\|_{L^\infty(0, T_2; L^\infty(\Omega))}
+\|(d^\varepsilon-d)\|_{L^\infty(0,T_2; W^{1,\infty}(\Omega))}
\le C \varepsilon^{\frac{3}{10}},
\end{equation}
which $C$ depends only on the norm $\|(\rho_0, u_0)\|_{H^3},
\|d_0\|_{H^4}$
and $\|(p(\rho_0), u_0, d_0)\|_{X_{m,n}^\varepsilon}$.
\end{theo}


The rest of the paper is organized as follows:
In section \ref{Preliminary}, we collect some inequalities that will be used later.
In section \ref{estimates}, the a priori estimates in Theorem \ref{Theoream3.1}
are proved. By using these a priori estimates, one give the proof for the
Theorem \ref{Theorem1.1} in section \ref{Proof1}.
Based on the uniform estimates obtained in Theorem \ref{Theorem1.1},
we establish the convergence rate for the solutions from \eqref{eq1} to \eqref{eq2}
and complete the proof for Theorem \ref{Theorem1.2}.

\section{Preliminaries}\label{Preliminary}
\quad The following lemma \cite{{Xiao-Xin},{Temam}} allows us to control
the $H^m(\Omega)$-norm
of a vector valued function $u$ by its $H^{m-1}-$norm of $\nabla \times u$
and ${\rm div}u$, together with the $H^{m-\frac{1}{2}}(\partial \Omega)$
of $u\cdot n$.

\begin{prop}\label{prop2.1}
Let $m \in \mathbb{N}_+$ be an integer. Let $u\in H^m$ be a vector-valued function.
Then, there exists a constant $C>0$ independent of $u$, such that
\begin{equation}\label{21}
\|u\|_{H^m}\le C(\|\nabla \times u\|_{H^{m-1}}+\|{\rm div}u\|_{H^{m-1}}
                 +\|u\|_{H^{m-1}}+|u\cdot n|_{H^{m-\frac{1}{2}}(\partial \Omega)}),
\end{equation}
and
\begin{equation}\label{22}
\|u\|_{H^m}\le C(\|\nabla \times u\|_{H^{m-1}}+\|{\rm div}u\|_{H^{m-1}}
                 +\|u\|_{H^{m-1}}+|n \times u|_{H^{m-\frac{1}{2}}(\partial \Omega)}).
\end{equation}
\end{prop}

In this paper, one repeatedly use the Gagliardo-Nirenberg-Morser type inequality,
whose proof can be find in \cite{Gues}. First, define the space
\begin{equation}
W^m(\Omega \times [0, T])=\{f(x,t)\in L^2(\Omega \times [0, T])
|\mathcal{Z}^\alpha f \in L^2(\Omega \times [0, T]), |\alpha|\le m\}.
\end{equation}
Then, the Gagliardo-Nirenberg-Morser type inequality  can be stated as follows:
\begin{prop}\label{prop2.2}
For $u, v \in L^\infty(\Omega \times [0, T])\cap \mathcal{W}^m(\Omega \times [0, T])$
with $m \in \mathbb{N}_+$ an integer, it holds that
\begin{equation}\label{23}
\int_0^t \|(\mathcal{Z}^\beta u \mathcal{Z}^\gamma v)(\tau)\|_{L^2}^2 d\tau
\lesssim \|u\|_{L^\infty_{t,x}}^2\int_0^t \|v(\tau)\|_{\mathcal{H}^m}^2 d\tau
+\|v\|_{L^\infty_{t,x}}^2\int_0^t \|u(\tau)\|_{\mathcal{H}^m}^2 d\tau,
~|\beta|+|\gamma|=m.
\end{equation}
\end{prop}

We also need the following anisotropic Sobolev embedding and trace theorems,
refer to \cite{{Wang-Xin-Yong}}.
\begin{prop}\label{prop2.3}
Let $m_1 \ge 0, m_2 \ge 0$ be integers and $f\in H_{co}^{m_1}(\Omega)\cap H_{co}^{m_2}(\Omega)$
and $\nabla f\in H_{co}^{m_2}(\Omega)$.\\
$(1)$ The following anisotropic Sobolev embedding holds:
\begin{equation}\label{24}
\|f\|_{L^\infty}^2\le C(\|\nabla f\|_{H^{m_2}_{co}}+\|f\|_{H^{m_2}_{co}})
\cdot \|f\|_{H^{m_1}_{co}},
\end{equation}
provided $m_1+m_2 \ge 3$.\\
$(2)$The following trace estimate holds:\\
\begin{equation}\label{25}
|f|_{H^s(\partial \Omega)}^2 \le  C(\|\nabla f\|_{H^{m_2}_{co}}+\|f\|_{H^{m_2}_{co}})
\cdot \|f\|_{H^{m_1}_{co}},
\end{equation}
provided $m_1+m_2 \ge 2s \ge 0$.
\end{prop}

\section{A priori estimates}\label{estimates}

\quad The aim of this section is to prove the following a priori estimates,
which are crucial to prove Theorem \eqref{Theorem1.1}.
For notational convenience, we drop the superscript $\varepsilon$ throughout
this section.

\begin{theo}[a priori estimates]\label{Theoream3.1}
Let $m$ be an integer satisfying $m \ge 6$, $\Omega$ be a
$C^{m+2}$ domain, and $A\in \mathcal{C}^{m+1}(\partial \Omega)$.
For sufficiently smooth solutions defined on $[0, T]$ of
\eqref{eq1} and \eqref{bc1}, then it holds that
\begin{equation}\label{31a}
|\rho(x, 0)|{\rm exp}\left(-\int_0^t \|{\rm div}u(\tau)\|_{L^\infty}d\tau\right)
\le \rho(x,t) \le |\rho(x, 0)|{\rm exp}\left(\int_0^t \|{\rm div}u(\tau)\|_{L^\infty}d\tau\right),
\end{equation}
for $(x, t)\in \Omega \times [0, T]$. In addition, if
\begin{equation}\label{31b}
0< c_0 \le \rho(x,t) \le \frac{1}{c_0}<\infty,
\quad (x, t)\in \Omega \times [0, T],
\end{equation}
where $c_0$ is any given small positive constant, then the following
a priori estimate holds
\begin{equation}\label{31c}
\begin{aligned}
&N_m(t)
+\int_0^t(\|\nabla \partial_t^{m-1}p(\tau)\|_{L^2}^2
+\|\Delta p(\tau)\|_{\mathcal{H}^2}^2)d\tau
+\varepsilon \int_0^t \|\nabla u(\tau)\|_{\mathcal{H}^{m}}^2 d\tau\\
&\quad \quad
+\varepsilon \sum_{k=0}^{m-2}\int_0^t \|\nabla^2 \partial_t^k u(\tau)\|_{m-1-k}^2d\tau
+\varepsilon^2 \int_0^t \|\nabla^2 \partial_t^{m-1}u(\tau)\|_{L^2}^2 d\tau\\
&\quad \quad
+\int_0^t \|\Delta d(\tau)\|_{\mathcal{H}^{m}}^2d\tau
+\int_0^t\|\nabla \Delta d(\tau)\|_{\mathcal{H}^{m-1}}^2d\tau\\
&
\le \widetilde{C}_2C_{m+2}\left\{P(N_m(0))+P(N_m(t))\int_0^t P(N_m(\tau)) d\tau\right\},
\quad \forall t\in [0, T],
\end{aligned}
\end{equation}
where $\widetilde{C}_2$ depends only on $\frac{1}{c_0},$ $P(\cdot)$
is a polynomial, and
\begin{equation}\label{def1}
\begin{aligned}
N_m(t)\triangleq
\underset{0\le \tau \le t}{\sup}
&\left\{1+\|(p,u)(\tau)\|_{\mathcal{H}^{m}}^2
       +\|d(\tau)\|_{L^2}^2
       +\|\nabla d(\tau)\|_{\mathcal{H}^{m}}^2
       +\|\nabla u(\tau)\|_{\mathcal{H}^{m-1}}^2\right.\\
&\quad +\|\Delta d(\tau)\|_{\mathcal{H}^{m-1}}^2
       +\sum_{k=0}^{m-2}\|\partial_t^k \nabla p(\tau)\|_{m-1-k}^2
       +\varepsilon\|\nabla \partial_t^{m-1}p(\tau)\|_{L^2}^2\\
&\quad \left.+\|\Delta p(\tau)\|_{\mathcal{H}^1}^2
       +\varepsilon\|\Delta p(\tau)\|_{\mathcal{H}^2}^2
       +\|\nabla u(\tau)\|_{\mathcal{H}^{1,\infty}}^2\right\}.
\end{aligned}
\end{equation}
\end{theo}

Throughout this section, we shall work on the interval of time
$[0, T]$ such that $c_0 \le \rho(x, t)\le \frac{1}{c_0}$.
Furthermore, we point out that the generic constant $C$
may depend on $\frac{1}{c_0}$ in this section.
Since the proof of Theorem \ref{Theoream3.1} is quite lengthy and involved,
we divide the proof into the following several subsections.

\subsection{Conormal Energy Estimates for $\rho, u$ and $\nabla d$}

\quad For any smooth function $f$, notice that
\begin{equation*}
\Delta f=\nabla {\rm div}f-\nabla \times (\nabla \times f),
\end{equation*}
and then $\eqref{eq1}_1$ can be written as
\begin{equation}\label{eq1-1}
\rho u_t+\rho u\cdot \nabla u+\nabla p
=-\mu \varepsilon \nabla \times (\nabla \times u)
+(2\mu+\lambda)\varepsilon \nabla {\rm div}u-\nabla d\cdot \Delta d.
\end{equation}

In this subsection, we first give the basic a priori $L^2$ estimate which holds
for \eqref{eq1} and \eqref{bc2}.

\begin{lemm}\label{lemma3.2}
For a smooth solution to \eqref{eq1} and \eqref{bc2}, it holds that
for $\varepsilon \in (0, 1]$
\begin{equation}\label{311}
\begin{aligned}
&\int (\frac{1}{2}\rho|u|^2+\frac{\gamma}{\gamma-1}\rho^\gamma
+\frac{1}{2}|\nabla d|^2) dx
+c_1 \varepsilon \int_0^t \|\nabla  u\|_{L^2}^2 d\tau
+\int_0^t \|\Delta d\|_{L^2}^2  d\tau\\
&\le \int (\frac{1}{2}\rho_0|u_0|^2+\frac{\gamma}{\gamma-1}\rho_0^\gamma
+\frac{1}{2}|\nabla d_0|^2) dx
+\|\nabla d\|_{L^\infty}^2 \int_0^t\|\nabla d\|_{L^2}^2 d\tau
+C_2 \int_0^t \|u\|_{L^2}^2 d\tau.
\end{aligned}
\end{equation}
\end{lemm}
\begin{proof}
Multiplying \eqref{eq1}$_2$ by $d$, one arrives at
\begin{equation*}
\frac{d}{dt}\frac{1}{2}\int (|d|^2-1)dx
+\int u\cdot \nabla (|d|^2-1)dx
=\int \Delta d \cdot d ~dx+\int |\nabla d|^2 |d|^2 dx,
\end{equation*}
which, integrating by part and applying the boundary condition \eqref{bc2}, yields that
\begin{equation}\label{312-1}
\frac{d}{dt}\int(|d|^2-1)dx+2\int (|d|^2-1)(|\nabla d|^2-{\rm div}u) dx=0.
\end{equation}
In view of the Gr\"{o}nwall inequality, one deduces from the identity \eqref{312-1} that
\begin{equation}
|d|=1\quad {\rm in}~\overline{\Omega}.
\end{equation}
Multiplying \eqref{eq1-1} by $u$, integrating by parts
and applying the boundary condition \eqref{bc2}, we find
\begin{equation}\label{312}
\begin{aligned}
&\frac{1}{2} \frac{d}{dt}\int \rho |u|^2 dx+\int \nabla p \cdot u~ dx
+\mu \varepsilon \int \nabla \times (\nabla \times u)\cdot u ~dx\\
&=(2\mu+\lambda)\varepsilon \int \nabla {\rm div}u \cdot u~dx
-\int (u \cdot \nabla) d \cdot \Delta d ~dx.
\end{aligned}
\end{equation}
By virtue of the equation \eqref{eq1}$_1$, one deduces that
\begin{equation}
\int \nabla p\cdot u~dx
=\frac{\gamma}{\gamma-1}\int \nabla(\rho^{\gamma-1})\cdot \rho u dx
=\frac{\gamma}{\gamma-1}\int \rho^{\gamma-1}\rho_t dx
=\frac{d}{dt}\frac{\gamma}{\gamma-1} \int \rho^\gamma dx.
\end{equation}
Integrating by part and applying the boundary condition \eqref{bc2}, we get
\begin{equation*}
\begin{aligned}
\int \nabla \times (\nabla \times u)u ~dx
&=\int_{\partial \Omega}n\times (\nabla \times u)\cdot u~ d\sigma
+\int |\nabla \times u|^2 dx \\
&=\int_{\partial \Omega}[Bu]_\tau \cdot  u_\tau~ d\sigma
+\int |\nabla \times u|^2 dx,
\end{aligned}
\end{equation*}
and
\begin{equation}
\int \nabla {\rm div}u\cdot u ~dx
=\int_{\partial \Omega}({\rm div}u)u\cdot n~d\sigma
  -\int |{\rm div}u|^2 dx
=-\int |{\rm div}u|^2 dx.
\end{equation}
which, together with \eqref{312}, gives directly
\begin{equation}\label{313}
\begin{aligned}
&\frac{d}{dt}\int(\frac{1}{2}  \rho |u|^2+\frac{\gamma}{\gamma-1}\rho^\gamma) dx
+\mu \varepsilon \int |\nabla \times u|^2 dx
+(2\mu+\lambda)\varepsilon\int |{\rm div}u|^2 dx\\
&=-\varepsilon\int_{\partial \Omega}[Bu]_\tau \cdot  u_\tau~ d\sigma
-\int (u \cdot \nabla) d \cdot \Delta d ~dx.
\end{aligned}
\end{equation}
Multiplying \eqref{eq1}$_3$ by $\Delta d$, one arrives at
\begin{equation}\label{314}
\int(d_t+u \cdot \nabla d)\cdot \Delta d ~dx
=\int |\Delta d|^2 dx+\int |\nabla d|^2 d \cdot \Delta d ~dx.
\end{equation}
Integration by part and application of boundary condition \eqref{bc2} yield directly
\begin{equation}\label{315}
\int d_t \cdot \Delta d ~ dx
=\int_{\partial \Omega}d_t \cdot \frac{\partial d}{\partial n}d\sigma
-\frac{1}{2}\frac{d}{dt}\int |\nabla d|^2 dx
=-\frac{1}{2}\frac{d}{dt}\int |\nabla d|^2 dx.
\end{equation}
By virtue of the basic fact $|d|=1$, we find $\Delta d \cdot d=-|\nabla d|^2$.
Then, the combination of \eqref{314} and \eqref{315} gives
\begin{equation*}
\frac{1}{2}\frac{d}{dt}\int |\nabla d|^2 dx+\int |\Delta d|^2 dx
=\int (u \cdot \nabla) d\cdot \Delta d ~dx+\int |\nabla d|^4 dx,
\end{equation*}
which, together with \eqref{313}, yields directly
\begin{equation}\label{316}
\begin{aligned}
&\frac{d}{dt}\int(\frac{1}{2}  \rho |u|^2+\frac{\gamma}{\gamma-1}\rho^\gamma
+\frac{1}{2} |\nabla d|^2) dx
+\int |\Delta d|^2 dx\\
&\quad +\mu \varepsilon \int |\nabla \times u|^2 dx
+(2\mu+\lambda)\varepsilon\int |{\rm div}u|^2 dx\\
&=-\varepsilon\int_{\partial \Omega}[Bu]_\tau \cdot  u_\tau~ d\sigma
+\int |\nabla d|^4dx.
\end{aligned}
\end{equation}
The trace theorem in Proposition \ref{prop2.3} implies
\begin{equation}\label{317}
|u|_{L^2(\partial \Omega)}^2\le \delta \|\nabla u\|_{L^2}^2+C_\delta \|u\|_{L^2}^2.
\end{equation}
The application of Proposition \ref{prop2.1} gives immediately
\begin{equation}\label{318}
\begin{aligned}
&\mu\|\nabla \times u\|_{L^2}^2
+(2\mu+\lambda)\|{\rm div}u\|_{L^2}^2\\
&\ge \min\{\mu, 2\mu+\lambda\}(\|\nabla \times u\|_{L^2}^2
+\|{\rm div}u\|_{L^2}^2)\\
&\ge 2c_1 \|\nabla u\|_{L^2}^2-C\|u\|_{L^2}^2.
\end{aligned}
\end{equation}
Substituting \eqref{317} and \eqref{318} into \eqref{316}
and choosing $\delta $ small enough, one arrives at
\begin{equation*}
\begin{aligned}
&\frac{d}{dt}\int (\frac{1}{2}\rho|u|^2+\frac{\gamma}{\gamma-1}\rho^\gamma
+\frac{1}{2}|\nabla d|^2) dx
+c_1 \varepsilon \int |\nabla  u|^2 dx
+\int |\Delta d|^2 dx\\
&\le \int |\nabla d|^4 dx
+C_2 \int  |u|^2 dx,
\end{aligned}
\end{equation*}
which, integrating over $[0, t]$, yields
\begin{equation*}
\begin{aligned}
&\int (\frac{1}{2}\rho|u|^2+\frac{\gamma}{\gamma-1}\rho^\gamma
+\frac{1}{2}|\nabla d|^2) dx
+c_1 \varepsilon \int_0^t \int |\nabla  u|^2 dx d\tau
+\int_0^t \int |\Delta d|^2 dx d\tau\\
&\le \int (\frac{1}{2}\rho_0|u_0|^2+\frac{\gamma}{\gamma-1}\rho_0^\gamma
+\frac{1}{2}|\nabla d_0|^2) dx
+\|\nabla d\|_{L^\infty}^2\int_0^t\int |\nabla d|^2 dxd\tau
+C_2 \int_0^t \int  |u|^2 dxd\tau,
\end{aligned}
\end{equation*}
Therefore, we complete the proof of Lemma \ref{lemma3.2}.
\end{proof}

However, the above basic energy estimation is insufficient to get the
vanishing viscosity limit. Some conormal derivative estimates are needed.
Let
\begin{equation}\label{def2}
Q(t)\triangleq \underset{0\le \tau \le t}{\sup}
\{\|(\nabla p, \nabla u)\|_{\mathcal{H}^{1,\infty}}^2
+\|(p, u, p_t, u_t)\|_{L^\infty}^2+\|d_t\|_{W^{1,\infty}}^2
+\|\nabla d\|_{W^{1,\infty}}^2+\|\nabla \Delta d\|_{L^\infty}^2\}.
\end{equation}
and
\begin{equation}\label{def3}
\begin{aligned}
\Lambda_m(t)
&\triangleq \|(p, u, \nabla d)(t)\|_{\mathcal{H}^m}^2
+\|(\nabla u, \Delta d)(t)\|_{\mathcal{H}^{m-1}}^2
+\|\nabla u(\tau)\|_{\mathcal{H}^{1,\infty}}^2\\
&\quad +\sum_{k=0}^{m-2}\|\nabla \partial_t^k p(t)\|_{m-1-k}^2
+\varepsilon \|\nabla \partial_t^{m-1}p(t)\|_{L^2}^2.
\end{aligned}
\end{equation}

\begin{lemm}\label{lemma3.3}
For $m\in \mathbb{N}^+$ and a smooth solution to \eqref{eq1} and \eqref{bc2}, it holds that
for $\varepsilon \in (0, 1]$,
\begin{equation}\label{331}
\begin{aligned}
&\underset{0\le \tau \le t}{\sup}\|(u, p, \nabla d)(\tau)\|_{\mathcal{H}^{m}}^2
+C\varepsilon \int_0^t \|\nabla u(\tau)\|_{\mathcal{H}^{m}}^2 d\tau
+\int_0^t \|\Delta d(\tau)\|_{\mathcal{H}^{m}}^2 d\tau\\
&\le C_{m+2}\left\{\|(u_0, p_0, \nabla d_0)\|_{\mathcal{H}^{m}}^2
     +\delta \int_0^t \|\nabla \Delta d(\tau)\|_{\mathcal{H}^{m-1}}^2d\tau
     +\delta \varepsilon^2 \int_0^t\|\nabla^2 u(\tau)\|_{\mathcal{H}^{m-1}}^2 d\tau\right.\\
&\quad \quad  \quad \quad \quad \quad
   \left. +\delta \int_0^t \|\nabla \partial_t^{m-1}p(\tau)\|_{L^2}^2 d\tau
+C_{\delta}[1+P(Q(t))]\int_0^t \Lambda_m(\tau)d\tau\right\},
\end{aligned}
\end{equation}
where $\delta$ is a small constant which will be chosen late,
$C_\delta$ is a polynomial function of $\frac{1}{\delta}$,
and the generic positive constant $C>0$ depends on $\mu$ and $\lambda$.
\end{lemm}
\begin{proof}
The case for $m=0$ is already proved in Lemma \ref{lemma3.2}. Assume that
\eqref{331} is proved for $k=m-1$. We shall prove that is holds for $k=m \ge 1$.
Applying the operator $\mathcal{Z}^\alpha(|\alpha_0|+|\alpha_1|=m)$
to the equation \eqref{eq1-1}, we find
\begin{equation}\label{332}
\begin{aligned}
&\rho \mathcal{Z}^\alpha u_t+\rho u\cdot \nabla \mathcal{Z}^\alpha u
+\mathcal{Z}^\alpha \nabla p\\
&=-\mu \varepsilon \mathcal{Z}^\alpha \nabla \times (\nabla \times u)
+(2\mu+\lambda)\varepsilon \mathcal{Z}^\alpha \nabla {\rm div}u
-\mathcal{Z}^\alpha(\nabla d \cdot \Delta d)
+\mathcal{C}_1^\alpha+\mathcal{C}_2^\alpha,
\end{aligned}
\end{equation}
where
\begin{equation*}
\mathcal{C}_1^\alpha=-[\mathcal{Z}^\alpha, \rho ]u_t, \quad
\mathcal{C}_2^\alpha=-[\mathcal{Z}^\alpha, \rho u\cdot \nabla]u.
\end{equation*}
Multiplying \eqref{332} by $\mathcal{Z}^\alpha u$, we obtain
\begin{equation}\label{333}
\begin{aligned}
&\frac{1}{2}\frac{d}{dt}\int \rho |\mathcal{Z}^\alpha u|^2 dx
+\int \mathcal{Z}^\alpha \nabla p \cdot \mathcal{Z}^\alpha u ~dx\\
&=-\mu \varepsilon \int \mathcal{Z}^\alpha \nabla \times(\nabla \times u)\cdot \mathcal{Z}^\alpha u ~dx
-(2\mu+\lambda) \varepsilon \int \mathcal{Z}^\alpha \nabla {\rm div} u\cdot \mathcal{Z}^\alpha u ~dx\\
&\quad -\int \mathcal{Z}^\alpha(\nabla d \cdot \Delta d)\cdot \mathcal{Z}^\alpha u~dx
+\int \mathcal{C}_1^\alpha \cdot \mathcal{Z}^\alpha u ~dx
+\int \mathcal{C}_2^\alpha \cdot \mathcal{Z}^\alpha u ~dx.
\end{aligned}
\end{equation}
Using the same argument as Lemma 3.4 of \cite{Wang-Xin-Yong},
one can obtain the following estimates
\begin{equation}\label{334}
\begin{aligned}
&-\varepsilon\int \mathcal{Z}^\alpha \nabla \times(\nabla \times u)
\cdot\mathcal{Z}^\alpha u ~dx\\
&\le -\frac{3\varepsilon}{4}\|\nabla \times \mathcal{Z}^\alpha u\|_{L^2}^2
+\delta \varepsilon^2 \|\nabla^2 u\|_{\mathcal{H}^{m-1}}^2
+C_\delta C_{m+2}(\|\nabla u\|_{\mathcal{H}^{m-1}}^2+\|u\|_{\mathcal{H}^m}^2)
\end{aligned}
\end{equation}
and
\begin{equation}\label{335}
\begin{aligned}
&\varepsilon\int \mathcal{Z}^\alpha \nabla{\rm div} u
\cdot\mathcal{Z}^\alpha u ~dx\\
&\le -\frac{3\varepsilon}{4}\|{\rm div} \mathcal{Z}^\alpha u\|_{L^2}^2
+\delta \varepsilon^2 \|\nabla^2 u\|_{\mathcal{H}^{m-1}}^2
+C_\delta C_{m+2}(\|\nabla u\|_{\mathcal{H}^{m-1}}^2+\|u\|_{\mathcal{H}^m}^2).
\end{aligned}
\end{equation}
On the other hand, it follows from Proposition \ref{prop2.1} that
\begin{equation}\label{336}
\begin{aligned}
2c_1\|\nabla \mathcal{Z}^\alpha u\|_{L^2}^2
&\le (\mu\|\nabla \times \mathcal{Z}^\alpha u\|_{L^2}^2
+(2\mu+\lambda)\|{\rm div}\mathcal{Z}^\alpha u\|_{L^2}^2
+\|\mathcal{Z}^\alpha u\|_{L^2}^2
+|\mathcal{Z}^\alpha u\cdot n|_{H^{\frac{1}{2}}(\partial\Omega)})\\
&\le (\mu\|\nabla \times \mathcal{Z}^\alpha u\|_{L^2}^2
+(2\mu+\lambda)\|{\rm div}\mathcal{Z}^\alpha u\|_{L^2}^2)
+C_{m+2}(\|u\|_{\mathcal{H}^m}^2+\|\nabla u\|_{\mathcal{H}^{m-1}}^2),
\end{aligned}
\end{equation}
where we have using the fact
$$|\mathcal{Z}^\alpha u\cdot n|_{H^{\frac{1}{2}}(\partial\Omega)}
\le C_{m+2}(\|u\|_{\mathcal{H}^m}^2+\|\nabla u\|_{\mathcal{H}^{m-1}}^2).$$
Substituting \eqref{334}-\eqref{336} into \eqref{333} and integrating
the resulting inequality over $[0, t]$, we find
\begin{equation}\label{337}
\begin{aligned}
&\frac{1}{2}\int \rho|\mathcal{Z}^\alpha u(t)|^2 dx
+\frac{3c_1\varepsilon}{2}\int_0^t \int |\nabla \mathcal{Z}^\alpha u|^2 dx d\tau
+\int_0^t \int \mathcal{Z}^\alpha \nabla p\cdot \mathcal{Z}^\alpha u~dxd\tau\\
&\le \frac{1}{2}\int \rho_0|\mathcal{Z}^\alpha u_0|^2 dx
     +C\delta_1 \varepsilon \int_0^t\|\nabla u\|_{\mathcal{H}^{m}}^2 d\tau
      +C\delta \varepsilon^2 \int_0^t\|\nabla^2 u\|_{\mathcal{H}^{m-1}}^2d\tau\\
&\quad \quad +C_\delta C_{m+2}\int_0^t
(\|\nabla u\|_{\mathcal{H}^{m-1}}^2+\|u\|_{\mathcal{H}^{m}}^2)d\tau
+\int_0^t \int \mathcal{C}_1^\alpha \cdot \mathcal{Z}^\alpha u ~dxd\tau\\
&\quad \quad
+\int_0^t \int \mathcal{C}_2^\alpha \cdot \mathcal{Z}^\alpha u ~dxd\tau
-\int_0^t\int \mathcal{Z}^\alpha(\nabla d \cdot \Delta d)\cdot \mathcal{Z}^\alpha u~dxd\tau.
\end{aligned}
\end{equation}
Applying the transport equation \eqref{eq1}$_1$, we follow the same argument
as Lemma 3.4 of \cite{Wang-Xin-Yong} to obtain
\begin{equation}\label{338}
\begin{aligned}
-\int \mathcal{Z}^\alpha \nabla p \cdot \mathcal{Z}^\alpha u ~dx
&\le -\int \frac{1}{2\gamma p}|\mathcal{Z}^\alpha p|^2 dx
     +\int \frac{1}{2\gamma p_0}|\mathcal{Z}^\alpha p_0|^2 dx
     +C\delta \int_0^t \|\nabla p\|_{\mathcal{H}^{m-1}}^2 dx\\
&\quad +C_\delta[1+P(Q(t))]
\int (\|(p,u)\|_{\mathcal{H}^{m}}^2+\|\nabla u\|_{\mathcal{H}^{m-1}}^2)d\tau.
\end{aligned}
\end{equation}
In view of the Proposition \ref{prop2.2}, we obtain
\begin{equation*}
\begin{aligned}
\int_0^t \|\mathcal{Z}^\alpha(\nabla d\cdot \Delta d)\|_{L^2}^2 d\tau
\le C\|\nabla d\|_{L^\infty_{x,t}}^2\int_0^t \|\Delta d\|_{\mathcal{H}^{m}}^2d\tau
    +C\|\Delta d\|_{L^\infty_{x,t}}^2\int_0^t \|\nabla d\|_{\mathcal{H}^{m}}^2d\tau
\end{aligned}
\end{equation*}
which, by using the Cauchy inequality, yields directly
\begin{equation}\label{339}
\begin{aligned}
&\left|-\int_0^t \int \mathcal{Z}^\alpha(\nabla d \cdot \Delta d)\cdot \mathcal{Z}^\alpha u~dxd\tau\right|\\
&\le \delta_1 \int_0^t \|\Delta d\|_{\mathcal{H}^{m}}^2d\tau
+C_{\delta_1}(\|\nabla d\|_{L^\infty_{x,t}}^2+\|\Delta d\|_{L^\infty_{x,t}}^2)
\int_0^t(\|u\|_{\mathcal{H}^{m}}^2+\|\nabla d\|_{\mathcal{H}^{m}}^2)d\tau.
\end{aligned}
\end{equation}
Similarly, it is easy to deduce that(or see Lemma 3.4 of \cite{Wang-Xin-Yong})
\begin{equation}\label{3310}
\int_0^t \int \mathcal{C}_1^\alpha \cdot \mathcal{Z}^\alpha u ~dxd\tau
\le C[1+P(Q(t))]\int_0^t \|(p, u)\|_{\mathcal{H}^m}^2d\tau
\end{equation}
and
\begin{equation}\label{3310-1}
\int_0^t \int \mathcal{C}_2^\alpha \cdot \mathcal{Z}^\alpha u ~dxd\tau
\le C[1+P(Q(t))]\int_0^t (\|(p, u)\|_{\mathcal{H}^m}^2
+\|\nabla u\|_{\mathcal{H}^{m-1}}^2)d\tau.
\end{equation}
Substituting \eqref{338}-\eqref{3310-1} into \eqref{337}, one attains
\begin{equation}\label{3311}
\begin{aligned}
&\frac{1}{2}\int \rho|\mathcal{Z}^\alpha u|^2 dx
+\int \frac{1}{2\gamma p}|\mathcal{Z}^\alpha p|^2 dx
+\frac{3c_1\varepsilon}{2}\int_0^t \int |\nabla \mathcal{Z}^\alpha u|^2 dx d\tau\\
&\le \frac{1}{2}\int \rho_0|\mathcal{Z}^\alpha u_0|^2 dx
      +\int \frac{1}{2\gamma p_0}|\mathcal{Z}^\alpha p_0|^2 dx
     +C\delta_1 \varepsilon \int_0^t\|\nabla u\|_{\mathcal{H}^{m}}^2 d\tau
      +C\delta \varepsilon^2 \int_0^t\|\nabla^2 u\|_{\mathcal{H}^{m-1}}^2d\tau\\
&\quad \quad +C_\delta C_{m+2}[1+P(Q(t))]
\int_0^t(\|\nabla u\|_{\mathcal{H}^{m-1}}^2
+\|(p,u,\nabla d)\|_{\mathcal{H}^{m}}^2)d\tau.
\end{aligned}
\end{equation}
Applying the operator $\mathcal{Z}^\alpha \nabla(|\alpha_0|+|\alpha_1|=m$)
to the equation \eqref{eq1}$_3$, we find
\begin{equation}\label{3312}
\mathcal{Z}^\alpha \nabla d_t -\mathcal{Z}^\alpha \nabla \Delta d
=-\mathcal{Z}^\alpha \nabla(u\cdot \nabla d)
 +\mathcal{Z}^\alpha \nabla(|\nabla d|^2 d).
\end{equation}
Multiplying \eqref{3312} by $\mathcal{Z}^\alpha \nabla d$,
it is easy to deduce that
\begin{equation}\label{3313}
\begin{aligned}
&\frac{1}{2}\frac{d}{dt}\int |\mathcal{Z}^\alpha \nabla d|^2 dx
-\int \mathcal{Z}^\alpha \nabla \Delta d \cdot \mathcal{Z}^\alpha \nabla d~dx\\
&=-\int \mathcal{Z}^\alpha \nabla(u\cdot \nabla d)\cdot \mathcal{Z}^\alpha \nabla d~ dx
+\int \mathcal{Z}^\alpha \nabla(|\nabla d|^2 d)\cdot \mathcal{Z}^\alpha \nabla d ~dx.
\end{aligned}
\end{equation}
Integrating by part, it is easy to check that
\begin{equation*}
\begin{aligned}
&-\int \mathcal{Z}^\alpha \nabla \Delta d \cdot \mathcal{Z}^\alpha \nabla d~dx\\
&=-\int \nabla \mathcal{Z}^\alpha  \Delta d \cdot \mathcal{Z}^\alpha \nabla d~dx
  -\int [\mathcal{Z}^\alpha, \nabla] \Delta d \cdot \mathcal{Z}^\alpha \nabla d~dx\\
&=-\int_{\partial \Omega} \mathcal{Z}^\alpha  \Delta d \cdot \mathcal{Z}^\alpha \nabla d
    \cdot n~d\sigma
  +\int \mathcal{Z}^\alpha  \Delta d \cdot {\rm div}(\mathcal{Z}^\alpha \nabla d)~dx
  -\int [\mathcal{Z}^\alpha, \nabla] \Delta d \cdot \mathcal{Z}^\alpha \nabla d~dx\\
&=-\int_{\partial \Omega} \mathcal{Z}^\alpha  \Delta d \cdot \mathcal{Z}^\alpha \nabla d
    \cdot n~d\sigma
  +\int |{\rm div}(\mathcal{Z}^\alpha \nabla d)|^2dx
+\int [\mathcal{Z}^\alpha, {\rm div}]\nabla d \cdot {\rm div}(\mathcal{Z}^\alpha \nabla d)~dx\\
&\quad -\int [\mathcal{Z}^\alpha, \nabla] \Delta d \cdot \mathcal{Z}^\alpha \nabla d~dx.
\end{aligned}
\end{equation*}
This, together with \eqref{3313}, reads
\begin{equation}\label{3314}
\begin{aligned}
&\frac{1}{2}\int |\mathcal{Z}^\alpha \nabla d(t)|^2 dx
+\int_0^t \int |{\rm div}(\mathcal{Z}^\alpha \nabla d)|^2dx\\
&=\frac{1}{2}\int |\mathcal{Z}^\alpha \nabla d_0|^2 dx
-\int_0^t \!\!\int \mathcal{Z}^\alpha \nabla(u\cdot \nabla d)\cdot \mathcal{Z}^\alpha \nabla d~ dxd\tau\\
&\quad
+\int_0^t\!\!\int \mathcal{Z}^\alpha \nabla(|\nabla d|^2 d)\cdot \mathcal{Z}^\alpha \nabla d ~dxd\tau
-\int_0^t\int [\mathcal{Z}^\alpha, {\rm div}]\nabla d \cdot {\rm div}(\mathcal{Z}^\alpha \nabla d)~dxd\tau\\
&\quad
+\int_0^t\int [\mathcal{Z}^\alpha, \nabla] \Delta d \cdot \mathcal{Z}^\alpha \nabla d~dxd\tau
+\int_0^t\int_{\partial \Omega} \mathcal{Z}^\alpha  \Delta d \cdot \mathcal{Z}^\alpha \nabla d\cdot n~d\sigma d\tau\\
&\triangleq I_1+I_2+I_3+I_4+I_5+I_6.
\end{aligned}
\end{equation}
{\bf {Deal with the term $I_2$}}.
Integrating by part, one arrives at
\begin{equation}\label{3315}
\begin{aligned}
I_2
&=-\int_0^t\int  \nabla \mathcal{Z}^\alpha(u\cdot \nabla d)\cdot \mathcal{Z}^\alpha \nabla d~ dxd\tau
  -\int_0^t\int  [\mathcal{Z}^\alpha, \nabla](u\cdot \nabla d)\cdot \mathcal{Z}^\alpha \nabla d~ dxd\tau\\
&=-\int_0^t\int_{\partial \Omega } \mathcal{Z}^\alpha(u\cdot \nabla d)\cdot
    \mathcal{Z}^\alpha \nabla d\cdot n~ d\sigma d\tau
-\int_0^t\int \mathcal{Z}^\alpha(u\cdot \nabla d)\cdot{\rm div} (\mathcal{Z}^\alpha \nabla d)~ dxd\tau\\
&\quad -\int_0^t\int[\mathcal{Z}^\alpha, \nabla](u\cdot \nabla d)\cdot \mathcal{Z}^\alpha \nabla d~ dxd\tau.
\end{aligned}
\end{equation}
To estimate the boundary term on the right hand side of \eqref{3315}.
If $|\alpha_0|=m$, we apply the boundary condition \eqref{bc2} to deduce that
$$
-\int_0^t\int_{\partial \Omega } \mathcal{Z}^\alpha(u\cdot \nabla d)\cdot
    \mathcal{Z}^\alpha \nabla d\cdot n~ d\sigma d\tau=0.
$$
If $|\alpha_{13}|\neq 0$, the proposition of \eqref{pc} implies
$\mathcal{Z}^\alpha \nabla d=0$ on the boundary. Then, one arrives at
$$
-\int_0^t\int_{\partial \Omega } \mathcal{Z}^\alpha(u\cdot \nabla d)\cdot
    \mathcal{Z}^\alpha \nabla d\cdot n~ d\sigma d\tau=0.
$$
Hence, we deal with the case of $|\alpha_{13}|=0$ and $|\alpha_0|\le m-1$.
For $|\beta|=m-1-\alpha_0(|\alpha_0|\le m-1)$,
we integrating by part along the boundary to deduce that
\begin{equation}\label{3316}
-\int_0^t\int_{\partial \Omega } \mathcal{Z}^\alpha(u\cdot \nabla d)\cdot
    \mathcal{Z}^\alpha \nabla d\cdot n~ d\sigma d\tau
\le \int_0^t |\partial_t^{\alpha_0}Z_y^\beta (u\cdot \nabla d)|_{L^2{(\partial \Omega)}}
      |\mathcal{Z}^\alpha \nabla d\cdot n|_{H^1(\partial \Omega)}d\tau.
\end{equation}
Applying the trace theorem in Proposition \ref{prop2.3}
and the Proposition \ref{prop2.2}, one arrives at
\begin{equation}\label{3317}
\begin{aligned}
&\int_0^t |\partial_t^{\alpha_0}Z_y^\beta (u\cdot \nabla d)|_{L^2{(\partial \Omega)}}^2d\tau\\
&\le C\int_0^t (\|\nabla \partial_t^{\alpha_0} (u\cdot \nabla d)\|_{m-1-\alpha_0}^2
              +\| \partial_t^{\alpha_0} (u\cdot \nabla d)\|_{m-1-\alpha_0}^2)d\tau\\
&\le C\int_0^t(\|\nabla(u\cdot \nabla d)\|_{\mathcal{H}^{m-1}}^2
              +\|u\cdot \nabla d\|_{\mathcal{H}^{m-1}}^2)d\tau\\
&\le CQ(t)\int_0^t(\|(u,\nabla d)\|_{\mathcal{H}^{m-1}}^2
     +\|\nabla(u,\nabla d)\|_{\mathcal{H}^{m-1}}^2)d\tau.
\end{aligned}
\end{equation}
With the help of boundary condition \eqref{bc2}
and trace theorem in Proposition \ref{prop2.3}, we find
\begin{equation}\label{3318}
\int_0^t |\mathcal{Z}^\alpha \nabla d\cdot n|_{H^1(\partial \Omega)}^2d\tau
\le C_{m+2}\int_0^t (\|\nabla^2 d\|_{\mathcal{H}^{m}}^2
+\|\nabla d\|_{\mathcal{H}^{m}}^2)d\tau.
\end{equation}
The combination of \eqref{3316}-\eqref{3318} and Cauchy inequality, it is easy to deduce that
\begin{equation}\label{3319}
\begin{aligned}
&-\int_0^t\int_{\partial \Omega } \mathcal{Z}^\alpha(u\cdot \nabla d)\cdot
    \mathcal{Z}^\alpha \nabla d\cdot n~ d\sigma d\tau\\
&\le \delta_1  \!\int_0^t \! \|\nabla^2 d\|_{\mathcal{H}^{m}}^2 d\tau
     +C_{\delta_1}C_{m+2}(1+Q(t))
     \!\ \! \int_0^t \! \!(\|(u,\nabla d)\|_{\mathcal{H}^{m}}^2+\|\nabla(u,\nabla d)\|_{\mathcal{H}^{m-1}}^2)d\tau.
\end{aligned}
\end{equation}
Applying the Young inequality and the Proposition \ref{prop2.2},
one arrives at
\begin{equation}\label{3320}
\begin{aligned}
&-\int_0^t\int \mathcal{Z}^\alpha(u\cdot \nabla d)\cdot{\rm div} (\mathcal{Z}^\alpha \nabla d)~ dxd\tau\\
&\le \delta_1\int_0^t\|{\rm div} (\mathcal{Z}^\alpha \nabla d)\|_{L^2}^2d\tau
     +C_{\delta_1} \|u\|_{L^\infty_{x,t}}^2\int_0^t \|\nabla d\|_{\mathcal{H}^{m}}^2 d\tau
+C_{\delta_1} \|\nabla d\|_{L^\infty_{x,t}}^2\int_0^t \|u\|_{\mathcal{H}^{m}}^2 d\tau\\
&\le \delta_1\int_0^t\|{\rm div} (\mathcal{Z}^\alpha \nabla d)\|_{L^2}^2d\tau
     +C_{\delta_1} C_1 Q(t)
     \int_0^t (\|u\|_{\mathcal{H}^{m}}^2
      +\|\nabla d\|_{\mathcal{H}^{m}}^2) d\tau\\
\end{aligned}
\end{equation}
and
\begin{equation}\label{3321}
\begin{aligned}
&-\int_0^t\int[\mathcal{Z}^\alpha, \nabla](u\cdot \nabla d)\cdot \mathcal{Z}^\alpha \nabla d~ dxd\tau\\
&\le  \sum_{|\beta|\le m-1}\int_0^t \|\mathcal{Z}^\beta(\nabla u\cdot \nabla d+
        u\cdot \nabla^2 d)\|_{L^2} \|\mathcal{Z}^\alpha \nabla d\|_{L^2}dx\\
&\le C\int_0^t \|\nabla d\|_{\mathcal{H}^{m}}^2 d\tau
      +C(\|\nabla u\|_{L^\infty_{x,t}}^2+\|\nabla d\|_{L^\infty_{x,t}}^2)
      \int_0^t (\|\nabla u\|_{\mathcal{H}^{m-1}}^2+\|\nabla d\|_{\mathcal{H}^{m-1}}^2)d\tau\\
&\quad  +C(\|u\|_{L^\infty_{x,t}}^2+\|\nabla^2 d\|_{L^\infty_{x,t}}^2)
      \int_0^t (\|u\|_{\mathcal{H}^{m-1}}^2+\|\nabla^2 d\|_{\mathcal{H}^{m-1}}^2)d\tau\\
&\le C(1+Q(t))
      \int_0^t (\|u\|_{\mathcal{H}^{m-1}}^2+\|\nabla u\|_{\mathcal{H}^{m-1}}^2
       +\|\nabla d\|_{\mathcal{H}^{m}}^2
      +\|\nabla^2 d\|_{\mathcal{H}^{m-1}}^2)d\tau.
\end{aligned}
\end{equation}
Substituting \eqref{3319}-\eqref{3321} into \eqref{3315}, we obtain
\begin{equation}\label{3322}
|I_2|
\le \delta_1 \int_0^t \|\nabla^2 d\|_{\mathcal{H}^{m}}^2 d\tau
     +C_{\delta_1}C_1(1+Q(t))
     \int_0^t(\|(u,\nabla d)\|_{\mathcal{H}^{m}}^2+\|\nabla(u,\nabla d)\|_{\mathcal{H}^{m-1}}^2)d\tau.
\end{equation}
{\bf {Deal with the term $I_3$}}.
Indeed, by integrating by part, one arrives at
\begin{equation}\label{3323}
\begin{aligned}
I_3
&=\int_0^t\int \nabla \mathcal{Z}^\alpha (|\nabla d|^2 d)\cdot \mathcal{Z}^\alpha \nabla d ~dxd\tau
+\int_0^t\int [\mathcal{Z}^\alpha, \nabla ](|\nabla d|^2 d)\cdot \mathcal{Z}^\alpha \nabla d ~dxd\tau\\
&=
-\int_0^t\int \mathcal{Z}^\alpha (|\nabla d|^2 d)\cdot {\rm div}(\mathcal{Z}^\alpha \nabla d) ~dxd\tau
+\int_0^t\int [\mathcal{Z}^\alpha, \nabla ](|\nabla d|^2 d)\cdot \mathcal{Z}^\alpha \nabla d ~dxd\tau\\
&\quad
+\int_0^t\int_{\partial \Omega} \mathcal{Z}^\alpha (|\nabla d|^2 d)
  \cdot \mathcal{Z}^\alpha \nabla d \cdot n~d\sigma d\tau.
\end{aligned}
\end{equation}
It is easy to deduce that
\begin{equation}\label{3324}
\begin{aligned}
&-\int_0^t\int \mathcal{Z}^\alpha (|\nabla d|^2 d)\cdot {\rm div}(\mathcal{Z}^\alpha \nabla d) ~dxd\tau\\
&=-\underset{|\beta| \ge 1}{\sum}\int_0^t\int\mathcal{Z}^\gamma(|\nabla d|^2)\mathcal{Z}^\beta d\cdot {\rm div}(\mathcal{Z}^\alpha \nabla d) ~dxd\tau\\
&\quad
-\int_0^t\int \mathcal{Z}^\alpha(|\nabla d|^2)d \cdot {\rm div}(\mathcal{Z}^\alpha \nabla d) ~dxd\tau.
\end{aligned}
\end{equation}
By virtue of the Proposition \ref{prop2.2}, we obtain
\begin{equation}\label{3325}
\begin{aligned}
&\underset{|\beta| \ge 1}{\sum}\int_0^t
\|\mathcal{Z}^\gamma(|\nabla d|^2)\mathcal{Z}^\beta d\|_{L^2}^2d\tau\\
&\le  \|\mathcal{Z}d\|_{L^\infty_{x,t}}^2\int_0^t \||\nabla d|^2\|_{\mathcal{H}^{m-1}}^2 d\tau
      +\||\nabla d|^2\|_{L^\infty_{x,t}}^2
      \int_0^t \| \mathcal{Z} d\|_{\mathcal{H}^{m-1}}^2 d\tau\\
&\le \|\mathcal{Z}d\|_{L^\infty}^2\|\nabla d\|_{L^\infty}^2
     \int_0^t \|\nabla d\|_{\mathcal{H}^{m-1}}^2 d\tau
      +\|\nabla d\|_{L^\infty}^4 \int_0^t \|\mathcal{Z}d\|_{\mathcal{H}^{m-1}}^2d\tau\\
&\le \|\nabla d\|_{L^\infty}^4 \int_0^t \|\partial_td\|_{\mathcal{H}^{m-1}}^2d\tau
     +C_1(1+P(Q(t)))\int_0^t \Lambda_m(\tau)d\tau.
\end{aligned}
\end{equation}
By virtue of equation \eqref{eq1}$_2$, we find
\begin{equation}\label{3326}
\begin{aligned}
\int_0^t \|\partial_td\|_{\mathcal{H}^{m-1}}^2d\tau
&\le \|u\|_{L^\infty}^2\int_0^t \|\nabla d\|_{\mathcal{H}^{m-1}}^2d\tau
     +\|\nabla d\|_{L^\infty}^2\int_0^t \|u\|_{\mathcal{H}^{m-1}}^2d\tau\\
&\quad +\int_0^t \|\Delta d\|_{\mathcal{H}^{m-1}}^2d\tau
       +\int_0^t \||\nabla d|^2 d\|_{\mathcal{H}^{m-1}}^2 d\tau.
\end{aligned}
\end{equation}
By virtue of the Proposition \ref{prop2.2}, we obtain
\begin{equation}\label{3327}
\begin{aligned}
\int_0^t \||\nabla d|^2 d\|_{\mathcal{H}^{m-1}}^2 d\tau
&\le \sum_{|\gamma|\ge 1,|\beta|+|\gamma|\le m-1}
     \int_0^t \|\mathcal{Z}^\beta(|\nabla d|^2)\mathcal{Z}^\gamma d\|_{L^2}^2 d\tau
     +\int_0^t \||\nabla d|^2\|_{\mathcal{H}^{m-1}}^2d\tau\\
&\lesssim \|\mathcal{Z}d\|_{L^\infty}^2 \|\nabla d\|_{L^\infty}^2
     \int_0^t \|\nabla d\|_{\mathcal{H}^{m-2}}^2d\tau
     +\|\nabla d\|_{L^\infty}^4\int_0^t \|\nabla d\|_{\mathcal{H}^{m-2}}^2 d\tau\\
&\quad +\|\nabla d\|_{L^\infty}^4\int_0^t \|\partial_t d\|_{\mathcal{H}^{m-2}}^2 d\tau
       +\|\nabla d\|_{L^\infty}^2\int_0^t \|\nabla d\|_{\mathcal{H}^{m-1}}^2 d\tau.
\end{aligned}
\end{equation}
Substituting \eqref{3327} into \eqref{3326}, one arrives at immediately
\begin{equation}\label{3328}
\begin{aligned}
\int_0^t \|d_t\|_{\mathcal{H}^{m-1}}^2d\tau
&\lesssim \|\nabla d\|_{L^\infty}^4\int_0^t \|\partial_t d\|_{\mathcal{H}^{m-2}}^2 d\tau
    +\int_0^t \|\Delta d\|_{\mathcal{H}^{m-1}}^2d\tau\\
&\quad    +C_1(1+P(Q(t)))\int_0^t (\|u\|_{\mathcal{H}^{m-1}}^2
    +\|\nabla d\|_{\mathcal{H}^{m-1}}^2)d\tau.
\end{aligned}
\end{equation}
On the other hand, it is easy to deduce that
\begin{equation}\label{3329}
\int_0^t \|d_t\|_{L^2}^2 d\tau
\lesssim \int_0^t \|\Delta d\|_{L^2}^2 d\tau+(1+\|\nabla d\|_{L^\infty}^2)
    \int_0^t (\|u\|_{L^2}^2+\|\nabla d\|_{L^2}^2)d\tau.
\end{equation}
The combination  of  \eqref{3328} and \eqref{3329} yields directly
\begin{equation}\label{3330}
\begin{aligned}
\int_0^t  \|d_t\|_{\mathcal{H}^{m-1}}^2 d\tau
&\le C_1(1+P(Q(t)))\int_0^t\!\! \|\Delta d\|_{\mathcal{H}^{m-1}}^2 d\tau\\
&\quad +C_1(1+P(Q(t)))\int_0^t \!\!(\|u\|_{\mathcal{H}^{m-1}}^2
    +\|\nabla d\|_{\mathcal{H}^{m-1}}^2)d\tau,
\end{aligned}
\end{equation}
which, together with \eqref{3325}, gives directly
\begin{equation}\label{3331}
\begin{aligned}
&\left|-\underset{|\beta| \ge 1}{\sum}\int_0^t\int\mathcal{Z}^\gamma(|\nabla d|^2)\mathcal{Z}^\beta d\cdot {\rm div}(\mathcal{Z}^\alpha \nabla d) ~dxd\tau\right|\\
&\le \delta_1 \int_0^t\|{\rm div}(\mathcal{Z}^\alpha \nabla d)\|_{L^2}d\tau
+C_{\delta_1}C_1(1+P(Q(t)))\int_0^t \Lambda_m(\tau)d\tau.
\end{aligned}
\end{equation}
In view of the Proposition \ref{prop2.2} and Cauchy inequality, we obtain
\begin{equation}\label{3332}
\begin{aligned}
&\left|-\int_0^t\int \mathcal{Z}^\alpha(|\nabla d|^2)d \cdot {\rm div}(\mathcal{Z}^\alpha \nabla d) ~dxd\tau\right|\\
&\le \delta_1 \int_0^t\|{\rm div}(\mathcal{Z}^\alpha \nabla d)\|_{L^2}d\tau
      +C_{\delta_1}\|\nabla d\|_{L^\infty}^2\int_0^t \|\nabla d\|_{\mathcal{H}^{m}}^2d\tau.
\end{aligned}
\end{equation}
Then combination of  \eqref{3331} and \eqref{3332} yields immediately
\begin{equation}\label{3333}
\begin{aligned}
&-\int_0^t\int \mathcal{Z}^\alpha (|\nabla d|^2 d)\cdot {\rm div}(\mathcal{Z}^\alpha \nabla d) ~dxd\tau\\
&\le \delta_1 \int_0^t\|{\rm div}(\mathcal{Z}^\alpha \nabla d)\|_{L^2}d\tau
+C_1C_{\delta_1}(1+P(Q(t)))\int_0^t \Lambda_m(\tau)d\tau.
\end{aligned}
\end{equation}
On the other hand, we find that
\begin{equation}\label{3334}
\begin{aligned}
&\int_0^t \int [\mathcal{Z}^\alpha, \nabla ](|\nabla d|^2 d)\cdot \mathcal{Z}^\alpha \nabla d ~dxd\tau\\
&=\sum_{|\beta|\le m-1}\int_0^t \int
   d \cdot \mathcal{Z}^{\beta} (\nabla d\cdot \nabla^2 d)
   \cdot \mathcal{Z}^\alpha \nabla d ~dxd\tau\\
&\quad +\sum_{|\beta|+|\gamma|\le m-1}^{|\beta|\ge 1|}\int_0^t\int
   \mathcal{Z}^\beta d \cdot \mathcal{Z}^\gamma (\nabla d\cdot \nabla^2 d)
   \cdot \mathcal{Z}^\alpha \nabla d ~dxd\tau\\
&\quad  +\underset{|\beta|+|\gamma|\le m-1}{\sum} \int_0^t\int
   \mathcal{Z}^\beta (|\nabla d|^2)\mathcal{Z}^\gamma \nabla d
   \cdot \mathcal{Z}^\alpha \nabla d ~dxd\tau\\
&=I\!I_1+I\!I_2+I\!I_3.
\end{aligned}
\end{equation}
In view of the Proposition \ref{prop2.2}, we find
\begin{equation}\label{3335}
I\!I_1\lesssim
\|\nabla d\|_{L^\infty_{x,t}}^2\int_0^t \|\nabla^2 d\|_{\mathcal{H}^{m-1}}^2d\tau
+\|\nabla^2 d\|_{L^\infty_{x,t}}^2\int_0^t \|\nabla d\|_{\mathcal{H}^{m-1}}^2d\tau
+\int_0^t \|\nabla d\|_{\mathcal{H}^{m}}^2 d\tau.
\end{equation}
and
\begin{equation}\label{3336}
\begin{aligned}
I\!I_2
&\lesssim  \|\nabla d\|_{W^{1,\infty}_{x,t}}^4
       \int_0^t \|\mathcal{Z} d\|_{\mathcal{H}^{m-2}}^2d\tau
      +\|\mathcal{Z} d\|_{L^\infty_{x,t}}^2\|\nabla d\|_{L^\infty}^2
        \int_0^t \|\nabla^2 d\|_{\mathcal{H}^{m-2}}^2d\tau\\
&\quad +\|\mathcal{Z} d\|_{L^\infty_{x,t}}^2\|\nabla^2 d\|_{L^\infty}^2
        \int_0^t \|\nabla d\|_{\mathcal{H}^{m-2}}^2d\tau
       +\int_0^t \|\nabla d\|_{\mathcal{H}^{m}}^2 d\tau\\
&\le C_1(1+P(Q(t)))\int_0^t \Lambda_m(\tau)d\tau,
\end{aligned}
\end{equation}
where we have used the estimate \eqref{3330}.
Similarly, it is easy to deduce that
\begin{equation}\label{3337}
I\!I_3
\le C \|\nabla d\|_{L^\infty}^4 \int_0^t \|\nabla d\|_{\mathcal{H}^{m-1}}^2d\tau
+C\int_0^t \|\nabla d\|_{\mathcal{H}^{m}}^2 d\tau.
\end{equation}
Substituting \eqref{3335}-\eqref{3337} into \eqref{3334}, one arrives at
\begin{equation}\label{3338}
\int_0^t\int [\mathcal{Z}^\alpha, \nabla ](|\nabla d|^2 d)\cdot \mathcal{Z}^\alpha \nabla d ~dxd\tau
\le C_1(1+P(Q(t)))\int_0^t \Lambda_m(\tau)d\tau.
\end{equation}
Deal with the boundary term on the right hand side of \eqref{3323}.
If $|\alpha_0|=m$ or $|\alpha_{13}|\ge 1$, we obtain
\begin{equation}\label{3339}
\int_0^t\int_{\partial \Omega} \mathcal{Z}^\alpha (|\nabla d|^2 d)
\cdot \mathcal{Z}^\alpha \nabla d \cdot n~d\sigma d\tau
=0.
\end{equation}
On the other hand, it is easy to deduce that for $|\beta|=m-1-\alpha_0$
\begin{equation}\label{3340}
\int_0^t\int_{\partial \Omega} \mathcal{Z}^\alpha (|\nabla d|^2 d)
  \cdot \mathcal{Z}^\alpha \nabla d \cdot n~d\sigma d\tau
\le \int_0^t |\partial_t^{\alpha_0}Z_y^{\beta} (|\nabla d|^2 d)|_{L^2(\partial \Omega)}
    |\mathcal{Z}^\alpha \nabla d \cdot n|_{H^1(\partial \Omega)}d\tau.
\end{equation}
By virtue of the trace theorem in Proposition \ref{prop2.3},
we find for $|\beta|=m-1-\alpha_0$
\begin{equation}\label{3341}
\begin{aligned}
&|\partial_t^{\alpha_0}Z_y^{\beta} (|\nabla d|^2 d)|_{L^2(\partial \Omega)}^2\\
&\le C\|\nabla \partial_t^{\alpha_0}(|\nabla d|^2 d)\|_{m-1-\alpha_0}
       \|\partial_t^{\alpha_0}(|\nabla d|^2 d)\|_{m-1-\alpha_0}
     +C\|\partial_t^{\alpha_0}(|\nabla d|^2 d)\|_{m-1-\alpha_0}^2\\
&\le C\|\nabla(|\nabla d|^2 d)\|_{\mathcal{H}^{m-1}}
       \||\nabla d|^2 d\|_{\mathcal{H}^{m-1}}
    +C\||\nabla d|^2 d\|_{\mathcal{H}^{m-1}}^2,
\end{aligned}
\end{equation}
and
\begin{equation}\label{3342}
|\mathcal{Z}^\alpha \nabla d \cdot n|_{H^1(\partial \Omega)}^2
\le C_{m+2}(\|\nabla^2 d\|_{\mathcal{H}^{m}}^2+\|\nabla d\|_{\mathcal{H}^{m}}^2).
\end{equation}
On the other hand,  we obtain just following the idea as \eqref{3324}
and \eqref{3334} that
\begin{equation}\label{3343}
\int_0^t\||\nabla d|^2 d\|_{\mathcal{H}^{m-1}}^2 d\tau
\le C(1+P(Q(t)))\int_0^t \Lambda_m(\tau)d\tau
\end{equation}
and
\begin{equation}\label{3344}
\int_0^t\|\nabla(|\nabla d|^2 d)\|_{\mathcal{H}^{m-1}}^2 d\tau
\le C(1+P(Q(t)))\int_0^t \Lambda_m(\tau)d\tau.
\end{equation}
The combination of \eqref{3340}-\eqref{3344} gives directly
\begin{equation}\label{3345}
\begin{aligned}
&\int_0^t\int_{\partial \Omega} \mathcal{Z}^\alpha (|\nabla d|^2 d)
  \cdot \mathcal{Z}^\alpha \nabla d \cdot n~d\sigma d\tau\\
&\le \delta_1 \int_0^t\|\nabla^2 d\|_{\mathcal{H}^{m}}^2d\tau
     +C_{\delta_1}C_{m+2}(1+P(Q(t)))\int_0^t \Lambda_m(\tau)d\tau.
\end{aligned}
\end{equation}
Substituting \eqref{3333}, \eqref{3338} and \eqref{3345} into \eqref{3323}, we attains
\begin{equation}\label{3346}
|I_3|
\le \delta_1 \!\int_0^t \!\|{\rm div}(\mathcal{Z}^\alpha \nabla d)\|_{L^2}^2 d\tau
  +\delta_1 \!\int_0^t \!\|\nabla^2 d\|_{\mathcal{H}^{m}}^2d\tau
     +C_{\delta_1}C_{m+2}(1+P(Q(t)))\int_0^t \Lambda_m(\tau)d\tau.
\end{equation}
{\bf {Deal with the term $I_4$ and $I_5$}}.
In view of the Cauchy inequality, it is easy to deduce that
\begin{equation}\label{3347}
|I_4|
\le \delta_1 \int_0^t\|{\rm div}(\mathcal{Z}^\alpha \nabla d)\|_{L^2}^2d\tau
     +C_{\delta_1} \int_0^t\|\nabla^2 d\|^2_{\mathcal{H}^{m-1}}d\tau,
\end{equation}
and
\begin{equation}\label{3348}
|I_5|
\le \delta \int_0^t\|\nabla \Delta d\|_{\mathcal{H}^{m-1}}^2d\tau
    +C_\delta \int_0^t\|\nabla d\|_{\mathcal{H}^{m}}^2 d\tau.
\end{equation}
{\bf {Deal with the term $I_6$}}.
If $|\alpha_0|=m$ or $|\alpha_{13}|\ge 1$, it is easy to deduce that
\begin{equation}\label{3349}
\int_0^t \int_{\partial \Omega}\mathcal{Z}^\alpha \Delta d
  \cdot \mathcal{Z}^\alpha \nabla d\cdot n d\sigma d\tau
=0.
\end{equation}
For the case of $|\alpha_0|\le m-1$ or $|\alpha_{13}|=0$,
integrating by part along the boundary, we have for $|\beta|=m-1-\alpha_0$
\begin{equation}\label{3350}
\int_0^t \int_{\partial \Omega}\mathcal{Z}^\alpha \Delta d
  \cdot \mathcal{Z}^\alpha \nabla d\cdot n d\sigma d\tau
\le C\int_0^t |\partial_t^{\alpha_0}Z_y^\beta \Delta d|_{L^2(\partial \Omega)}
    |\mathcal{Z}^\alpha \nabla d\cdot n|_{H^1(\partial \Omega)}d\tau.
\end{equation}
By virtue of the trace theorem in Proposition \ref{prop2.3}, one arrives at
\begin{equation}\label{3351}
\begin{aligned}
|\partial_t^{\alpha_0}Z_y^\beta \Delta d|_{L^2(\partial \Omega)}
&\le C(\|\nabla \partial_t^{\alpha_0} Z_y^\beta \Delta d\|^{\frac{1}{2}}
      +\|\partial_t^{\alpha_0} Z_y^\beta \Delta d\|^{\frac{1}{2}})
      \| \partial_t^{\alpha_0} Z_y^\beta \Delta d\|^{\frac{1}{2}}\\
&\le C(\|\nabla \Delta d\|_{\mathcal{H}^{m-1}}^{\frac{1}{2}}
      +\|\Delta d\|_{\mathcal{H}^{m-1}}^{\frac{1}{2}})
      \|\Delta d\|_{\mathcal{H}^{m-1}}^{\frac{1}{2}}.
\end{aligned}
\end{equation}
Similarly, in view of boundary condition
\eqref{bc2} and trace theorem in Proposition \ref{prop2.3}, one attains
\begin{equation}\label{3352}
|\mathcal{Z}^\alpha \nabla d\cdot n|_{H^1(\partial \Omega)}
\le C_{m+2}(\|\nabla^2 d\|_{\mathcal{H}^{m}}^{\frac{1}{2}}
     +\|\nabla d\|_{\mathcal{H}^{m}}^{\frac{1}{2}})
     \|\nabla d\|_{\mathcal{H}^{m}}^{\frac{1}{2}}.
\end{equation}
Substituting  \eqref{3351} and \eqref{3352}  into \eqref{3350}
and applying the Cauchy inequality, we find
\begin{equation}\label{3353}
|I_6|
\le \!\delta \!\!\int_0^t \!\! \|\nabla \Delta d\|_{\mathcal{H}^{m-1}}^2 d\tau
     +\delta_1 \!\! \int_0^t\!\!  \|\nabla^2 d\|_{\mathcal{H}^{m}}^2 d\tau
     +C_{\delta, \delta_1}C_{m+2}
     \!\!\int_0^t\!\!(\|\nabla d\|_{\mathcal{H}^{m}}^2
      +\|\Delta d\|_{\mathcal{H}^{m-1}}^2)d\tau.
\end{equation}
Substituting \eqref{3322}, \eqref{3346}-\eqref{3348}
and \eqref{3353} into \eqref{3314}
and choosing $\delta_1$ small enough, we find
\begin{equation}\label{3354}
\begin{aligned}
&\frac{1}{2}\int |\mathcal{Z}^\alpha \nabla d(t)|^2 dx
+\frac{3}{4}\int_0^t \int |\mathcal{Z}^\alpha \Delta d|^2dxd\tau\\
&\le \frac{1}{2}\int |\mathcal{Z}^\alpha \nabla d_0|^2 dx
  +\delta_2 \!\int_0^t \!\|\nabla^2 d\|_{\mathcal{H}^{m}}^2d\tau
  +\delta \int_0^t \|\nabla \Delta d\|_{\mathcal{H}^{m-1}}^2 d\tau\\
&\quad +C_{\delta, \delta_2}C_{m+2}(1+P(Q(t)))\int_0^t \Lambda_m(\tau)d\tau.
\end{aligned}
\end{equation}
In view of the standard elliptic regularity results with Neumann boundary condition,
we get that
\begin{equation}\label{3355}
\begin{aligned}
&\|\nabla^2 d\|_{\mathcal{H}^{m}}^2
=\|\nabla^2 \partial_t^{\alpha_0}d\|_{m-\alpha_0}^2\\
&\le C_{m+2}(\|\nabla \partial_t^{\alpha_0}d\|_{L^2}^2
+\|\Delta\partial_t^{\alpha_0}d\|_{m-\alpha_0}^2)\\
&\le C_{m+2}(\|\nabla d\|_{\mathcal{H}^{m}}^2+\|\Delta d\|_{\mathcal{H}^{m}}^2).
\end{aligned}
\end{equation}
The combination of \eqref{3311}, \eqref{3354} and \eqref{3355} yields directly
\begin{equation*}
\begin{aligned}
&\underset{0\le \tau \le t}{\sup}\|(u, p, \nabla d)\|_{\mathcal{H}^{m}}^2
+C\varepsilon \int_0^t \|\nabla u\|_{\mathcal{H}^{m}}^2 d\tau
+\int_0^t \|\Delta d\|_{\mathcal{H}^{m}}^2 d\tau\\
&\le C_{m+2}\left\{\|(u_0, p_0, \nabla d_0)\|_{\mathcal{H}^{m}}^2
     +\delta \int_0^t \|\nabla \Delta d\|_{\mathcal{H}^{m-1}}^2d\tau
     +\delta \varepsilon^2 \int\|\nabla^2 u\|_{\mathcal{H}^{m-1}}^2 d\tau\right.\\
&\quad \quad \quad \quad \quad
\left. +\delta \int_0^t \|\nabla \partial_t^{m-1}p(\tau)\|_{L^2}^2 d\tau
+C_{\delta}[1+P(Q(t))]\int_0^t \Lambda_m(\tau)d\tau\right\}.
\end{aligned}
\end{equation*}
Therefore, we complete the proof of Lemma \ref{lemma3.3}.
\end{proof}

\subsection{Normal Derivatives Estimates}
\quad
In order to estimate $\|\nabla u\|_{\mathcal{H}^{m-1}}$, it remains to
estimate $\|\chi_j \partial_n u\|_{\mathcal{H}^{m-1}}$, where $\chi_j$
is supported compactly in one of the $\Omega_j$ and with value one in a neighborhood
of the boundary. Indeed, it follows from the definition of the norm that
$\|\chi \partial_{y_i}u\|_{\mathcal{H}^{m-1}}\le C\|u\|_{\mathcal{H}^m}, i=1,2$.
Then, it suffices to estimate $\|\chi \partial_n u\|_{\mathcal{H}^{m-1}}$.

Note that
\begin{equation}\label{32a}
{\rm div}u=\partial_n u\cdot n
+(\Pi \partial_{y_1}u)_1
+(\Pi \partial_{y_1}u)_2
\end{equation}
and
\begin{equation}\label{32b}
\partial_n u=(\partial_n u\cdot n)n+\Pi(\partial_n u).
\end{equation}
Then, it follows from \eqref{32a}  and  \eqref{32b} that
$$
\begin{aligned}
\|\chi \partial_n u\|_{\mathcal{H}^{m-1}}
&\le \|\chi \partial_n u\cdot n\|_{\mathcal{H}^{m-1}}
    +\|\chi \Pi(\partial_n u)\|_{\mathcal{H}^{m-1}}\\
&\le C_{m}\{\|\chi {\rm div}u\|_{\mathcal{H}^{m-1}}
      +\|\chi \Pi(\partial_n u)\|_{\mathcal{H}^{m-1}}
      +\|u\|_{\mathcal{H}^{m}}\}.
\end{aligned}
$$
Thus, it suffices to estimate $\|\chi \Pi(\partial_n u)\|_{\mathcal{H}^{m-1}}$
and $\|\chi {\rm div}u\|_{\mathcal{H}^{m-1}}$, since $\|u\|_{\mathcal{H}^{m}}$
has been estimated before(see Lemma \ref{lemma3.3}).
We extend the smooth symmetric matrix $A$ to be
$A(y, z)=A(y)$. Define
\begin{equation}\label{32c}
\eta=\chi(w\times n+\Pi(Bu))=\chi(\Pi(w\times n)+\Pi(Bu)).
\end{equation}
In view of the Navier-slip boundary condition \eqref{bc2}, $\eta$ satisfies
\begin{equation}\label{32d}
\eta|_{\partial \Omega}=0.
\end{equation}
Since $w\times n=(\nabla u-(\nabla u)^t)\cdot n$, then $\eta$ can be rewritten as
\begin{equation*}
\eta=\chi\left\{\Pi(\partial_n u)-\Pi(\nabla(u\cdot n))
+\Pi((\nabla n)^t \cdot u)+\Pi(Bu)\right\},
\end{equation*}
which, yields immediately that
\begin{equation}\label{32e}
\|\chi \Pi(\partial_n u)\|_{\mathcal{H}^{m-1}}
\le C_{m+1}(\|\eta\|_{\mathcal{H}^{m-1}}+\|u\|_{\mathcal{H}^m}).
\end{equation}
Hence, it remains to estimate $\|\eta\|_{\mathcal{H}^{m-1}}$.

\begin{lemm}\label{lemma3.4}
For $m\ge 1$, it holds that
\begin{equation}\label{341}
\begin{aligned}
&\underset{0\le \tau \le t}{\sup}\|\eta(\tau)\|_{L^2}^2
+\varepsilon\int_0^t \|\nabla \eta(\tau)\|_{L^2}^2d\tau\\
&\le CC_3\left\{\|u_0\|_{H^1}^2
+\delta \varepsilon^2 \int_0^t \|\nabla^2 u(\tau)\|_{L^2}^2d\tau
+\delta \int_0^t\|\nabla \Delta d\|_{L^2}^2d\tau\right\}\\
&\quad + C_3C_\delta [1+P(Q(t))]\int_0^t \Lambda_m(\tau)d\tau.
\end{aligned}
\end{equation}
\end{lemm}
\begin{proof}
Notice that
$$
\nabla \times ((u\cdot \nabla)u)
=(u\cdot \nabla)w-(w\cdot \nabla)u+w{\rm div}u,
$$
so $w$ satisfies the following equations:
\begin{equation}\label{342}
\rho w_t+\rho (u\cdot \nabla)w
=\mu \varepsilon \Delta w+F_1,
\end{equation}
where
\begin{equation*}
F_1\triangleq
-\nabla \rho \times u_t-\nabla \rho \times (u\cdot \nabla)u
+\rho (w \cdot \nabla)u-\rho w {\rm div}u
-\nabla \times (\nabla d\cdot \Delta d).
\end{equation*}
Consequently, the system for $\eta$ is
\begin{equation}\label{343}
\begin{aligned}
&\rho \eta_t+\rho u_1 \partial_{y_1}\eta+\rho u_2 \partial_{y_2}\eta
+\rho u\cdot N \partial_z \eta-\mu \varepsilon \Delta \eta\\
&=\chi [F_1 \times n+\Pi(BF_2)]+\chi F_3+F_4-\mu \varepsilon \chi \Delta (\Pi B)\cdot u,
\end{aligned}
\end{equation}
where
\begin{equation*}
\begin{aligned}
F_2=&(\mu+\lambda)\varepsilon \nabla {\rm div}u-\nabla p-\nabla d\cdot \Delta d,\\
F_3=&-2\mu \varepsilon\sum_{i=1}^2 \partial_j w\times \partial_j n
    -\mu \varepsilon w\times \Delta n+\sum_{j=1}^2 \rho u_i w \times \partial_i n\\
    &+\sum_{i=1}^2 \rho u_i \partial_i(\Pi B)u
     -2\mu \varepsilon \sum_{i=1}^2\partial_i(\Pi B)\partial_i u,\\
F_4=&\sum_{i=1}^2\rho u_i \partial_{y_i}\chi \cdot (w\times n+\Pi(Bu))
  +\rho u\cdot N \partial_z \chi \cdot(w\times n+\Pi(Bu))\\
&-2\mu \varepsilon \sum_{i=1}^3 \partial_i \chi \partial_i(w\times n+\Pi(Bu))
-\mu \varepsilon \Delta \chi \cdot (w\times n+\Pi(Bu)).
\end{aligned}
\end{equation*}
Multiplying \eqref{343} by $\eta$, it is easy to check that
\begin{equation}\label{344}
\frac{1}{2}\frac{d}{dt}\int |\eta|^2 dx+\varepsilon \int |\nabla \eta|^2 dx
=\int F\cdot \eta dx-\mu \varepsilon  \int \chi \Delta(\Pi B)\cdot u \cdot \eta dx,
\end{equation}
where $F\triangleq \chi [F_1 \times n+\Pi(BF_2)]+\chi F_3+F_4$.
It is easy to deduce that
\begin{equation}\label{345}
\begin{aligned}
&\|\chi F_1 \times n\|_{L^2}
\le C_2\left\{[1+P(Q(t))](\|\nabla u\|_{L^2} +\|\nabla p\|_{L^2})
       +\|\nabla d\|_{L^\infty}\|\nabla \Delta d\|_{L^2}\right\},\\
&\|\chi\Pi(BF_2)\|_{L^2}
\le C_2(\varepsilon \|\nabla^2 u\|_{L^2}
    +\|\nabla d\|_{L^\infty}\|\Delta d\|_{L^2}+\|\nabla p\|_{L^2}),\\
&\|\chi F_3\|_{L^2}
\le \varepsilon \|\nabla^2 u\|_{L^2}+C_3(1+\|u\|_{L^\infty})
(\|u\|_{L^2}+\|\nabla u\|_{L^2}).
\end{aligned}
\end{equation}
Notice that the term $F_4$ are supported away from the boundary, we can
control all the derivatives by the $\|\cdot\|_{\mathcal{H}^m}$. Hence, we find
\begin{equation}\label{346}
\|F_4\|_{L^2}
\le \varepsilon \|\nabla^2 u\|_{L^2}+C_3(1+\|u\|_{L^\infty})\|u\|_{\mathcal{H}^{1}}.
\end{equation}
Integrating by parts, it is easy to deduce that
\begin{equation}\label{347}
-\mu \varepsilon \int \chi \Delta(\Pi B)\cdot u \cdot \eta dx
\le \delta\varepsilon \int |\nabla \eta|^2  dx
+C_{\delta}C_{3}(\|\nabla u\|_{L^2}^2+\|u\|_{\mathcal{H}^{1}}^2).
\end{equation}
Substituting \eqref{345}-\eqref{347} into \eqref{344}
and integrating the resulting inequality over $[0, t]$, we have
\begin{equation*}
\begin{aligned}
&\frac{1}{2}\int |\eta|^2(t) dx+\varepsilon \int_0^t \int |\nabla \eta|^2 dxd\tau\\
&\le \frac{1}{2}\int |\eta_0|^2 dx
+\delta \varepsilon^2 \int_0^t \|\nabla^2 u\|_{L^2}^2d\tau
+\delta \int_0^t\|\nabla \Delta d\|_{L^2}^2d\tau
+C[1+P(Q(t))]\int_0^t\Lambda_1 (\tau) d\tau.
\end{aligned}
\end{equation*}
Therefore, we complete the proof of Lemma \ref{lemma3.4}.
\end{proof}

\begin{lemm}\label{lemma3.5}
For $m\ge 1$, it holds that
\begin{equation}\label{351}
\begin{aligned}
&\underset{0\le \tau \le t}{\sup}\|\eta(\tau)\|_{\mathcal{H}^{m-1}}^2
+\mu \varepsilon\int_0^t \|\nabla \eta(\tau)\|_{\mathcal{H}^{m-1}}^2d\tau\\
&\le CC_{m+2}\left\{\|(u_0, \nabla u_0)\|_{\mathcal{H}^{m-1}}^2
+\delta \int_0^t \|\nabla \partial_t^{m-1} p\|_{L^2}^2 d\tau
+\delta \varepsilon^2 \int_0^t \|\nabla^2 u(\tau)\|_{\mathcal{H}^{m-1}}^2d\tau
\right\}\\
&\quad +CC_{m+2}\left\{
\delta \int_0^t \|\nabla \Delta d\|_{\mathcal{H}^{m-1}}^2 d\tau
+C_\delta [1+P(Q(t))]\int_0^t \Lambda_m(\tau) d\tau\right\}.
\end{aligned}
\end{equation}
\end{lemm}
\begin{proof}
The case for $m=1$ is already proved in Lemma \ref{lemma3.4}. Assume that
\eqref{351} is proved for $k=m-2$. We shall prove that is holds for $k=m-1 \ge 1$.
For $|\alpha|=m-1$, applying the operator $\mathcal{Z}^\alpha$ to the
equation \eqref{343} , we find
\begin{equation}\label{352}
\rho \mathcal{Z}^\alpha \eta_t
+\rho(u\cdot \nabla)\mathcal{Z}^\alpha \eta
-\mu \varepsilon  \mathcal{Z}^\alpha \Delta \eta=\mathcal{Z}^\alpha F
-\mathcal{Z}^\alpha [\mu \varepsilon \chi (\Delta(\Pi B)\cdot u)]
+\mathcal{C}_3^{\alpha}
+\mathcal{C}_4^{\alpha},
\end{equation}
where
\begin{equation*}
\begin{aligned}
\mathcal{C}_3^{\alpha}=&-
\sum_{|\beta|\ge 1, \beta+\gamma=\alpha}C_{\alpha, \beta}
\mathcal{Z}^\beta \rho \mathcal{Z}^\gamma \eta_t,\\
\mathcal{C}_4^{\alpha}=&
-\sum_{|\beta|\ge 1, \beta+\gamma=\alpha}
\sum_{i=1}^2 C_{\alpha, \beta}\mathcal{Z}^\beta (\rho u_i)
\mathcal{Z}^\gamma \partial_{y_i}\eta,\\
&-\sum_{|\beta|\ge 1, \beta+\gamma=\alpha}
C_{\alpha, \beta}\mathcal{Z}^\beta (\rho u\cdot N)
\mathcal{Z}^\gamma \partial_z \eta,\\
&-\rho(u\cdot N)\sum_{|\beta|\ge m-2}C(\alpha, \beta, z)
\partial_z \mathcal{Z}^\beta \eta,
\end{aligned}
\end{equation*}
where $C(\alpha, \beta, z)$ is smooth function depending on $\alpha, \beta$
and $\varphi(z)$.
Multiplying \eqref{352} by $\mathcal{Z}^\alpha \eta$, it is easy to deduce that
\begin{equation}\label{353}
\begin{aligned}
&\frac{1}{2}\int \rho|\mathcal{Z}^\alpha \eta(t)|^2 dx
-\frac{1}{2}\int \rho_0|\mathcal{Z}^\alpha \eta_0|^2 dx\\
&=\mu \varepsilon \int_0^t \int \mathcal{Z}^\alpha \Delta \eta \cdot \mathcal{Z}^\alpha \eta~ dxd\tau
+\int_0^t \int \mathcal{Z}^\alpha F \cdot \mathcal{Z}^\alpha \eta ~dxd\tau\\
&\quad -\mu \varepsilon\int_0^t \int  \mathcal{Z}^\alpha (\chi \Delta(\Pi B)\cdot u)
\cdot \mathcal{Z}^\alpha \eta ~dxd\tau
+\int_0^t \int (\mathcal{C}_3^\alpha+\mathcal{C}_4^\alpha)
\cdot \mathcal{Z}^\alpha \eta ~dxd\tau.
\end{aligned}
\end{equation}
In the local basis, it holds that
\begin{equation*}
\partial_j=\beta_j^1 \partial_{y_1}+\beta_j^2 \partial_{y_2}+\beta_j^3 \partial_{z},
~~j=1,2,3,
\end{equation*}
for harmless functions $\beta_j^i, ~i,j=1,2,3$ depending on the boundary regularity and
weighted function $\varphi(z)$. Therefore, the following commutation expansion holds:
\begin{equation*}
\mathcal{Z}^\alpha \Delta \eta
=\Delta \mathcal{Z}^\alpha \eta
+\sum_{|\beta|\le m-2} C_{1\beta}\partial_{zz}\mathcal{Z}^\beta \eta
+\sum_{|\beta|\le m-1} (C_{2\beta}\partial_{z}\mathcal{Z}^\beta \eta
                        +C_{3\beta}Z_y\mathcal{Z}^\beta \eta).
\end{equation*}
Then integrating by part and applying the Cauchy inequality, we obtain
\begin{equation}\label{354}
\begin{aligned}
&\mu \varepsilon \int_0^t \int \mathcal{Z}^\alpha \Delta \eta \cdot \mathcal{Z}^\alpha \eta dxd\tau\\
&=\mu \varepsilon\int_0^t \int \Delta \mathcal{Z}^\alpha \eta \cdot \mathcal{Z}^\alpha \eta dxd\tau
  +\sum_{|\beta|\le m-2}\mu \varepsilon \int_0^t \int C_{1\beta}\partial_{zz} \mathcal{Z}^\beta \eta\cdot \mathcal{Z}^\alpha \eta dxd\tau\\
&\quad +\sum_{|\beta|\le m-1}\mu \varepsilon \int_0^t
   \int(C_{2\beta}\partial_{z}\mathcal{Z}^\beta \eta
  +C_{3\beta}Z_y\mathcal{Z}^\beta \eta)\cdot \mathcal{Z}^\alpha \eta dxd\tau\\
&\le -\frac{3}{4} \mu \varepsilon\int_0^t \|\nabla \mathcal{Z}^\alpha \eta\|_{L^2}^2 d\tau
     +C \mu \varepsilon\int_0^t \|\nabla \eta\|_{\mathcal{H}^{m-2}}^2 d\tau
     +C_{m+2} \mu \varepsilon \int_0^t  \|\eta\|_{\mathcal{H}^{m-1}}^2 d\tau.
\end{aligned}
\end{equation}
Note that there is no boundary term in the integrating by parts since $\mathcal{Z}^\alpha \eta$
vanishes one the boundary.
Substituting \eqref{354} into \eqref{353}, we find
\begin{equation}\label{355}
\begin{aligned}
&\frac{1}{2}\int |\mathcal{Z}^\alpha \eta(t)|^2 dx
+\frac{3}{4}\mu \varepsilon \int_0^t \|\nabla \mathcal{Z}^\alpha \eta\|_{L^2}^2 d\tau\\
&\le\frac{1}{2}\int |\mathcal{Z}^\alpha \eta_0|^2 dx
+C\mu \varepsilon \int_0^t \|\nabla \eta\|_{\mathcal{H}^{m-2}}^2 d\tau
+C_{m+2}\varepsilon\int_0^t  \|\eta\|_{\mathcal{H}^{m-1}}^2 d\tau\\
&\quad+\int_0^t \int \mathcal{Z}^\alpha F \cdot \mathcal{Z}^\alpha \eta dxd\tau
-\mu \varepsilon\int_0^t \int  \mathcal{Z}^\alpha (\chi \Delta(\Pi B)\cdot u)
\cdot \mathcal{Z}^\alpha \eta ~dxd\tau\\
&\quad  +\int_0^t \int (\mathcal{C}_3^{\alpha}+\mathcal{C}_4^{\alpha})
\cdot \mathcal{Z}^\alpha \eta ~dxd\tau.
\end{aligned}
\end{equation}
Similar to \eqref{345}-\eqref{346}, we apply the Proposition \ref{prop2.2} to deduce that
\begin{equation}\label{356}
\int_0^t \int \mathcal{Z}^\alpha(\chi F_1 \times n)\cdot \mathcal{Z}^\alpha \eta dxd\tau
\le  C_m\left\{ \delta \! \!\int_0^t  \! \!\|\nabla \Delta d\|_{\mathcal{H}^{m-1}}^2 d\tau
    +C_\delta(1+P(Q(t))) \! \!\int_0^t \! \! \Lambda_m(\tau)d\tau\right\},
\end{equation}
\begin{equation}\label{356-1}
\begin{aligned}
&\int_0^t \int \mathcal{Z}^\alpha(\chi \Pi(BF_2))\cdot \mathcal{Z}^\alpha \eta dxd\tau\\
&\le C_{m+1}\left\{\delta \int_0^t \|\nabla \partial_t^{m-1} p\|_{L^2}^2d\tau
+\delta \varepsilon^2 \int_0^t \|\nabla^2 u\|_{\mathcal{H}^{m-1}}^2d\tau
+(1+P(Q(t)))\int_0^t \Lambda_m (\tau)d\tau\right\},\\
\end{aligned}
\end{equation}
\begin{equation}\label{356-2}
\int_0^t \!\!\int \mathcal{Z}^\alpha(\chi F_3)\cdot \mathcal{Z}^\alpha \eta dxd\tau
\le C_{m+2}\left\{\delta \varepsilon^2\!\! \int_0^t \|\nabla^2 u\|_{\mathcal{H}^{m-1}}^2 d\tau+C_\delta(1+P(Q(t)))\!\!\int_0^t \Lambda_m(\tau)d\tau\right\},
\end{equation}
and
\begin{equation}\label{357}
\int_0^t \int \mathcal{Z}^\alpha F_4 \cdot \mathcal{Z}^\alpha \eta dxd\tau
\le C_{m+1}\left\{\delta \varepsilon^2 \int_0^t \|\nabla^2 u\|_{\mathcal{H}^{m-1}}^2 d\tau
     +C_\delta(1+P(Q(t)))\int_0^t \Lambda_m(\tau)d\tau\right\}.
\end{equation}
Then, the combination of \eqref{356}-\eqref{357} gives directly
\begin{equation}\label{358}
\begin{aligned}
\int_0^t \int \mathcal{Z}^\alpha F \cdot \mathcal{Z}^\alpha \eta dxd\tau
&\le C_{m+2}\left\{ \delta \int_0^t \|\nabla \Delta d\|_{\mathcal{H}^{m-1}}^2 d\tau
+\delta \varepsilon^2 \int \|\nabla^2 u\|_{\mathcal{H}^{m-1}}^2 d\tau\right\}\\
&\quad +C_{m+2}\left\{\delta\int_0^t \|\nabla \partial_t^{m-1}p\|_{L^2}^2 d\tau
     +C_\delta (1+P(Q(t)))\int_0^t \Lambda_m(\tau)d\tau\right\}.
\end{aligned}
\end{equation}
Integrating by parts, one arrives at directly
\begin{equation}\label{359}
\left|\mu \varepsilon\int_0^t \int  \mathcal{Z}^\alpha (\chi \Delta(\Pi B)\cdot u)
\cdot \mathcal{Z}^\alpha \eta ~dxd\tau\right|
\le \delta \mu \varepsilon^2 \int_0^t \|\nabla \eta\|_{\mathcal{H}^{m-1}}^2 d\tau
+C_\delta C_{m+2}\int_0^t \Lambda_m(\tau)d\tau.
\end{equation}
Using the same argument as Lemma 3.13 of \cite{Wang-Xin-Yong},
one can obtain the following estimates
\begin{equation}\label{3510}
\int_0^t \int (\mathcal{C}_3^{\alpha}+\mathcal{C}_4^{\alpha})
\cdot \mathcal{Z}^\alpha \eta dxd\tau
\le C_m(1+P(Q(t)))\int_0^t \Lambda_m(\tau)d\tau.
\end{equation}
Substituting \eqref{358}, \eqref{359} and \eqref{3510} into \eqref{355}, we find
\begin{equation*}
\begin{aligned}
&\frac{1}{2}\int |\mathcal{Z}^\alpha \eta(t)|^2 dx
+\frac{3\mu\varepsilon}{4} \int_0^t \|\nabla \mathcal{Z}^\alpha \eta\|_{L^2}^2 d\tau\\
&\le C_{m+2}\left\{\frac{1}{2}\int |\mathcal{Z}^\alpha \eta_0|^2 dx
+C\varepsilon \int_0^t \|\nabla \eta\|_{\mathcal{H}^{m-2}}^2 d\tau
+\delta \int_0^t \|\nabla \Delta d\|_{\mathcal{H}^{m-1}}^2d\tau\right\}\\
&\quad +C_{m+2}\left\{
     \delta \varepsilon^2 \int_0^t \|\nabla^2 u\|_{\mathcal{H}^{m-1}}^2 d\tau
     +\delta \int_0^t \|\nabla \partial_t^{m-1}p\|_{L^2}^2d\tau
     +C_\delta (1+P(Q(t)))\int_0^t \Lambda_m(t)d\tau\right\}.
\end{aligned}
\end{equation*}
By the induction assumption, one can eliminate the term $\varepsilon \int_0^t \|\nabla \eta\|_{\mathcal{H}^{m-2}}^2 d\tau$.
Therefore, we complete the proof of Lemma \ref{lemma3.5}.
\end{proof}

\subsection{Estimates for the $\Delta d$, {\rm div}$u$ and $\Delta p$ }
\quad
In this subsection, we shall get some uniform estimates for
$\Delta d$, {\rm div}$u$ and $\Delta p$
in conormal Sobolev space.

\begin{lemm}\label{lemma3.3-1}
For a smooth solution to \eqref{eq1} and \eqref{bc2}, it holds that
for $\varepsilon \in (0, 1]$
\begin{equation}\label{331-1}
\underset{0\le \tau \le t}{\sup}\|\Delta d(\tau)\|_{L^2}^2
+\int_0^t\|\nabla \Delta d\|_{L^2}^2 d\tau
\le \|\Delta d_0\|_{L^2}^2
+C(1+Q(t)^2)\int_0^t \Lambda_1(\tau)d\tau.
\end{equation}
\end{lemm}
\begin{proof}
Taking $\nabla$ operator to the equation \eqref{eq1}$_2$, one arrives at
\begin{equation}\label{32222}
\nabla d_t-\nabla \Delta d=-\nabla(u\cdot \nabla d)+\nabla(|\nabla d|^2 d).
\end{equation}
Multiplying \eqref{32222} by $-\nabla \Delta d$,  we find
\begin{equation}\label{339}
\begin{aligned}
&-\int \nabla d_t\cdot \nabla \Delta d ~dx +\int |\nabla \Delta d|^2 dx\\
&=\int\nabla(u\cdot \nabla d)\cdot \nabla \Delta d~dx
-\int \nabla(|\nabla d|^2 d)\cdot \nabla \Delta d~dx.
\end{aligned}
\end{equation}
By integrating by parts and applying the Neumann boundary condition \eqref{bc2}, we get
\begin{equation}\label{3310}
\begin{aligned}
-\int \nabla d_t \cdot \nabla \Delta d ~dx
=-\int_{\partial \Omega}n\cdot \nabla d_t \cdot \Delta d~d\sigma
 +\int \Delta d_t \cdot \Delta d ~dx
=\frac{1}{2}\frac{d}{dt}\int |\Delta d|^2 dx.
\end{aligned}
\end{equation}
In view of the Cauchy inequality, we obtain
\begin{equation}\label{3311}
\begin{aligned}
&\int\nabla(u\cdot \nabla d)\cdot \nabla \Delta d~dx
\le \delta \|\nabla \Delta d\|_{L^2}^2
+ C_\delta \|u\|_{W^{1,\infty}}^2(\|\nabla d\|_{L^2}^2+\|\nabla^2 d\|_{L^2}^2),\\
&-\int \nabla(|\nabla d|^2 d)\cdot \nabla \Delta d~dx
\le \delta \|\nabla \Delta d\|_{L^2}^2
+\|\nabla d\|_{L^\infty}^4 \|\nabla d\|_{L^2}^2
+\|\nabla d\|_{L^2}^2 \|\nabla^2 d\|_{L^2}^2.
\end{aligned}
\end{equation}
Substituting \eqref{3311} into \eqref{3310}, choosing $\delta$ small enough
and integrating over $[0, t]$, one attains
\begin{equation*}
\begin{aligned}
&\frac{1}{2}\int |\Delta d|^2(t) dx
+\frac{3}{4}\int |\nabla \Delta d|^2 dx\\
&\le \int |\Delta d_0|^2 dx
+C(\|u\|_{W^{1,\infty}}^2+\|\nabla d\|_{L^\infty}^4)
\int_0^t(\|\nabla d\|_{L^2}^2+\|\nabla^2 d\|_{L^2}^2)d\tau.
\end{aligned}
\end{equation*}
Therefore, we complete the proof of Lemma \ref{lemma3.3-1}.
\end{proof}

Next, we can establish the following conormal estimates for the quantity $\Delta d$.

\begin{lemm}\label{lemma3.4-1}
For $m \ge  1$ and a smooth solution to \eqref{eq1} and \eqref{bc2}, it holds that
for $\varepsilon \in (0, 1]$
\begin{equation}\label{341-1}
\begin{aligned}
&\underset{0\le \tau \le t}{\sup}
\|\Delta d(\tau)\|_{\mathcal{H}^{m-1}}^2
+\int_0^t \|\nabla \Delta d\|_{\mathcal{H}^{m-1}}^2 d\tau\\
&\le\!\! C_{m}\left\{
      \|\Delta d_0\|_{\mathcal{H}^{m-1}}^2
      +\delta \!\!\int_0^t \!\!\| \Delta d\|_{\mathcal{H}^{m}}^2 d\tau
      +C_\delta\left(1+P(Q(t))\right)\int_0^t \Lambda_m(t)d\tau\right\}.
\end{aligned}
\end{equation}
\end{lemm}
\begin{proof}
The case for $m=1$ is already proved in Lemma \ref{lemma3.3-1}. Assume that
\eqref{341-1} is proved for $k=m-2$. We shall prove that is holds for $k=m-1 \ge 1$.
For $|\alpha|=m-1$,
multiplying \eqref{32222} by $-\nabla \mathcal{Z}^\alpha \Delta d$, we find
\begin{equation}\label{3412}
\begin{aligned}
&-\int \mathcal{Z}^\alpha \nabla d_t \cdot \nabla \mathcal{Z}^\alpha \Delta d ~dx
+\int \mathcal{Z}^\alpha \nabla \Delta d\cdot \nabla \mathcal{Z}^\alpha \Delta d ~dx\\
&=\int \mathcal{Z}^\alpha \nabla(u\cdot \nabla d)\cdot
  \nabla \mathcal{Z}^\alpha \Delta d ~dx
  +\int \mathcal{Z}^\alpha \nabla(|\nabla d|^2 d)\cdot
  \nabla \mathcal{Z}^\alpha \Delta d ~dx.
\end{aligned}
\end{equation}
Integrating by part, it is easy to deduce that
\begin{equation}\label{3413}
\begin{aligned}
&-\int \mathcal{Z}^\alpha \nabla d_t \cdot \nabla \mathcal{Z}^\alpha \Delta d ~dx\\
&=-\int_{\partial \Omega}n\cdot\mathcal{Z}^\alpha \nabla d_t\cdot \mathcal{Z}^\alpha \Delta d~d\sigma
+\int {\nabla \cdot}(\mathcal{Z}^\alpha \nabla d_t)\cdot \mathcal{Z}^\alpha \Delta d~dx\\
&=-\int_{\partial \Omega}n\cdot\mathcal{Z}^\alpha \nabla d_t\cdot \mathcal{Z}^\alpha \Delta d~d\sigma
+\frac{1}{2}\frac{d}{dt}\int |\mathcal{Z}^\alpha \Delta d|^2dx
-\int [\mathcal{Z}^\alpha, {\nabla \cdot}] \nabla d_t\cdot \mathcal{Z}^\alpha \Delta d~dx.
\end{aligned}
\end{equation}
It is easy to check that
\begin{equation}\label{3414}
\int \mathcal{Z}^\alpha \nabla \Delta d\cdot \nabla \mathcal{Z}^\alpha \Delta d ~dx
=\int |\nabla \mathcal{Z}^\alpha \Delta d|^2dx
+\int [\mathcal{Z}^\alpha, \nabla] \Delta d\cdot \nabla \mathcal{Z}^\alpha \Delta d ~dx.
\end{equation}
Substituting \eqref{3413} and \eqref{3414} into \eqref{3412} and
integrating over $[0, t]$,  we find
\begin{equation}\label{3415}
\begin{aligned}
&\frac{1}{2}\int |\mathcal{Z}^\alpha \Delta d(t)|^2dx
+\int_0^t \int |\nabla \mathcal{Z}^\alpha \Delta d|^2dx d\tau\\
&=
\frac{1}{2}\int |\mathcal{Z}^\alpha \Delta d_0|^2dx
+\int_0^t \int_{\partial \Omega}n\cdot\mathcal{Z}^\alpha \nabla d_t\cdot \mathcal{Z}^\alpha \Delta d~d\sigma d\tau
+\int_0^t \int [\mathcal{Z}^\alpha, {\nabla \cdot}] \nabla d_t\cdot \mathcal{Z}^\alpha \Delta d~dxd\tau\\
&\quad -\int_0^t \int [\mathcal{Z}^\alpha, \nabla] \Delta d\cdot \nabla \mathcal{Z}^\alpha \Delta d ~dxd\tau
+\int_0^t\int \mathcal{Z}^\alpha \nabla(u\cdot \nabla d)\cdot
  \nabla \mathcal{Z}^\alpha \Delta d ~dxd\tau\\
&\quad +\int_0^t \int \mathcal{Z}^\alpha \nabla(|\nabla d|^2 d)\cdot
  \nabla \mathcal{Z}^\alpha \Delta d ~dxd\tau\\
&:=\Rmnum{3}_1+\Rmnum{3}_2+\Rmnum{3}_3+\Rmnum{3}_4+\Rmnum{3}_5+\Rmnum{3}_6.
\end{aligned}
\end{equation}
To deal with the boundary term on the right hand side of \eqref{3415}.
If $|\alpha_0|=m-1$, then we have
\begin{equation}\label{3416-1}
\int_0^t \int_{\partial \Omega}n\cdot\mathcal{Z}^\alpha \nabla d_t\cdot \mathcal{Z}^\alpha \Delta d~d\sigma d\tau=0.
\end{equation}
On the other hand, it is easy to deduce that for $|\alpha_0|\le m-2$
\begin{equation}\label{3416}
\int_0^t \int_{\partial \Omega}n\cdot\mathcal{Z}^\alpha \nabla d_t\cdot \mathcal{Z}^\alpha \Delta d~d\sigma d\tau
\le \int_0^t|n\cdot\mathcal{Z}^\alpha \nabla d_t|_{L^2(\partial \Omega)}
    |\mathcal{Z}^{{\alpha}}\Delta d|_{L^2(\partial \Omega)}d\tau.
\end{equation}
The application of trace inequality in Proposition \ref{prop2.3}
and the boundary condition \eqref{bc2} implies
\begin{equation}\label{3417}
\begin{aligned}
&|\mathcal{Z}^{{{\alpha}}}\Delta d|_{L^2(\partial \Omega)}
=|\partial_t^{\alpha_0} \Delta d|_{H^{m-1-|\alpha_0|}(\partial \Omega)}\\
&\le C\|\nabla \partial_t^{\alpha_0} \Delta d\|_{m-1-|\alpha_0|}
     +C\|\partial_t^{\alpha_0} \Delta d\|_{m-|\alpha_0|}\\
&\le C\|\nabla \Delta d\|_{\mathcal{H}^{m-1}}
     + C\|\Delta d\|_{\mathcal{H}^{m}},
\end{aligned}
\end{equation}
and
\begin{equation}\label{3418}
\begin{aligned}
&|n\cdot\mathcal{Z}^\alpha \nabla d_t|_{L^2(\partial \Omega)}
\le C_{m}|\partial_t^{\alpha_0}\nabla d_t|_{H^{m-2-|\alpha_0|}(\partial \Omega)}\\
&\le C_{m}\|\partial_t^{\alpha_0}\nabla^2 d_t\|_{m-2-|\alpha_0|}
     +C_{m}\|\partial_t^{\alpha_0}\nabla d_t\|_{m-1-|\alpha_0|}\\
&\le C_{m}\|\nabla^2 d\|_{\mathcal{H}^{m-1}}
     +C_{m}\|\nabla d\|_{\mathcal{H}^{m}}.
\end{aligned}
\end{equation}
Substituting \eqref{3417} and \eqref{3418} into \eqref{3416} and applying
the Cauchy inequality, one attains
\begin{equation}\label{3419}
\begin{aligned}
III_2
&\le \delta_1 \int_0^t \|\nabla \Delta d\|_{\mathcal{H}^{m-1}}^2 d\tau
     +\delta \int_0^t \| \Delta d\|_{\mathcal{H}^{m}}^2 d\tau\\
&\quad +C_m\left\{C_{\delta,\delta_1} \int_0^t \|\nabla^2 d\|_{\mathcal{H}^{m-1}}^2d\tau
     +C_{\delta,\delta_1} \int_0^t\|\nabla d\|_{\mathcal{H}^{m}}^2d\tau\right\}.
\end{aligned}
\end{equation}
By virtue of the Cauchy inequality, one arrives at
\begin{equation}\label{3420}
\begin{aligned}
&III_3
\le C\int_0^t \|\Delta d_t\|_{\mathcal{H}^{m-2}}\|\Delta d\|_{\mathcal{H}^{m-1}}d\tau
\le C\int_0^t \|\Delta d\|_{\mathcal{H}^{m-1}}^2d\tau,\\
&III_4
\le \delta_1 \int_0^t\|\nabla \mathcal{Z}^\alpha \Delta d\|_{L^2}^2 d\tau
+C_{\delta_1} \int_0^t \|\nabla \Delta d\|_{\mathcal{H}^{m-2}}^2 d\tau.
\end{aligned}
\end{equation}
The application of Proposition \ref{prop2.2} yields directly
\begin{equation}\label{3421}
\begin{aligned}
III_5
&=\int_0^t\int \mathcal{Z}^\alpha (\nabla u\cdot \nabla d)\cdot
  \nabla \mathcal{Z}^\alpha \Delta d ~dxd\tau
  +\int_0^t\int \mathcal{Z}^\alpha (u\cdot \nabla^2 d)\cdot
  \nabla \mathcal{Z}^\alpha \Delta d ~dxd\tau\\
&\le \delta_1 \!\!\int_0^t\!\|\nabla \mathcal{Z}^\alpha \Delta d\|_{L^2}^2 d\tau
     +C_\delta\|\nabla u\|_{L^\infty_{x,t}}^2
     \!\!\int_0^t\|\nabla d\|_{\mathcal{H}^{m-1}}^2 d\tau
     +C_\delta\|\nabla d\|_{L^\infty_{x,t}}^2
     \!\!\int_0^t\|\nabla u\|_{\mathcal{H}^{m-1}}^2 d\tau\\
&\quad  +C_\delta\|u\|_{L^\infty_{x,t}}^2
     \int_0^t\|\nabla^2 d\|_{\mathcal{H}^{m-1}}^2 d\tau
     +C_\delta\|\nabla^2 d\|_{L^\infty_{x,t}}^2
     \int_0^t\|u\|_{\mathcal{H}^{m-1}}^2 d\tau\\
&\le \delta_1 \int_0^t\|\nabla \mathcal{Z}^\alpha \Delta d\|_{L^2}^2 d\tau
     +C_{\delta_1}(1+P(Q(t)))\int_0^t \Lambda_m(\tau)d\tau.
\end{aligned}
\end{equation}
It is easy to deduce that
\begin{equation}\label{3422}
\begin{aligned}
III_6
&=\sum_{|\beta|\ge 1}\int_0^t \int
   \mathcal{Z}^\gamma (\nabla d\cdot \nabla^2 d)\cdot\mathcal{Z}^\beta d
    \cdot \nabla \mathcal{Z}^\alpha \Delta d ~dxd\tau\\
&\quad  +\int_0^t \int\mathcal{Z}^\alpha (\nabla d\cdot \nabla^2 d)\cdot d \cdot
     \nabla \mathcal{Z}^\alpha \Delta d ~dxd\tau\\
&\quad +\sum_{|\beta|+|\gamma|=m-1}
     \int_0^t\int \mathcal{Z}^\gamma (|\nabla d|^2)\mathcal{Z}^\beta \nabla d
         \cdot \nabla \mathcal{Z}^\alpha \Delta d~dxd\tau\\
&=III_{61}+III_{62}+III_{63}.
\end{aligned}
\end{equation}
By virtue of the Proposition \ref{prop2.2} and Cauchy inequality, one arrives at
\begin{equation}\label{3423}
\begin{aligned}
III_{61}
&\le \delta \int_0^t \|\nabla \mathcal{Z}^\alpha \Delta d\|_{L^2}^2 d\tau
     +C_\delta\|\mathcal{Z}d\|_{L^{\infty}_{x,t}}^2
     \int_0^t \|\nabla d\cdot \nabla^2 d\|_{\mathcal{H}^{m-2}}^2 d\tau\\
&\quad +C_\delta\|\nabla d\cdot \nabla^2 d \|_{L^{\infty}_{x,t}}^2
     \int_0^t \|\mathcal{Z}d\|_{\mathcal{H}^{m-2}}^2 d\tau\\
&\le \delta_1 \int_0^t \|\nabla \mathcal{Z}^\alpha \Delta d\|_{L^2}^2 d\tau
     +C_1C_{\delta_1}(1+P(Q(t)))\int_0^t \Lambda_m(\tau)d\tau.
\end{aligned}
\end{equation}
Similarly, it is easy to deduce that
\begin{equation}\label{3424}
\begin{aligned}
&III_{62}\le \delta \int_0^t \|\nabla \mathcal{Z}^\alpha \Delta d\|_{L^2}^2 d\tau
     +C_\delta \|\nabla d\|_{W^{1,\infty}_{x,t}}^2
     \int_0^t (\|\nabla^2 d\|_{\mathcal{H}^{m-1}}^2+\|\nabla d\|_{\mathcal{H}^{m-1}}^2) d\tau,\\
&III_{63}\le \delta \int_0^t \|\nabla \mathcal{Z}^\alpha \Delta d\|_{L^2}^2 d\tau
     +C_\delta \|\nabla d\|_{L^{\infty}_{x,t}}^4
     \int_0^t (\|\nabla^2 d\|_{\mathcal{H}^{m-1}}^2+\|\nabla d\|_{\mathcal{H}^{m-1}}^2) d\tau.
\end{aligned}
\end{equation}
Substituting \eqref{3423} and \eqref{3424} into \eqref{3422}, we obtain
\begin{equation}\label{3425}
III_6
\le \delta \int_0^t \|\nabla \mathcal{Z}^\alpha \Delta d\|_{L^2}^2 d\tau
  +C_1C_\delta(1+P(Q(t)))\int_0^t \Lambda_m(\tau)d\tau.
\end{equation}
Substituting \eqref{3419}-\eqref{3421} and \eqref{3425} into \eqref{3415}
and choosing $\delta_1$ small enough, we find
\begin{equation*}
\begin{aligned}
&\frac{1}{2}\int |\mathcal{Z}^\alpha \Delta d(t)|^2dx
+\int_0^t \int |\nabla \mathcal{Z}^\alpha \Delta d|^2dx d\tau\\
&\le C_m\left\{\frac{1}{2}\int |\mathcal{Z}^\alpha \Delta d_0|^2dx
     +\delta \int_0^t \| \Delta d\|_{\mathcal{H}^{m}}^2 d\tau
     +C\int_0^t \|\nabla \Delta d\|_{\mathcal{H}^{m-2}}^2 d\tau\right\}\\
&\quad   +C_m C_\delta [1+P(Q(t))]\int_0^t \Lambda_m(\tau)d\tau.
\end{aligned}
\end{equation*}
By the induction assumption, one can eliminate the term
$\int_0^t \|\nabla \Delta d\|_{\mathcal{H}^{m-2}}^2d\tau$.
Therefore, we complete the proof of Lemma \ref{lemma3.4-1}.
\end{proof}

Next, we derive the following lower order estimates on $\|({\rm div}u, p)\|_{L^2}^2$.
\begin{lemm}\label{lemma3.8}
For every $m\in \mathbb{N}_+$, it holds that
\begin{equation}\label{381}
\begin{aligned}
&\underset{0\le \tau \le t}{\sup}\int\left(\frac{1}{2}
\rho|{\rm div}u|^2+\frac{1}{2\gamma p}|\nabla p(\tau)|^2\right)dx
+\varepsilon \int_0^t \|\nabla {\rm div}u(\tau)\|_{L^2}^2 d\tau\\
&\le \int(\frac{1}{2}
\rho_0|{\rm div}u_0|^2+\frac{1}{2\gamma p_0}|\nabla p_0|^2)dx
+C_3[1+P(Q(t))]\int_0^t(\Lambda_m(\tau)+\|\nabla \Delta d\|_{\mathcal{H}^{m-1}}^2)d\tau.
\end{aligned}
\end{equation}
\end{lemm}
\begin{proof}
Multiplying \eqref{eq1-1} by $\nabla {\rm div}u$ yields that
\begin{equation}\label{382}
\begin{aligned}
&\int_0^t \int (\rho u_t+\rho u\cdot \nabla u)\cdot \nabla {\rm div}u~dxd\tau
+\int_0 ^t \int \nabla p\cdot\nabla {\rm div}u~dxd\tau\\
&=-\mu \varepsilon \int_0^t \int \nabla \times w\cdot \nabla {\rm div}u~dxd\tau
+(2\mu+\lambda)\varepsilon\int_0^t \int |\nabla {\rm div}u|^2 dxd\tau\\
&\quad -\int_0^t \int (\nabla d \cdot \Delta d)\cdot \nabla {\rm div}u~dxd\tau
=IV_1+IV_2+IV_3+IV_4.
\end{aligned}
\end{equation}
Using the same argument as Lemma 3.5 of \cite{Wang-Xin-Yong},
one can obtain the following estimates
\begin{equation}\label{383}
IV_1\le -\frac{1}{2}\int \rho |{\rm div}u|^2 dx+\frac{1}{2}\int \rho_0 |{\rm div}u_0|^2 dx
       +C_2[1+P(Q(t))]\int_0^t \Lambda_m (\tau)d\tau,
\end{equation}
\begin{equation}\label{384}
IV_2\le -\int \frac{1}{2\gamma p} |\nabla p|^2 dx
       +\int\frac{1}{2\gamma p_0} |\nabla p_0|^2 dx
       +C_2[1+P(Q(t))]\int_0^t \|\nabla p\|_{L^2}^2 d\tau
\end{equation}
and
\begin{equation}\label{385}
IV_3\le \frac{\varepsilon}{4}\int_0^t \|\nabla {\rm div}u\|_{L^2}^2 d\tau
       +C_3 \varepsilon \int_0^t (\|u(\tau)\|_{L^2}^2+\|\nabla u(\tau)\|_{L^2}^2)d\tau.
\end{equation}
Integrating by part and applying the boundary condition \eqref{bc2}, we find
\begin{equation}\label{386}
\begin{aligned}
IV_4
&=-\int_0^t \int_{\partial \Omega} n\cdot (\nabla d \cdot \Delta d) {\rm div}u~d\sigma d\tau
+\int_0^t \int \nabla  (\nabla d \cdot \Delta d) {\rm div}u~dx d\tau\\
&\le C(\|\nabla d\|_{L^\infty}^2+\|\Delta d\|_{L^\infty})
     \int_0^t(\|\nabla u\|_{L^2}^2+\|\Delta d\|_{L^2}^2+\|\nabla \Delta d\|_{L^2}^2)d\tau.
\end{aligned}
\end{equation}
Substituting \eqref{383}-\eqref{386} into \eqref{382}, we find
\begin{equation*}
\begin{aligned}
&\int\left(\frac{1}{2}\rho |{\rm div}u|^2+\frac{1}{2\gamma p}|\nabla p|^2\right)dx
 +\varepsilon\int_0^t \|\nabla {\rm div}u(\tau)\|_{L^2}^2 d\tau\\
&\le \int\left(\frac{1}{2}\rho_0 |{\rm div}u_0|^2
      +\frac{1}{2\gamma p_0}|\nabla p_0|^2\right)dx
      +C_3[1+P(Q(t))] \int_0^t (\Lambda_m(\tau)+\|\nabla \Delta d\|_{L^2}^2)d\tau.
\end{aligned}
\end{equation*}
Therefore, we complete the proof of Lemma \ref{lemma3.8}.
\end{proof}

\begin{lemm}\label{lemma3.9}
For $m \ge 1$ and $|\alpha|\le m-1$ with $|\alpha_0|\le m-2$, it holds that
\begin{equation}\label{391}
\begin{aligned}
&\underset{0\le \tau \le t}{\sup}\int \left(\rho |\mathcal{Z}^\alpha {\rm div}u(\tau)|^2+\frac{1}{\gamma p}|\mathcal{Z}^\alpha \nabla p(\tau)|^2\right)dx
+\varepsilon\int_0^t \|\nabla \mathcal{Z}^\alpha {\rm div}u\|_{L^2}^2 d\tau\\
&\le C\int\left(\rho_0 |\mathcal{Z}^\alpha {\rm div}u_0|^2
+\frac{1}{\gamma p_0}|\mathcal{Z}^\alpha \nabla p_0|^2\right)dx
+CC_{m+2}\delta\int_0^t \|\nabla \partial_t^{m-1}p(\tau)\|_{L^2}^2 d\tau\\
&\quad \quad
+CC_{m+2}\left\{\varepsilon\int_0^t\|\nabla^2 u\|_{\mathcal{H}^{m-2}}^2 d\tau
+(\delta +\varepsilon)\int_0^t\ \|\nabla\mathcal{Z}^{m-2}{\rm div}u\|_{L^2}^2 d\tau\right\}.\\
&\quad \quad
+C_\delta C_{m+2}[1+P(Q(t))]\int_0^t (\Lambda_m+\|\nabla \Delta d\|_{\mathcal{H}^{m-1}}^2)d\tau.
\end{aligned}
\end{equation}
\end{lemm}
\begin{proof}
The case for $|\alpha|=0$ is already proved in Lemma \ref{lemma3.8}. Assuming it is
proved for $|\alpha|\le m-2$, one needs to prove it for $|\alpha|=m-1$
with $|\alpha_0|\le m-2$. Multiplying \eqref{eq1-1} by
$\nabla \mathcal{Z}^\alpha {\rm div}u$ leads to
\begin{equation}\label{392}
\begin{aligned}
&\underset{V_1}{\underbrace{\int_0^t \int(\rho \mathcal{Z}^\alpha u_t+\rho u\cdot \nabla \mathcal{Z}^\alpha u)\cdot \nabla \mathcal{Z}^\alpha {\rm div}u~dxd\tau}}
+\underset{V_2}{\underbrace{\int_0^t\int \mathcal{Z}^\alpha \nabla p \cdot \nabla \mathcal{Z}^\alpha {\rm div}u~ dxd\tau}}\\
&=\underset{V_3}{\underbrace{-\mu \varepsilon \int_0^t \int \mathcal{Z}^\alpha \nabla \times w\cdot\nabla \mathcal{Z}^\alpha {\rm div}u~dxd\tau}}
+\underset{V_4}{\underbrace{(2\mu+\lambda)\varepsilon \int_0^t \int \mathcal{Z}^\alpha \nabla {\rm div}u\cdot\nabla \mathcal{Z}^\alpha {\rm div}u~dxd\tau}}\\
&\quad \underset{V_5}{\underbrace{-\int_0^t \int \mathcal{Z}^\alpha(\nabla d\cdot \Delta d)\cdot\nabla \mathcal{Z}^\alpha {\rm div}u ~dxd\tau}}
+\underset{V_6}{\underbrace{\int_0^t \int (\mathcal{C}_1^\alpha+\mathcal{C}_2^\alpha)\cdot
      \nabla \mathcal{Z}^\alpha {\rm div}u~dx\tau}}.
\end{aligned}
\end{equation}
Using the same argument as Lemma 3.6 of \cite{Wang-Xin-Yong},
one can obtain the following estimates
\begin{equation}\label{393}
\begin{aligned}
V_1&\le  -\int\frac{\rho}{2}|\mathcal{Z}^\alpha {\rm div}u|^2 dx
     +\int \frac{\rho_0}{2}|\mathcal{Z}^\alpha {\rm div}u_0|^2 dx
     +\delta \int_0^t\!\! \|\nabla \mathcal{Z}^{\alpha-2}{\rm div}u\|_{L^2}^2 d\tau\\
     &\quad \quad +C_\delta[1+P(Q(t))]\int_0^t \Lambda_m (\tau)d\tau,
\end{aligned}
\end{equation}
\begin{equation}\label{394}
V_4\ge \frac{3(2\mu+\lambda)\varepsilon}{4}\int_0^t
        \|\nabla\mathcal{Z}^\alpha{\rm div}u \|_{L^2}^2 d\tau
        -C\varepsilon\int_0^t  \|\nabla\mathcal{Z}^{m-2}{\rm div}u \|_{L^2}^2 d\tau
        -C\varepsilon \int_0^t \Lambda_m(\tau)d\tau,
\end{equation}
\begin{equation}\label{395}
V_3\ge -\frac{(2\mu+\lambda)\varepsilon}{4}\int_0^t
         \|\nabla\mathcal{Z}^\alpha{\rm div}u \|_{L^2}^2 d\tau
         -C\varepsilon\int_0^t \|\nabla^2 u\|_{\mathcal{H}^{m-2}}^2 d\tau
         -C_{m+2}\int_0^t P(\Lambda_m(\tau))d\tau,
\end{equation}
\begin{equation}\label{396}
\begin{aligned}
V_2\le &-\int \frac{1}{2\gamma p}|\mathcal{Z}^\alpha \nabla p|^2 dx
       +\int \frac{1}{2\gamma p_0}|\mathcal{Z}^\alpha \nabla p_0|^2 dx
       +C\delta \int_0^t \|\nabla \mathcal{Z}^{m-2}{\rm div}u\|_{L^2}^2 d\tau\\
       &+C\delta \int_0^t \|\nabla \partial_t^{m-1} p\|_{L^2}^2 d\tau
       +C_\delta C_{m+1} (1+P(Q(t)))\int_0^t \Lambda_m(\tau)d\tau,
\end{aligned}
\end{equation}
and
\begin{equation}\label{397}
V_6\le \delta \int_0^t \|\nabla \mathcal{Z}^{m-2}{\rm div}u\|_{L^2}^2 d\tau
       +C_\delta [1+P(Q(t))]\int_0^t \Lambda_m(\tau)d\tau.
\end{equation}
On the other hand, the integration by parts yields directly
\begin{equation}\label{398}
V_5
=-\int_0^t\int_{\partial \Omega}n\cdot \mathcal{Z}^{\alpha}
  (\nabla d\cdot \Delta d)\cdot \mathcal{Z}^{\alpha} {\rm div}u~d\sigma d\tau
  +\int_0^t \int {\rm div}\mathcal{Z}^{\alpha}(\nabla d\cdot \Delta d)
  \cdot \mathcal{Z}^{\alpha} {\rm div}u~ dxd\tau.
\end{equation}
In view of the trace inequality in Proposition \ref{prop2.3}, we find
\begin{equation}\label{399}
|Z_y^{m-2-\alpha_0}Z_t^{\alpha_0}{\rm div}u|_{L^2{(\partial \Omega)}}^2
\le  C\|\nabla Z_y^{m-2-\alpha_0}Z_t^{\alpha_0}{\rm div}u\|_{L^2}^2+
     C\|Z_y^{m-2-\alpha_0}Z_t^{\alpha_0}{\rm div}u\|_{L^2}^2,
\end{equation}
and
\begin{equation}\label{3910}
|Z_y(n\cdot Z_y^{m-1-\alpha_0}Z_t^{\alpha_0})(\nabla d\cdot \Delta d)
|_{L^2(\partial \Omega)}^2
\le C \|\nabla Z_t^{\alpha_0}(\nabla d\cdot \Delta d)\|_{H^{|\alpha_1|}_{co}}^2
       +C\|Z_t^{\alpha_0}(\nabla d\cdot \Delta d)\|_{H^{|\alpha_1|}_{co}}^2.
\end{equation}
Integrating by parts along the boundary and applying the estimates
\eqref{399} and \eqref{3910}, one attains
\begin{equation}\label{3911}
\begin{aligned}
&-\int_0^t\int_{\partial \Omega}n\cdot \mathcal{Z}^{\alpha}
  (\nabla d\cdot \Delta d)\cdot \mathcal{Z}^{\alpha} {\rm div}u~d\sigma d\tau\\
&\le \int_0^t |Z_y(n\cdot Z_y^{m-1-\alpha_0}Z_t^{\alpha_0})
       (\nabla d\cdot \Delta d)|_{L^2(\partial \Omega)}
       |Z_y^{m-2-\alpha_0}Z_t^{\alpha_0}{\rm div}u|_{L^2{(\partial \Omega)}}d\tau\\
&\le  \delta \int_0^t \|\nabla \mathcal{Z}^{m-2}{\rm div}u\|_{L^2}^2d\tau
     +C_\delta \int_0^t \|\mathcal{Z}^{m-2}{\rm div}u\|_{L^2}^2d\tau
     +C_\delta\|\nabla^2 d\|_{L^\infty}^2\int_0^t \|\Delta d\|_{\mathcal{H}^{m-1}}^2 d\tau\\
&\quad +C_\delta\|\nabla d\|_{L^\infty}^2\int_0^t \|\nabla \Delta d\|_{\mathcal{H}^{m-1}}^2 d\tau
       +C_\delta\|\nabla \Delta d\|_{L^\infty}^2\int_0^t \|\nabla d\|_{\mathcal{H}^{m-1}}^2 d\tau\\
&\quad +C_\delta\|\nabla d\|_{L^\infty}^2\int_0^t \|\Delta d\|_{\mathcal{H}^{m-1}}^2 d\tau
       +C_\delta\|\Delta d\|_{L^\infty}^2\int_0^t \|\nabla d\|_{\mathcal{H}^{m-1}}^2 d\tau\\
&\le \delta \int_0^t \|\nabla \mathcal{Z}^{m-2}{\rm div}u\|_{L^2}^2d\tau
     +C_\delta[1+P(Q(t))]\int_0^t (\Lambda_m+\|\nabla \Delta d\|_{\mathcal{H}^{m-1}}^2)d\tau.
\end{aligned}
\end{equation}
Applying the Proposition \ref{prop2.2}, it is easy to check that
\begin{equation}\label{3912}
\int_0^t \int {\rm div}\mathcal{Z}^{\alpha}(\nabla d\cdot \Delta d)
  \cdot \mathcal{Z}^{\alpha} {\rm div}u~ dxd\tau
\le C [1+P(Q(t))]\int_0^t (\Lambda_m+\|\nabla \Delta d\|_{\mathcal{H}^{m-1}}^2)d\tau.
\end{equation}
Substituting \eqref{3911} and \eqref{3912} into \eqref{3910}, we obtain
\begin{equation}\label{3913}
V_5\le \delta \int_0^t \|\nabla \mathcal{Z}^{m-2}{\rm div}u\|_{L^2}^2d\tau
     +C_\delta[1+P(Q(t))]\int_0^t(\Lambda_m+\|\nabla \Delta d\|_{\mathcal{H}^{m-1}}^2)d\tau
\end{equation}
Substituting \eqref{393}-\eqref{397} and \eqref{3913}
into \eqref{392}, we complete the proof of Lemma \ref{lemma3.9}.
\end{proof}

\begin{lemm}\label{lemma3.10}
For $m \ge 1$, it holds that
\begin{equation}\label{3101}
\begin{aligned}
&\underset{0\le \tau \le t}{\sup}\varepsilon \int
\left(\rho |\partial_t^{m-1}{\rm div}u(\tau)|^2
      +\frac{1}{\gamma p}|\partial_t^{m-1} \nabla p(\tau)|^2 \right)dx
+\varepsilon^2 \int_0^t \|\nabla \partial_t^{m-1} {\rm div}u(\tau)\|_{L^2}^2 d\tau\\
&\le C\varepsilon \int \left(\rho_0 |\partial_t^{m-1}{\rm div}u_0|^2
      +\frac{1}{\gamma p_0}|\partial_t^{m-1} \nabla p_0|^2 \right)dx\\
&\quad \quad
      +C_{m+1}[1+P(Q(t))]\int_0^t(\Lambda_m (\tau)
      + \|\nabla \Delta d(\tau)\|_{\mathcal{H}^{m-1}}^2)d\tau.
\end{aligned}
\end{equation}
\end{lemm}
\begin{proof}
Applying $\partial_t^{m-1}$ to equation \eqref{eq1-1}, we find that
\begin{equation}\label{3102}
\begin{aligned}
&\rho \partial_t^{m-1} u_t+\rho u\cdot \nabla \partial_t^{m-1} u
+\mu \varepsilon \nabla \times \partial_t^{m-1}(\nabla \times u)
+ \partial_t^{m-1}\nabla p\\
&=(2\mu+\lambda)\varepsilon \nabla \partial_t^{m-1}{\rm div}u
-\partial_t^{m-1}(\nabla d\cdot \Delta d)
+\mathcal{C}_1^{m-1}+\mathcal{C}_2^{m-1},
\end{aligned}
\end{equation}
where
\begin{equation*}
\mathcal{C}_1^{m-1}\triangleq-[\partial_t^{m-1}, \rho]u_t,\quad
\mathcal{C}_2^{m-1}\triangleq-[\partial_t^{m-1}, \rho u\cdot \nabla]u.
\end{equation*}
Multiplying \eqref{3102} by $\varepsilon \nabla {\rm div}\partial_t^{m-1}u$,
it is easy to deduce that
\begin{equation}\label{3103}
\begin{aligned}
&\underset{VI_1}{\underbrace{\varepsilon\int_0^t\!\!\int \!\!(\rho \partial_t^{m-1} u_t+\rho u\cdot \nabla \partial_t^{m-1} u)
\cdot \nabla {\rm div}\partial_t^{m-1}u~dxd\tau}}
+\underset{VI_2}{\underbrace{\varepsilon\int_0^t \!\!\int\!\!  \partial_t^{m-1}\nabla p
\cdot \nabla {\rm div}\partial_t^{m-1}u~dxd\tau}}\\
&=\underset{VI_3}{\underbrace{-\mu \varepsilon^2
\int_0^t \!\! \int \nabla \times \partial_t^{m-1}(\nabla \times u)
\cdot \nabla {\rm div}\partial_t^{m-1}u~dxd\tau}}
+\underset{VI_4}{\underbrace{(2\mu+\lambda)\varepsilon^2 \int_0^t \!\!
\int |\nabla \partial_t^{m-1}{\rm div}u|^2 dxd\tau}}\\
&~
+\underset{VI_5}{\underbrace{\varepsilon \!\int_0^t \!\!\int (\mathcal{C}_1^{m-1}+\mathcal{C}_2^{m-1})
\cdot \nabla {\rm div}\partial_t^{m-1}u~dxd\tau}}
-\underset{VI_6}{\underbrace{\varepsilon \!\int_0^t \!\!
\int \partial_t^{m-1}(\nabla d\cdot \Delta d)
\cdot \nabla {\rm div}\partial_t^{m-1}u~dxd\tau}}.
\end{aligned}
\end{equation}
Using the same argument as Lemma 3.8 of \cite{Wang-Xin-Yong},
one can obtain the following estimates
\begin{equation}\label{3104}
|VI_3|\le \frac{2\mu+\lambda}{8} \varepsilon^4
\int_0^t \|\nabla \partial_t^{m-1}{\rm div}u\|_{L^2}^2 d\tau
+CC_3\int_0^t \Lambda_m(\tau)d\tau.
\end{equation}
\begin{equation}\label{3105}
\begin{aligned}
VI_1\le &-\varepsilon\int \rho |\partial_t^{m-1}{\rm div}u(t)|^2 dx
       +\varepsilon\int \rho_0 |\partial_t^{m-1}{\rm div}u_0|^2 dx\\
&+\frac{\varepsilon^2}{8}\int_0^t \|\nabla \partial_t^{m-1}{\rm div}u\|_{L^2}^2 d\tau
       +C[1+P(Q(t))]\int_0^t \Lambda_m(\tau)d\tau,
\end{aligned}
\end{equation}
\begin{equation}\label{3106}
|VI_5|\le \frac{2\mu+\lambda}{8} \varepsilon^2
\int_0^t \|\nabla \partial_t^{m-1}{\rm div}u\|_{L^2}^2 d\tau
+C[1+P(Q(t))]\int_0^t \Lambda_m(\tau)d\tau.
\end{equation}
and
\begin{equation}\label{3107}
VI_2\le -\varepsilon \int \frac{1}{2\gamma p}|\nabla \partial_t^{m-1} p|^2 dx
       +\varepsilon \int \frac{1}{2\gamma p_0}|\nabla \partial_t^{m-1} p_0|^2 dx
       +C_{m+1}[1+P(Q(t))]\int_0^t \Lambda_m(\tau)d\tau.
\end{equation}
On the other hand, integrating by parts and applying the boundary condition \eqref{bc2},
one attains
\begin{equation}\label{3108}
\begin{aligned}
VI_6
&=-\varepsilon\int_0^t \int n\cdot \partial_t^{m-1}(\nabla d\cdot \Delta d)
 \cdot {\rm div}\partial_t^{m-1}u~d\sigma d\tau\\
&\quad +\varepsilon \int_0^t \int {\rm div} \partial_t^{m-1}(\nabla d\cdot \Delta d)
   \cdot {\rm div}\partial_t^{m-1}u~dx d\tau\\
&\le C[1+P(Q(t))]\int_0^t (\Lambda_m(\tau)+\|\nabla \Delta d\|_{\mathcal{H}^{m-1}}^2)\tau.
\end{aligned}
\end{equation}
Substituting \eqref{3104}-\eqref{3108} into \eqref{3103},
we complete the proof of Lemma \ref{lemma3.10}.
\end{proof}

Next, we recall an important estimate that has been proved by Wang et al. \cite{Wang-Xin-Yong}.

\begin{lemm}\label{lemma3.11}
Define
\begin{equation}\label{31101}
\Lambda_{1m}(t)\triangleq\|(p, u, \nabla d)(t)\|_{\mathcal{H}^{m}}^2
             +\|\Delta d(t)\|_{\mathcal{H}^{m-1}}^2
             +\sum_{|\beta|\le m-2}\|\mathcal{Z}^\beta \nabla p(t)\|_1^2
             +\sum_{|\beta|\le m-2}\|\mathcal{Z}^\beta \nabla u(t)\|_1^2.
\end{equation}
Then, for every $m \ge 3$, it holds that
\begin{equation}\label{31102}
\|\partial_t^{m-1}{\rm div}u(t)\|_{L^2}^2
\le C_2\left\{P(\Lambda_{1m}(t))+P(Q(t))\right\}.
\end{equation}
\end{lemm}

\begin{lemm}\label{lemma3.12}
For every $m \ge 1$, it holds that
\begin{equation}\label{31201}
\int_0^t \|\nabla \partial_t^{m-1}p(\tau)\|_{L^2}^2 d\tau
\le C\varepsilon^2 \int_0^t \|\nabla^2 \partial_t^{m-1}u(\tau)\|_{L^2}^2 d\tau
+C[1+P(Q(t))]\int_0^t \Lambda_m(\tau)d\tau.
\end{equation}
\end{lemm}
\begin{proof}
Applying $\partial_t^{m-1}$ to \eqref{eq1-1}, we find
\begin{equation*}
\partial_t^{m-1}\nabla p
=\partial_t^{m-1}(-\rho u_t-\rho u\cdot \nabla u
-\mu \varepsilon \nabla \times (\nabla \times u)
+(2\mu+\lambda)\varepsilon \nabla {\rm div}u-\nabla d\cdot \Delta d).
\end{equation*}
By using the Proposition \ref{prop2.2}, it is easy to deduce the estimate
\eqref{31201}. Hence, we complete the proof of Lemma \ref{lemma3.12}.
\end{proof}

Next, we recall an important estimate that has been proved by Wang et al. \cite{Wang-Xin-Yong}.

\begin{lemm}\label{lemma3.13}
For every $m \ge 1$, it holds that
\begin{equation}\label{31301}
\int_0^t \|\nabla \mathcal{Z}^{m-2}{\rm div}u(\tau)\|_{L^2}^2 d\tau
\le C\int_0^t \|\nabla \partial_t^{m-1}p(\tau)\|_{L^2}^2 d\tau
+C_m[1+P(Q(t))]\int_0^t \Lambda_m(\tau)d\tau.
\end{equation}
\end{lemm}

Finally, we estimate the  estimate for the quantity $\nabla \Delta d$.

\begin{lemm}\label{lemma3.14}
For every $m \ge 1$, it holds that
\begin{equation}\label{31401}
\int_0^t \|\nabla \Delta d\|_{\mathcal{H}^{m-1}}^2 d\tau
\le C[1+P(Q(t))]\int_0^t \Lambda_m(\tau)d\tau.
\end{equation}
\end{lemm}
\begin{proof}
By virtue of the equation \eqref{eq1}$_3$, it is easy to deduce that
\begin{equation*}
\nabla \Delta d=\nabla d_t+\nabla(u\cdot \nabla d)-\nabla(|\nabla d|^2 d).
\end{equation*}
By using the Proposition \ref{prop2.2}, it is easy to deduce the estimate
\eqref{31401}. Then, we complete the proof of Lemma \ref{lemma3.14}.
\end{proof}

Substituting \eqref{31201}, \eqref{31301} and \eqref{31401}  
into \eqref{391}, it is easy to deduce that
\begin{equation}\label{3a1}
\begin{aligned}
&\underset{0\le \tau \le t}{\sup}\sum_{k=0}^{m-2}
\|(\partial_t^k \nabla p, \partial_t^k {\rm div}u)(\tau)\|_{m-1-k}^2
+\varepsilon \int_0^t \sum_{k=0}^{m-2}\|\partial_t^k \nabla {\rm div}u(\tau)\|_{m-1-k}^2 d\tau\\
&\le CC_{m+2}\left\{\Lambda_m(0)
+C\delta \varepsilon^2 \int_0^t \|\nabla^2 \partial_t^{m-1}u(\tau)\|_{L^2}^2d\tau
+\varepsilon\int_0^t \|\nabla^2 u(\tau)\|_{\mathcal{H}^{m-2}}^2d\tau\right\}\\
&\quad \quad +C_\delta C_{m+2}[1+P(Q(t))]\int_0^t \Lambda_m(\tau)d\tau.
\end{aligned}
\end{equation}
Substituting \eqref{31401} into \eqref{3101}, we obtain
\begin{equation}\label{3a2}
\begin{aligned}
&\underset{0\le \tau \le t}{\sup}\varepsilon
\left(\|\partial_t^{m-1}{\rm div}u(\tau)\|_{L^2}^2
      +\|\partial_t^{m-1} \nabla p(\tau)\|_{L^2}^2 \right)
+\varepsilon^2 \int_0^t \|\nabla \partial_t^{m-1} {\rm div}u(\tau)\|_{L^2}^2 d\tau\\
&\le C\varepsilon \Lambda_m(0)+C C_{m+1}[1+P(Q(t))]\int_0^t \Lambda_m (\tau)d\tau.
\end{aligned}
\end{equation}
Then, the combination of \eqref{3a1} and \eqref{3a2} yields directly
\begin{equation}\label{3a3}
\begin{aligned}
&\underset{0\le \tau \le t}{\sup}\sum_{k=0}^{m-2}
\left\{\|(\partial_t^k \nabla p, \partial_t^k {\rm div}u)(\tau)\|_{m-1-k}^2
+\varepsilon\|(\partial_t^{m-1}{\rm div}u,\partial_t^{m-1} \nabla p)(\tau)\|_{L^2}^2
\right\}\\
&+\varepsilon \int_0^t \sum_{k=0}^{m-2}\|\partial_t^k \nabla {\rm div}u(\tau)\|_{m-1-k}^2 d\tau
+\varepsilon^2 \int_0^t \|\nabla \partial_t^{m-1} {\rm div}u(\tau)\|_{L^2}^2 d\tau\\
&\le CC_{m+2}\left\{\Lambda_m(0)
+C\delta \varepsilon^2 \int_0^t \|\nabla^2 \partial_t^{m-1}u(\tau)\|_{L^2}^2d\tau
+\varepsilon\int_0^t \|\nabla^2 u(\tau)\|_{\mathcal{H}^{m-2}}^2d\tau\right\}\\
&\quad \quad +C_\delta C_{m+2}[1+P(Q(t))]\int_0^t \Lambda_m(\tau)d\tau.
\end{aligned}
\end{equation}
On the other hand, it is easy to check that
\begin{equation}\label{3a4}
\sum_{|\beta|\le m-2}\|\mathcal{Z}^\beta \nabla u\|_{1}^2
\le C_{m+1}(\|u\|_{\mathcal{H}^m}^2+\|\eta\|_{\mathcal{H}^{m-1}}^2
+\sum_{k=0}^{m-2}\|\partial_t^k {\rm div}u\|_{m-1-k}^2),
\end{equation}
\begin{equation}\label{3a5}
\int_0^t \|\nabla^2 u\|_{\mathcal{H}^{m-1}}^2 d\tau
\le C_{m+2}\int_0^t (\|\nabla u\|_{\mathcal{H}^m}^2
+\|\nabla \eta\|_{\mathcal{H}^{m-1}}^2
+\|\nabla {\rm div}u\|_{\mathcal{H}^{m-1}}^2+\Lambda_m)d\tau,
\end{equation}
\begin{equation}\label{3a6}
\varepsilon \int_0^t \|\nabla^2 \mathcal{Z}^{m-2}u\|_{L^2}^2 d\tau
\le C_{m+1}\varepsilon\int_0^t \|\nabla \eta\|_{\mathcal{H}^{m-2}}^2d\tau
+C_{m+1}\int_0^t \Lambda_m(\tau)d\tau,
\end{equation}
and
\begin{equation}\label{3a7}
\begin{aligned}
&\sum_{k=0}^{m-2}\int_0^t \|\nabla^2 \partial_t^k u\|_{m-1-k}^2d\tau\\
&\le C_{m+2}\int_0^t (\|\nabla u\|_{\mathcal{H}^m}^2
+\|\nabla \eta\|_{\mathcal{H}^{m-1}}^2)d\tau
+C_{m+2}\int_0^t \Lambda_m(\tau)d\tau\\
&\quad +C_{m+2}\sum_{k=0}^{m-2}\int_0^t \|\partial_t^k \nabla {\rm div}u\|_{m-1-k}^2d\tau.\\
\end{aligned}
\end{equation}
The combination of \eqref{3a3}-\eqref{3a7} yields directly
\begin{equation}\label{3a8}
\begin{aligned}
&\underset{0\le \tau \le t}{\sup}
\left\{\Lambda_{1m}(\tau)+\|(\eta, \Delta d)(\tau)\|_{\mathcal{H}^{m-1}}^2
+\varepsilon \|(\partial_t^{m-1} {\rm div}u, \partial_t^{m-1} \nabla p)(\tau)\|_{L^2}^2\right\}\\
&+\varepsilon \int_0^t(\|\nabla u\|_{\mathcal{H}^{m}}^2+
\|\nabla \eta\|_{\mathcal{H}^{m-1}}^2) d\tau
+\int_0^t (\|\Delta d\|_{\mathcal{H}^{m}}^2
+\|\nabla \Delta d\|_{\mathcal{H}^{m-1}}^2) d\tau\\
&+\varepsilon \int_0^t \sum_{k=0}^{m-2}\|\partial_t^k \nabla {\rm div}u\|_{m-1-k}^2d\tau
+\varepsilon^2 \int_0^t \|\partial_t^{m-1}\nabla {\rm div}u\|_{L^2}^2 d\tau\\
&+\varepsilon \int_0^t \sum_{k=0}^{m-2}\|\partial_t^k \nabla^2 u\|_{m-1-k}^2d\tau
+\varepsilon^2 \int_0^t \|\partial_t^{m-1}\nabla^2 u\|_{L^2}^2 d\tau
+\int_0^t \|\partial_t^{m-1}\nabla p\|_{L^2}^2 d\tau\\
&\le CC_{m+2}\left\{\Lambda_m(0)+[1+P(Q(t))]\int_0^t P(\Lambda_m(\tau))d\tau\right\}.
\end{aligned}
\end{equation}
\subsection{$L^\infty-estimates$}

\quad In this subsection, we shall provide the $L^\infty-$estimates
of $(p, u, d)$ which are needed to estimate on the right-hand side of
the estimate \eqref{3a8}.
\begin{lemm}\label{lemma3.15}
For a smooth solution $(p, u, d)$ to \eqref{eq1} and \eqref{bc2}, it holds that
\begin{equation}\label{31501}
\|\mathcal{Z}^\alpha ({\rm ln}\rho, p, u)\|_{L^\infty}^2
\le CP(\Lambda_{1m}(t)),\quad m \ge 2+|\alpha|,
\end{equation}
\begin{equation}\label{31502}
\|\nabla({\rm ln}\rho, p)\|_{\mathcal{H}^{1,\infty}}^2
\le C_3 \left(P(\|\Delta p\|_{\mathcal{H}^1}^2)
+P(\Lambda_{1m}(t))\right),\quad m \ge 5,
\end{equation}
\begin{equation}\label{31503}
\|{\rm div}u(t)\|_{\mathcal{H}^{1,\infty}}^2
\le C_3 \left(P(\|\Delta p\|_{\mathcal{H}^1}^2)
+P(\Lambda_{1m}(t))\right),\quad m \ge 5,
\end{equation}
\begin{equation}\label{31504}
\|\nabla {\rm div}u(t)\|_{L^\infty}^2\le C_3 P(Q(t)),
\end{equation}
\begin{equation}\label{31505}
\|\nabla {\rm div}u(t)\|_{\mathcal{H}^{1,\infty}}^2
\le C_4 [1+P(Q(t))]\left(\delta \|\Delta p\|_{\mathcal{H}^2}^2
+C_\delta P(\Lambda_{1m})\right),
\end{equation}
\begin{equation}\label{31507}
\|d_t\|_{W^{1,\infty}}^2
+\|\nabla d\|_{\mathcal{H}^{1,\infty}}^2
+\|\nabla^2 d\|_{\mathcal{H}^{1,\infty}}^2
+\|\nabla \Delta d\|_{\mathcal{H}^{1,\infty}}^2
\le C_{m+2} P (\Lambda_{m}(t)),\quad m\ge 3.
\end{equation}
\end{lemm}
\begin{proof}
The estimates \eqref{31501}-\eqref{31505} have been
proven by Wang et al.\cite{Wang-Xin-Yong}(see Lemma 3.14).
Hence, we give the proof for the estimate \eqref{31507}.
By virtue of the Sobolev inequality in Proposition \ref{prop2.3}, one arrives at
\begin{equation}\label{31508}
\|\nabla d\|_{L^\infty}^2\le C(\|\nabla^2 d\|_{1}^2+\|\nabla d\|_2^2).
\end{equation}
In view of the standard elliptic regularity results with Neumann boundary condition,
we get that
\begin{equation}\label{31509}
\|\nabla^2 d\|_{m}^2\le C_{m+2}(\|\Delta d\|_m^2+\|\nabla d\|_{L^2}^2).
\end{equation}
Then, the combination of \eqref{31508} and \eqref{31509} yields directly
\begin{equation}\label{315010}
\|\nabla d\|_{L^\infty}^2\le C_3(\|\Delta d\|_{1}^2+\|\nabla d\|_2^2).
\end{equation}
For $|\alpha|=1$, the application of Proposition \ref{prop2.3} gives
for $m \ge 3$
\begin{equation*}
\|\mathcal{Z}^\alpha \nabla d\|_{L^\infty}^2
\le C(\|\nabla(\mathcal{Z}^\alpha \nabla d)\|_1
+\|\mathcal{Z}^\alpha \nabla d\|_1)\|\mathcal{Z}^\alpha \nabla d\|_{2}
\le C_{m+2} P (\Lambda_{m}(t)),
\end{equation*}
which, together with \eqref{315010}, yields
\begin{equation}\label{315010-1}
\|\nabla d\|_{\mathcal{H}^{1,\infty}}^2
\le C_{m+2} P (\Lambda_{m}(t)),\quad {\rm for~} m \ge 3.
\end{equation}
By virtue of the equation \eqref{eq1}$_3$, we find
\begin{equation}\label{315011}
\begin{aligned}
\|d_t\|_{L^\infty}^2
&\le C(\|\nabla d_t\|_1^2+\|d_t\|_2^2)\\
&\le C(\|\nabla d_t\|_1^2+\|\Delta d\|_2^2+\|u\cdot \nabla d\|_2^2+\||\nabla d|^2 d\|_2^2).
\end{aligned}
\end{equation}
By Proposition \ref{prop2.2}, \eqref{31501} and \eqref{315010}, one attains
\begin{equation}\label{315012}
\|u\cdot \nabla d\|_2^2
\le C(\|u\|_{L^\infty}^2 \|\nabla d\|_2^2+\|\nabla d\|_{L^\infty}^2 \|u\|_2^2)
\le C_3 \Lambda_{m}(t), ~{\rm for}~m \ge 2;
\end{equation}
and
\begin{equation}\label{315013}
\begin{aligned}
\||\nabla d|^2 d\|_2^2
&\le \sum_{|\gamma|\ge 1,|\beta|+|\gamma|\le 2}
    \int |{Z}^{\beta}(|\nabla d|^2){Z}^\gamma d|^2 dx
     +\||\nabla d|^2\|_2^2\\
&\le \|Z d\|_{L^\infty}^2 \||\nabla d|^2\|_1^2
     +\||\nabla d|^2\|_{L^\infty}^2\|Z d\|_1^2
     +\|\nabla d\|_{L^\infty}^2\|\nabla d\|_2^2\\
&\le C_3 \Lambda_{m}^3(t),~{\rm for}~m \ge 3.
\end{aligned}
\end{equation}
Then the combination of \eqref{315012} and \eqref{315013} gives directly
\begin{equation}\label{315014}
\|d_t\|_{L^\infty}^2 \le C_3 P(\Lambda_{m}(t)),~{\rm for}~m \ge 3.
\end{equation}
By virtue of Proposition \ref{prop2.3}, we obtain for $m \ge 3$
\begin{equation*}
\|\nabla d_t\|_{L^\infty}^2
\le C(\|\nabla^2 d_t\|_1^2+\|\nabla d_t\|_2^2)
\le C(\|\Delta d_t\|_1^2+\|\nabla d_t\|_2^2)
\le C(\|\Delta d\|_{\mathcal{H}^{2}}^2+\|\nabla d\|_{\mathcal{H}^{3}}^2),
\end{equation*}
which, together with \eqref{315014}, yields immediately
\begin{equation}\label{315015}
\|d_t\|_{W^{1,\infty}}^2\le C_3 P(\Lambda_{m}(t)),~{\rm for}~m \ge 3.
\end{equation}
On the other hand, it is easy to check that
\begin{equation*}
\begin{aligned}
&\partial_{ii}=\partial_{y_i}^2-\partial_{y_i}(\partial_i \psi \partial_z)
                -\partial_i \psi \partial_z \partial_{y_i}
                +(\partial_i \psi)^2 \partial_z^2, \quad i=1,2,\\
&\partial_1 \partial_2=\partial_{y_1}\partial_{y_2}
               -\partial_{y_2}(\partial_1 \psi \partial_z)
               -\partial_2 \psi \partial_{y_1}\partial_z
               +\partial_2 \psi \partial_1 \psi \partial_z^2,\\
&\partial_i \partial_3=\partial_{y_i}\partial_z-\partial_i \psi \partial_z^2,\quad i=1,2.
\end{aligned}
\end{equation*}
Then, we find that
\begin{equation}\label{315016}
\Delta=(1+|\nabla \psi|^2)\partial_{z}^2
        +\sum_{i=1,2}(\partial_{y_i}^2
        -\partial_{y_i}(\partial_i \psi \partial_z)
        -\partial_i \psi \partial_z \partial_{y_i}).
\end{equation}
and
\begin{equation}\label{315017}
\begin{aligned}
\nabla^2
&=[(1+|\nabla \psi|^2)+\partial_2 \psi \partial_1 \psi
        -\partial_1 \psi-\partial_2 \psi]\partial_{z}^2
        +\partial_{y_1}\partial_{y_2}\\
&\quad  +\sum_{i=1,2}(\partial_{y_i}^2
        -\partial_{y_i}(\partial_i \psi \partial_z)
        -\partial_i \psi \partial_z \partial_{y_i})
        -\partial_{y_2}(\partial_1 \psi \partial_z)\\
&\quad  -\partial_2 \psi\partial_{y_1}\partial_z
        +\partial_{y_1}\partial_z
        +\partial_{y_2}\partial_z.
\end{aligned}
\end{equation}
The combination of \eqref{315016}-\eqref{315017}
and Proposition \ref{prop2.3} yield that
\begin{equation}\label{315018}
\begin{aligned}
\|\nabla^2 d\|_{L^\infty}^2
&\le C_1(\|\Delta d\|_{L^\infty}^2+\|\partial_z \partial_{y_i} d\|_{L^\infty}^2
      +\|\partial_{y_i}\partial_{y_j}d\|_{L^\infty}^2)\\
&\le C_1(\|\nabla \Delta d\|_1^2+\|\Delta d\|_2^2
      +\|\nabla \partial_z \partial_{y_i} d\|_1^2
       +\|\partial_z \partial_{y_i} d\|_2^2)\\
&\quad   +C(\|\nabla \partial_{y_i}\partial_{y_j}d\|_1^2
       +\|\partial_{y_i}\partial_{y_j}d\|_2^2)\\
&\le C_1(\|\nabla \Delta d\|_1^2+\|\Delta d\|_2^2
      +\|\nabla^2 d\|_2^2+\|\nabla d\|_3^2)\\
&\le C_4(\|\nabla \Delta d\|_1^2+\|\Delta d\|_2^2+\|\nabla d\|_3^2),
\end{aligned}
\end{equation}
where we have used the estimate \eqref{31509} in the last inequality.
In order to deal with the first term on the right hand side of \eqref{315018},
we apply the equation \eqref{eq1}$_3$ to attain that
\begin{equation}\label{315019}
\begin{aligned}
\|\nabla \Delta d\|_1^2
&\le \|\nabla(d_t+u\cdot \nabla d-|\nabla d|^2 d)\|_1^2\\
&\le \|\nabla d_t\|_1^2
     +\|\nabla u\cdot \nabla d\|_1^2
     +\|u\cdot \nabla^2 d\|_1^2
     +\|\nabla(|\nabla d|^2 d)\|_1^2.
\end{aligned}
\end{equation}
It is easy to check that
\begin{equation}\label{315020}
\|\nabla d_t\|_1^2 \le \|\nabla d\|_{\mathcal{H}^{2}}^2
\le \Lambda_{m}(t), ~\text{for}~m\ge2,
\end{equation}
\begin{equation}\label{315021}
\|\nabla u\cdot \nabla d\|_1^2
\le \|(\nabla u, \nabla d)\|_{L^\infty}^2 \|(\nabla u,\nabla d)\|_1^2
\le C_3 \Lambda_{m}^2(t), ~\text{for}~m\ge2,
\end{equation}
and
\begin{equation}\label{315022}
\|u\cdot \nabla^2 d\|_1^2
\le \|u\|_{L^\infty}^2\|\nabla^2 d\|^2
     +\|Z u\|_{L^\infty}^2\|\nabla^2 d\|^2
     +\|u\|_{L^\infty}^2\|\nabla^2 d\|_1^2
\le C\Lambda_{m}^2(t), ~\text{for}~m\ge3.
\end{equation}
In view of the basic fact $|d|=1$, one arrives at
\begin{equation}\label{315023}
\|\nabla(|\nabla d|^2 d)\|_1^2
\le C \Lambda_{m}^3(t), ~m\ge3.
\end{equation}
Substituting \eqref{315020}-\eqref{315023} into \eqref{315019}, we find
\begin{equation*}
\|\nabla \Delta d\|_1^2 \le C_3 \Lambda_{m}^3(t), ~~m\ge3,
\end{equation*}
which, together with \eqref{315018}, yields immediately
\begin{equation}\label{315024}
\|\nabla^2 d\|_{L^\infty}^2 \le C_4 P(\Lambda_{m}(t)), ~~m\ge3.
\end{equation}
Similarly, it is easy to check that for $|\alpha|=1$
\begin{equation}\label{315024-a}
\|\mathcal{Z}^\alpha \nabla^2 d\|_{L^\infty}^2
\le C_3 P(\Lambda_{m}(t)).
\end{equation}
By virtue of the \eqref{eq1}$_3$, \eqref{315014} and \eqref{315024},
one attains for $m \ge 3$
\begin{equation}\label{315025}
\begin{aligned}
\|\nabla \Delta d\|_{L^\infty}^2
&\le \|\nabla(d_t+u\cdot \nabla d-|\nabla d|^2 d)\|_{L^\infty}^2\\
&\le \|\nabla d_t\|_{L^\infty}^2
      +\|\nabla u\|_{L^\infty}^2\|\nabla d\|_{L^\infty}^2
      +\|u\|_{L^\infty}^2\|\nabla^2 d\|_{L^\infty}^2\\
&\quad  +\|\nabla d\|_{L^\infty}^2\|\nabla^2 d\|_{L^\infty}^2
      +\|\nabla d\|_{L^\infty}^4\|\nabla d\|_{L^\infty}^2\\
&\le  C_4 \Lambda_{m}^3(t).
\end{aligned}
\end{equation}
Similarly, it is easy to check that for $|\alpha|=1$
\begin{equation}\label{315026}
\|\mathcal{Z}^\alpha\nabla \Delta d\|_{L^\infty}^2
\le C_{m+2} P (\Lambda_{m}(t)),\quad m\ge 3.
\end{equation}
The combination of \eqref{315010-1}, \eqref{315015} and \eqref{315024}-\eqref{315026}
yields the estimate \eqref{31507}.
Therefore, we complete the proof of Lemma \ref{lemma3.15}.
\end{proof}

In order to give the estimate for $\|\nabla u\|_{\mathcal{H}^{1,\infty}}$,
we need the lemma as follows, refer to \cite{Wang-Xin-Yong}.
\begin{lemm}\label{lemma3.16}
Let $h$ be a smooth solution to
\begin{equation}\label{3.1601}
a(t,y)[\partial_t h+b_1(t,y)\partial_{y_1}h+b_2(t,y)\partial_{y_2}h
+z b_3(t, y)\partial_z h]-\varepsilon \partial_{zz}h=G,
~z>0, \quad h(t,y,0)=0,
\end{equation}
for some smooth function $d(t,y)=\frac{1}{a(t,y)}$ and vector field
$b=(b_1, b_2, b_3)^{tr}(t,y)$ satisfying \eqref{3.1601}.
Assume that $h$ and $G$ are compactly supported in $z$.
Then, it holds that
\begin{equation}\label{3.1602}
\begin{aligned}
\|h\|_{\mathcal{H}^{1,\infty}}
&\le C\|h_0\|_{\mathcal{H}^{1,\infty}}
+C\int_0^t \left\|\frac{1}{a}\right\|_{L^\infty} \|G\|_{\mathcal{H}^{1,\infty}}d\tau\\
&\quad +C\int_0^t \left(1+\left\|\frac{1}{a}\right\|_{L^\infty}\right)
\left(1+\|b\|_{L^\infty}^2+\sum_{i=0}^2 \|Z_i(a,b)\|_{L^\infty}^2\right)
\|h\|_{\mathcal{H}^{1,\infty}} d\tau.
\end{aligned}
\end{equation}
\end{lemm}

Finally, one gives the estimate for the quantity $\|\nabla u\|_{\mathcal{H}^{1,\infty}}$.
\begin{lemm}\label{lemma3.17}
For $m \ge 6$, we have the estimate
\begin{equation}\label{31701}
\begin{aligned}
\|\nabla u\|_{\mathcal{H}^{1,\infty}}^2
&\le CC_{m+2}\left\{\|(u_0, \nabla u_0)\|_{\mathcal{H}^{1,\infty}}^2
\!+P(\Lambda_{1m}(t))\!+P(\|\Delta p(t)\|_{\mathcal{H}^1}^2)
\!+\varepsilon^2 t \!\!\int_0^t \!\!\|\nabla^2 u\|_{\mathcal{H}^4}^2d\tau\right\}\\
&\quad +CC_{m+2}~ t\int_0^t(1+P(\Lambda_{m}(\tau))+P(Q(\tau)))
(1+\varepsilon^2 \|\Delta p\|_{\mathcal{H}^2}^2)d\tau.
\end{aligned}
\end{equation}
\end{lemm}
\begin{proof}
Away from the boundary, we clearly have by the classical isotropic Sobolev
embedding that
\begin{equation}\label{31702}
\|\chi \nabla u\|_{L^\infty}^2
+
\|\chi \mathcal{Z}^\alpha \nabla u\|_{L^\infty}^2
\lesssim \|u\|_{\mathcal{H}^{m}}^2,
\quad m \ge 4,\quad |\alpha|=1,
\end{equation}
where the support of $\chi$ is away from the boundary.
Consequently, by using a partition of unity subordinated to the covering
we only have to estimate $\|\chi_j \nabla u\|_{L^\infty}+
\|\chi_j \mathcal{Z}^\alpha \nabla u\|_{L^\infty}, ~j\ge 1, ~|\alpha|=1$.
For notational convenience, we shall denote $\chi_j$ by $\chi$.
Similar to \cite{Masmoudi-Rousset} or \cite{Wang-Xin-Yong},
we use the local parametrization in the
neighborhood of the boundary given by a normal geodesic system in which the
Laplacian takes a convenient form. Denote
\begin{equation*}
\Psi^n(y, z)=
\left(
\begin{array}{c}
y\\
\psi(y)
\end{array}
\right)
-zn(y)=x,
\end{equation*}
where
\begin{equation*}
n(y)=\frac{1}{\sqrt{1+|\nabla \psi(y)|^2}}
\left(
\begin{array}{c}
\partial_1 \psi(y)\\
\partial_2 \psi(y)\\
-1
\end{array}
\right)
\end{equation*}
is the unit outward normal. As before, one can extend $n$ and $\Pi$
in the interior by setting
\begin{equation*}
n(\Psi^n(y,z))=n(y), \quad \Pi(\Psi^n(y,z))=\Pi(y)=I-n\otimes n,
\end{equation*}
where $I$ is the unit matrix.
Note that $n(y,z)$ and $\Pi(y,z)$ have different definitions from
the ones used before.
The advantages of this parametrization is that in the associated local
basis $(e_{y_1}, e_{y_2}, e_z)$ of $\mathbb{R}^3$, it holds that
$\partial_z=\partial_n$ and
\begin{equation*}
\left.(e_{y_i})\right|_{\Psi^n(y,z)}
\cdot \left.(e_z)\right|_{\Psi^n(y,z)}=0,\quad i=1,2.
\end{equation*}
The scalar product on $\mathbb{R}^3$ induces in this coordinate system
the Riemannian metric $g$ with the norm
\begin{equation*}
g(y,z)=
\left(
\begin{array}{cc}
\widetilde{g}(y,z)& 0\\
0 & 1
\end{array}
\right).
\end{equation*}
Therefore, the Laplacian in this coordinate system has the form
\begin{equation}\label{31703}
\Delta f=\partial_{zz}f+\frac{1}{2}\partial_z(\ln |g|)\partial_z f+\Delta_{\widetilde{g}}f,
\end{equation}
where $|g|$ denotes the determinant of the matrix $g$,
and $\Delta_{\widetilde{g}}$ is defined by
\begin{equation*}
\Delta_{\widetilde{g}}f
=\frac{1}{\sqrt{|\widetilde{g}|}}\sum_{i,j=1,2}
\partial_{y_i}(\widetilde{g}^{ij}|\widetilde{g}|^{\frac{1}{2}}\partial_{y_j}f),
\end{equation*}
which  only involves the tangential derivatives and
$\{\widetilde{g}^{ij}\}$ is the inverse matrix to $g$.

Next, thanks to \eqref{32a}(in the coordinate system that
we have just defined) and Lemma \ref{lemma3.15},
we have for $m \ge 5,~|\alpha|=1$
\begin{equation}\label{31704}
\begin{aligned}
&\|\chi \nabla u\|_{L^\infty}^2+
\|\chi \mathcal{Z}^\alpha \nabla u\|_{L^\infty}^2\\
&\le C_2(\|\chi \Pi (\partial_n u)\|_{L^\infty}^2+
\|\chi {\rm div}u\|_{L^\infty}^2+
\|\chi Z_y u\|_{L^\infty}^2)\\
&\quad +C_2(\|\chi \mathcal{Z}^\alpha\Pi (\partial_n u)\|_{L^\infty}^2+
\|\chi \mathcal{Z}^\alpha{\rm div}u\|_{L^\infty}^2+
\|\mathcal{Z}^\alpha(\chi  Z_y u\|_{L^\infty}^2))\\
&\le C_3\left\{\|\chi \Pi\partial_n u\|_{L^\infty}^2
+\|\mathcal{Z}^\alpha(\chi \Pi \partial_n u)\|_{L^\infty}^2
+P(\Lambda_{1m})+P(\|\Delta p\|_{\mathcal{H}^1}^2)\right\}.
\end{aligned}
\end{equation}
Consequently, it suffices to estimate $\|\chi \Pi \partial_n u\|_{\mathcal{H}^{1,\infty}}$.
To this end, it is useful to use the vorticity $w=\nabla \times u$,
see \cite{{Xiao-Xin},{Wang-Xin-Yong},{Masmoudi-Rousset}}.
Indeed, it is easy to deduce that
\begin{equation*}
\Pi(w\times n)=\Pi((\nabla u-\nabla u^t)\cdot n)
=\Pi(\partial_n u-\nabla(u\cdot n)+\nabla n^t \cdot u),
\end{equation*}
which implies
\begin{equation}\label{31705}
\|\chi \Pi \partial_n u\|_{\mathcal{H}^{1,\infty}}^2
\le C_3(\|\chi \Pi(w\times n)\|_{\mathcal{H}^{1,\infty}}^2
+P(\Lambda_{1m}(t))),
\end{equation}
where we have used the Lemma \ref{lemma3.15}.
In other words, we only need to estimate
$\|\chi \Pi(w\times n)\|_{\mathcal{H}^{1,\infty}}$.
It is easy to see that $w$ solves the vorticity equation
\begin{equation}\label{31706}
\rho w_t+\rho (u\cdot \nabla)w
=\mu \varepsilon \Delta w+F_1,
\end{equation}
where
\begin{equation*}
F_1\triangleq
-\nabla \rho \times u_t-\nabla \rho \times (u\cdot \nabla)u
+\rho (w \cdot \nabla)u-\rho w {\rm div}u
-\nabla \times (\nabla d\cdot \Delta d).
\end{equation*}
In the support of $\chi$, let
\begin{equation*}
\widetilde{w}(y, z)=w(\Psi^n(y,z)),
\quad (\widetilde{\rho}, \widetilde{u}, \widetilde{d})(y, z)
=(\rho, u, d)(\Psi^n(y,z)),
\end{equation*}
The combination of \eqref{342} and \eqref{31703} yields directly
\begin{equation}\label{31707}
\widetilde{\rho} ~\partial_t {\widetilde{w}}
+\widetilde{\rho}~ \widetilde{u}^1\partial_{y_1}{\widetilde{w}}
+\widetilde{\rho}~ \widetilde{u}^2\partial_{y_2}{\widetilde{w}}
+\widetilde{\rho}~ \widetilde{u}\cdot n \partial_z {\widetilde{w}}
=\mu \varepsilon(\partial_{zz}{\widetilde{w}}
+\frac{1}{2}\partial_z (\ln |g|)\partial_z {\widetilde{w}}
+\Delta_{\widetilde{g}} {\widetilde{w}})
+\widetilde{F}_1
\end{equation}
and
\begin{equation}\label{31708}
\begin{aligned}
\widetilde{\rho} ~ \partial_t {\widetilde{u}}
+\widetilde{\rho} ~\widetilde{u}^1\partial_{y_1}{\widetilde{u}}
+\widetilde{\rho} ~\widetilde{u}^2\partial_{y_2}{\widetilde{u}}
+\widetilde{\rho} ~\widetilde{u}\cdot n \partial_z {\widetilde{u}}
=\mu \varepsilon(\partial_{zz}{\widetilde{u}}
+\frac{1}{2}\partial_z (\ln |g|)\partial_z {\widetilde{u}}
+\Delta_{\widetilde{g}} {\widetilde{u}})
+\widetilde{F}_2,
\end{aligned}
\end{equation}
where $\widetilde{F}_2 =F_2(\Psi^n(y,z))$
and $F_2=(\mu+\lambda)\varepsilon \nabla {\rm div}u-\nabla p-\nabla d\cdot \Delta d$.
Similar to \eqref{32c} , we define
\begin{equation*}
\widetilde{\eta}=\chi(\widetilde{w}\times n+\Pi(B\widetilde{u})).
\end{equation*}
It is easy to deduce taht $\widetilde{\eta}$ satisfies
\begin{equation*}
\widetilde{\eta}(y, 0)=0.
\end{equation*}
and solves the equation
\begin{equation}\label{31709}
\begin{aligned}
&\widetilde{\rho}~\partial_t \widetilde{\eta}
+\widetilde{\rho}~\widetilde{u}^1\partial_{y_1}\widetilde{\eta}
+\widetilde{\rho}~\widetilde{u}^2\partial_{y_2}\widetilde{\eta}
+\widetilde{\rho}~\widetilde{u}\cdot n \partial_z \widetilde{\eta}\\
&=\mu\varepsilon\left(\partial_{zz}\widetilde{\eta}
+\frac{1}{2}\partial_z(\ln |g|)\partial_z \widetilde{\eta}\right)
+\chi (\widetilde{F}_1 \times n)
+\chi \Pi(B \widetilde{F}_2)+F^\chi+\chi F^{\kappa},
\end{aligned}
\end{equation}
where the source terms are given by
\begin{equation}\label{31710}
\begin{aligned}
F^\chi
&=\left[(\widetilde{\rho}~\widetilde{u}^1\partial_{y_1}
+\widetilde{\rho}~\widetilde{u}^2\partial_{y_2}
+\widetilde{\rho}~\widetilde{u}\cdot n \partial_z)\chi\right]
(\widetilde{w}\times n+\Pi(B\widetilde{u}))\\
&-\mu\varepsilon\left(\partial_{zz}\chi+2 \partial_z \chi \partial_z
      +\frac{1}{2}\partial_z(\ln |g|)\partial_z \chi\right)
      (\widetilde{w}\times n+\Pi(B\widetilde{u})),
\end{aligned}
\end{equation}
and
\begin{equation}\label{31711}
\begin{aligned}
F^\kappa
=&(\widetilde{\rho}~\widetilde{u}^1 \partial_{y_1}\Pi
+\widetilde{\rho}~\widetilde{u}^2 \partial_{y_2}\Pi)\cdot (B\widetilde{u})
+w\times (\widetilde{\rho}~\widetilde{u}^1 \partial_{y_1}n
+\widetilde{\rho}~\widetilde{u}^2 \partial_{y_2}n)\\
&+\Pi\left[(\widetilde{\rho}~\widetilde{u}^1 \partial_{y_1}
+\widetilde{\rho}~\widetilde{u}^2 \partial_{y_2}
            +\widetilde{\rho}~\widetilde{u}\cdot n \partial_{z})B\cdot \widetilde{u}\right]
            +\mu\varepsilon \Delta_{\widetilde{g}}\widetilde{w} \times n
            +\mu\varepsilon \Pi (B \Delta_{\widetilde{g}}\widetilde{u}).
\end{aligned}
\end{equation}
Note that in the derivation of the source terms above, in particular,
$F^\kappa$, which contains all the commutators coming from the fact that
$n$ and $\Pi$ are not constant, we have used the fact that in the coordinate system
just defined, $n$ and $\Pi$ do not depend on the normal variable.
Since $\Delta_{\widetilde{g}}$ involves only the tangential derivatives,
and the derivatives of $\chi$ are compactly supported away from the boundary,
the following estimates hold for $m \ge 6$
\begin{equation}\label{31712}
\| F^\chi\|_{\mathcal{H}^{1,\infty}}^2
\le C_3(\|\rho u\|_{\mathcal{H}^{1,\infty}}^2\|u\|_{\mathcal{H}^{2,\infty}}^2
+\varepsilon^2 \|u\|_{\mathcal{H}^{3,\infty}}^2)
\le C_3 \left\{P(Q(t))+P(\Lambda_{1m})\right\},
\end{equation}
\begin{equation}\label{31713}
\|\chi (\widetilde{F}^1 \times n)\|_{\mathcal{H}^{1,\infty}}^2
\le C_2(P(Q(t))
+\|\nabla d\|_{\mathcal{H}^{1,\infty}}^2\|\nabla \Delta d\|_{\mathcal{H}^{1,\infty}}^2)
\le C_2(P(Q(t))+P(\Lambda_{m})),
\end{equation}
\begin{equation}\label{31714}
\begin{aligned}
\|\chi \Pi(B \widetilde{F}_2)\|_{\mathcal{H}^{1,\infty}}^2
&\le C_3(\varepsilon^2\|\nabla {\rm div}u\|_{\mathcal{H}^{1,\infty}}^2
+\|\nabla p\|_{\mathcal{H}^{1,\infty}}^2
+\|\nabla d\|_{\mathcal{H}^{1,\infty}}^2\|\Delta d\|_{\mathcal{H}^{1,\infty}}^2)\\
&\le C_4\left\{P(Q(t))+P(\Lambda_{m})
+C\varepsilon^2[1+P(Q(t))]\|\Delta p\|_{\mathcal{H}^2}^2\right\},
\end{aligned}
\end{equation}
and
\begin{equation}\label{31715}
\begin{aligned}
\|\chi F^\kappa\|_{\mathcal{H}^{1,\infty}}^2
&\le C_4\left\{\|u\|_{\mathcal{H}^{1,\infty}}^8
+\|u\|_{\mathcal{H}^{1,\infty}}^4\|\nabla u\|_{\mathcal{H}^{1,\infty}}^4
+\|\rho\|_{\mathcal{H}^{1,\infty}}^2\right.\\
&\quad \quad \quad \left.
+\varepsilon^2(\|\nabla u\|_{\mathcal{H}^{3,\infty}}^2
+\|u\|_{\mathcal{H}^{3,\infty}}^2)\right\}\\
&\le C_4\left\{ \varepsilon^2 \|\nabla^2 u\|_{\mathcal{H}^4}^2
+P(\Lambda_{1m})+P(Q(t))\right\}.
\end{aligned}
\end{equation}
It follows from \eqref{31712}-\eqref{31715} that
\begin{equation}\label{31716}
\|F\|_{\mathcal{H}^{1,\infty}}^2
\le C_4\left\{\varepsilon^2 \|\nabla^2 u\|_{\mathcal{H}^4}^2
+P(Q(t))+P(\Lambda_{m})
+\varepsilon^2 [1+P(Q(t))]\|\Delta p\|_{\mathcal{H}^2}^2\right\},
\end{equation}
where $\widetilde{F}=\chi (\widetilde{F}_1 \times n)
+\chi \Pi(B \widetilde{F}_2)+F^\chi+\chi F^{\kappa}$.
In order to be able to use Lemma \ref{lemma3.16}, we shall perform last change of
unknown in order to eliminate the term
$\partial_z(\ln |\widetilde{g}|)\partial_z \widetilde{\eta}$.
We set
$$
\widetilde{\eta}=\frac{1}{|g|^{\frac{1}{4}}}\overline{\eta}=\overline{\gamma} ~\overline{\eta}.
$$
Note that we have
\begin{equation}\label{31717}
\|\widetilde{\eta}\|_{\mathcal{H}^{1,\infty}}
\le C_3 \|\overline{\eta}\|_{\mathcal{H}^{1,\infty}},
\quad \|\overline{\eta}\|_{\mathcal{H}^{1,\infty}}
\le C_3 \|\widetilde{\eta}\|_{\mathcal{H}^{1,\infty}}
\end{equation}
and that, $\overline{\eta}$ solves the equation
\begin{equation*}
\begin{aligned}
&\widetilde{\rho}~\partial_t \overline{\eta}
+\widetilde{\rho}~\widetilde{u}^1\partial_{y_1}\overline{\eta}
+\widetilde{\rho}~\widetilde{u}^2\partial_{y_2}\overline{\eta}
+\widetilde{\rho}~\widetilde{u}\cdot n \partial_z \overline{\eta}
-\mu\varepsilon \partial_{zz}\overline{\eta}\\
&=\frac{1}{\overline{\gamma}}\left(\widetilde{F}
+\mu\varepsilon\partial_{zz}\overline{\gamma}\cdot \overline{\eta}
+\frac{\mu\varepsilon }{2}\partial_z(\ln |g|)\partial_z \overline{\gamma}\cdot \overline{\eta}
-\widetilde{\rho}(\widetilde{u}\cdot \nabla \overline{\gamma} )\overline{\eta}\right)
:=\mathcal{S}_1.
\end{aligned}
\end{equation*}
In the spirit of Wang et al. \cite{Wang-Xin-Yong}, we rewrite the equation
as follows
\begin{equation}\label{31718}
\widetilde{\rho}(t, y, 0)\left[\overline{\eta}_t+\widetilde{u}^1(t, y, 0)\partial_{y^1}
+\widetilde{u}^2(t, y, 0)\partial_{y^2}\overline{\eta}
+z\partial_z(\widetilde{u}\cdot n)(t, y, 0)\partial_z \overline{\eta}\right]
-\mu \varepsilon \partial_{zz}\overline{\eta}
=\mathcal{S}_1+\mathcal{S}_2,
\end{equation}
where $\mathcal{S}_2$ is defined as
\begin{equation*}\label{31719}
\begin{aligned}
\mathcal{S}_2\triangleq
&[\widetilde{\rho} (t, y, 0)-\widetilde{\rho}(t, y, z)]\eta_t
+\sum_{i=1,2}[(\widetilde{\rho} ~\widetilde{u}^i)(t, y, 0)
-(\widetilde{\rho} ~\widetilde{u}^i)(t, y, z)]\partial_{y_i}\overline{\eta}\\
&-\widetilde{\rho}(t, y, z)[(\widetilde{u}\cdot n)(t, y, z)
-z \partial_z(\widetilde{u}\cdot n)(t, y, 0)]\partial_z \overline{\eta}\\
&-[\widetilde{\rho}(t, y, z)-\widetilde{\rho}(t, y, 0)]
z\partial_z (\widetilde{u}\cdot n)(t, y, 0)\partial_z \overline{\eta}.
\end{aligned}
\end{equation*}
Consequently, by using Lemma \ref{lemma3.16}, we get from \eqref{31718} that for $m \ge 6$
\begin{equation}\label{31720}
\begin{aligned}
\|\overline{\eta}\|_{\mathcal{H}^{1,\infty}}
&\lesssim
C\|\overline{\eta}_0\|_{\mathcal{H}^{1,\infty}}
+C\int_0^t \|\widetilde{\rho}^{~-1}\|_{L^\infty}
\|(\mathcal{S}_1+\mathcal{S}_2)\|_{\mathcal{H}^{1,\infty}}d\tau\\
&\quad +C\int_0^t (1+\|\widetilde{\rho}^{~-1}\|_{L^\infty})
(1+\|(\rho, u, \nabla u)\|_{\mathcal{H}^{1,\infty}}^2)
\|\overline{\eta}\|_{\mathcal{H}^{1,\infty}}d\tau\\
&\lesssim
C\|\overline{\eta}_0\|_{\mathcal{H}^{1,\infty}}
+C\int_0^t \|(\mathcal{S}_1+\mathcal{S}_2)\|_{\mathcal{H}^{1,\infty}}d\tau\\
&\quad +C\int_0^t (1+P(\Lambda_{1m})+\|\mathcal{Z}\nabla u\|_{L^\infty}^2)
\|\overline{\eta}\|_{\mathcal{H}^{1,\infty}}d\tau.
\end{aligned}
\end{equation}
On the other hand, following the same  argument as \cite{Wang-Xin-Yong},
we have the following estimate
\begin{equation}\label{31721}
\|(\mathcal{S}_1+\mathcal{S}_2)\|_{\mathcal{H}^{1,\infty}}^2
\le C_4\left\{\varepsilon^2\|\nabla^2 u\|_{\mathcal{H}^4}^2
+\varepsilon^2[1+P(Q(t))]\|\Delta p\|_{\mathcal{H}^2}^2
+P(Q(t))+P(\Lambda_{m})\right\}.
\end{equation}
Then, we deduce from  \eqref{31720}-\eqref{31721} that
\begin{equation*}
\begin{aligned}
\|\overline{\eta}(t)\|^2_{\mathcal{H}^{1,\infty}}
\le &\|\overline{\eta}_0\|^2_{\mathcal{H}^{1,\infty}}
+C_4 t\int_0^t(1+P(Q(\tau))+P(\Lambda_{m}))d\tau\\
&+C_4 t \varepsilon^2 \int_0^t([1+P(Q(t))]\|\Delta p\|_{\mathcal{H}^2}^2
+\|\nabla^2 u\|_{\mathcal{H}^4}^2)d\tau.
\end{aligned}
\end{equation*}
which, together with  \eqref{31702}, \eqref{31717}, \eqref{31704},
\eqref{31705}, completes the proof of Lemma \ref{lemma3.17}.
\end{proof}

\subsection{Uniform estimate for $\Delta p$}

\quad In this subsection, we shall estimate $\Delta p$ to
complete the $L^\infty-$estimates and prove that the boundary
layers for the density is weaker that the one for the velocity.
Taking divergence operator to the equation \eqref{eq1-1}, it is easy to deduce that
\begin{equation}\label{35a}
-\varepsilon \Delta {\rm div}u+\frac{1}{2\mu+\lambda}\Delta p
=-\frac{1}{2\mu+\lambda}{\rm div}(\rho u_t+\rho u\cdot \nabla u)
-\frac{1}{2\mu+\lambda}{\rm div}(\nabla d\cdot \Delta d).
\end{equation}
On the other hand, it follows from the equation \eqref{eq1}$_1$ that
\begin{equation}\label{35b}
{\rm div}u=-({\rm ln}~\rho)_t-u\cdot \nabla {\rm ln}~\rho
=-\frac{p_t}{\gamma p}-\frac{u\cdot \nabla p}{\gamma p}.
\end{equation}
Then the combination of \eqref{35a} and  \eqref{35b} yields that
\begin{equation}\label{35c}
\begin{aligned}
&\varepsilon \Delta ({\rm ln}\rho)_t
+\varepsilon u\cdot \nabla \Delta {\rm ln}\rho
+\varepsilon \Delta u\cdot \nabla {\rm ln}\rho
+2 \varepsilon \sum_{k=1}^3\partial_k u \cdot \nabla \partial_k {\rm ln}\rho
+\frac{1}{2\mu+\lambda}\Delta p\\
&=
-\frac{1}{2\mu+\lambda}{\rm div}(\rho u_t+\rho u\cdot \nabla u)
-\frac{1}{2\mu+\lambda}{\rm div}(\nabla d\cdot \Delta d).
\end{aligned}
\end{equation}

\begin{lemm}\label{lemma3.18}
For $m \ge 6$, it holds that
\begin{equation}\label{31801}
\begin{aligned}
&\underset{0\le \tau \le t}{\sup}\left(\|\Delta p(\tau)\|_{\mathcal{H}^1}^2
+\varepsilon\|\Delta p(\tau)\|_{\mathcal{H}^2}^2\right)
+\int_0^t \|\Delta p(\tau)\|_{\mathcal{H}^2}^2 d\tau\\
&\le CC_{m+2}\left\{P(N_m(0))+[1+P(Q(t))] \int_0^t P(N_m(\tau))d\tau\right\}.
\end{aligned}
\end{equation}
\end{lemm}
\begin{proof}
Applying $\mathcal{Z}^\alpha(|\alpha|\le 2)$ to \eqref{35c} and multiplying by $\mathcal{Z}^\alpha \Delta {\rm ln}\rho$, it is easy to deduce that
\begin{equation}\label{31802}
\begin{aligned}
&\varepsilon \|\mathcal{Z}^\alpha \Delta {\rm ln}\rho\|^2
-\varepsilon\|\mathcal{Z}^\alpha \Delta {\rm ln}\rho_0\|^2
+\underset{\Rmnum{7}_1}{\underbrace{\frac{2}{2\mu+\lambda}\int_0^t \int \mathcal{Z}^\alpha \Delta p\cdot
\mathcal{Z}^\alpha \Delta {\rm ln}\rho dxd\tau}}\\
&=\underset{\Rmnum{7}_2}{\underbrace{-2\varepsilon\!\!\int_0^t \int \!\!\mathcal{Z}^\alpha (\Delta u\cdot \nabla {\rm ln}\rho)
\mathcal{Z}^\alpha \Delta {\rm ln}\rho~ dxd\tau}}
\quad \underset{\Rmnum{7}_3}{\underbrace{-4\varepsilon \sum_{k=1}^3 \int_0^t \int \mathcal{Z}^\alpha (\partial_k u\cdot \nabla \partial_k {\rm ln}\rho)\mathcal{Z}^\alpha \Delta {\rm ln}\rho ~dxd\tau}}\\
&\quad\underset{\Rmnum{7}_4}{\underbrace{ -2\varepsilon\!\!\int_0^t \!\!\int \mathcal{Z}^\alpha(u\cdot \nabla \Delta {\rm ln}\rho )
\mathcal{Z}^\alpha \Delta {\rm ln}\rho dxd\tau}}
+\underset{\Rmnum{7}_5}{\underbrace{\frac{-2}{2\mu+\lambda}\!\!\int_0^t \!\!\int \mathcal{Z}^\alpha {\rm div}
(\rho u_t+\rho u\cdot \nabla u)\mathcal{Z}^\alpha \Delta {\rm ln}\rho dxd\tau}}\\
&\quad \underset{\Rmnum{7}_6}{\underbrace{-\frac{2}{2\mu+\lambda}\int_0^t \int \mathcal{Z}^\alpha {\rm div}(\nabla d\cdot \Delta d)\mathcal{Z}^\alpha \Delta {\rm ln}\rho ~dxd\tau}}.
\end{aligned}
\end{equation}
Using the same argument as Lemma 3.16 of \cite{Wang-Xin-Yong},
one can obtain the following estimates
\begin{equation}\label{31803}
VII_5 \le C_{m+2}[1+P(Q(t))]\int_0^t P(\Lambda_m)\tau,
\end{equation}
\begin{equation}
VII_2 \le \delta \int_0^t \|\mathcal{Z}^\alpha \Delta {\ln}\rho\|_{L^2}^2 d\tau
+C_\delta[1+P(Q(t))]\int_0^t P(\Lambda_m)d\tau,
\end{equation}
\begin{equation}\label{31804}
\begin{aligned}
VII_3 &\le \delta \int_0^t \|\mathcal{Z}^\alpha \Delta {\ln}\rho\|_{L^2}^2 d\tau
+C_\delta \varepsilon^2 [1+P(Q(t))]
\int_0^t (P(\Lambda_m)+\|\Delta {\ln}\rho\|_{\mathcal{H}^2}^2)d\tau\\
&\quad +\varepsilon^2 \int_0^t \|\nabla^2 u\|_{\mathcal{H}^3}^2 d\tau
+C_\delta \varepsilon^2 \int_0^t \|\nabla u\|_{\mathcal{H}^{4}}^2
(\|\Delta {\ln}\rho\|_{L^2}^4+P(\Lambda_m))d\tau,
\end{aligned}
\end{equation}
\begin{equation}\label{31805}
VII_1 \ge \frac{\gamma}{2}p(c_0)\int_0^t \|\mathcal{Z}^\alpha \Delta {\ln}\rho\|_{L^2}^2 d\tau
-C[1+P(Q(t))]\int_0^t (P(\Lambda_m)+\|\Delta {\ln}\rho\|_{\mathcal{H}^1}^2)d\tau,
\end{equation}
\begin{equation}\label{31806}
VII_4
\le \varepsilon^2 \!\!\int_0^t\! \|\nabla^2 u\|_{\mathcal{H}^3}^2d\tau
+\delta \!\!\int_0^t\! \|\mathcal{Z}^\alpha \Delta {\ln}\rho\|_{L^2}^2 d\tau
+C_\delta C_2 [1+P(Q(t))]\!\!\int_0^t\! (\varepsilon^2 \|\Delta {\ln}\rho\|_{\mathcal{H}^2}^2
+P(\Lambda_m))d\tau.
\end{equation}
On the other hand, by using the Proposition \ref{prop2.2},
it is easy to check that
\begin{equation}\label{31807}
\begin{aligned}
VII_6&\le
\delta \int_0^t \|\Delta {\rm ln}\rho\|_{\mathcal{H}^2}^2 \tau
+C_\delta \|\Delta d\|_{L^\infty}^2\int_0^t \|\Delta d\|_{\mathcal{H}^2}^2d\tau\\
&\quad
+C_\delta \|\nabla d\|_{L^\infty}^2\int_0^t \|\nabla \Delta d\|_{\mathcal{H}^2}^2d\tau
+C_\delta\|\nabla \Delta d\|_{L^\infty}^2\int_0^t \|\nabla d\|_{\mathcal{H}^2}^2d\tau\\
&\le \delta \int_0^t \|\Delta {\rm ln}\rho\|_{\mathcal{H}^2}^2 d\tau
+C_\delta[1+P(Q(t))]\int_0^t \Lambda_m(\tau)d\tau.
\end{aligned}
\end{equation}
Hence, the combination of \eqref{31803}-\eqref{31807} gives directly
\begin{equation}\label{31808}
\begin{aligned}
&\varepsilon \|\Delta {\rm ln}\rho\|_{\mathcal{H}^2}^2
+\int_0^t \|\Delta {\rm ln}\rho\|_{\mathcal{H}^2}^2 d\tau\\
&\le C\varepsilon \|\Delta {\rm ln}\rho_0\|_{\mathcal{H}^2}^2
+C\varepsilon^2 \int_0^t \|\nabla^2 u\|_{\mathcal{H}^3}^2 d\tau\\
&\quad +C_\delta C_{m+2}[1+P(Q(t))]\int_0^t
(\varepsilon \|\Delta {\rm ln}\rho\|_{\mathcal{H}^2}^2
+\|\Delta {\rm ln}\rho\|_{\mathcal{H}^1}^4
+P(\Lambda_m))d\tau.
\end{aligned}
\end{equation}
On the other hand, it is easy to obtain that
\begin{equation*}
\begin{aligned}
\|\Delta {\rm ln}\rho(t)\|_{\mathcal{H}^1}^2
&\le \|\Delta {\rm ln}\rho_0\|_{\mathcal{H}^1}^2
+\int_0^t \| \Delta {\rm ln}\rho(t)\|_{\mathcal{H}^2}^2d\tau,
\end{aligned}
\end{equation*}
which, together with \eqref{31808}, yields directly
\begin{equation*}
\begin{aligned}
&\underset{0\le \tau \le t}{\sup}(\|\Delta {\rm ln}\rho(\tau)\|_{\mathcal{H}^1}^2
+\varepsilon\|\Delta {\rm ln}\rho(\tau)\|_{\mathcal{H}^2}^2)
+\int_0^t \|\Delta {\rm ln}\rho(\tau)\|_{\mathcal{H}^2}^2 d\tau\\
&\le  CC_{m+2}\left\{N_m(0)+[1+P(Q(t))]\int_0^t (1
+\varepsilon\|\Delta {\rm ln}\rho(\tau)\|_{\mathcal{H}^2}^2
+\|\Delta {\rm ln}\rho(\tau)\|_{\mathcal{H}^1}^4+P(\Lambda_m))d\tau\right\}.
\end{aligned}
\end{equation*}
Therefore, we complete the proof of Lemma \ref{lemma3.18}.
\end{proof}

\subsection{Proof of Theorem \ref{Theoream3.1}}
\quad By virtue of \eqref{31501}, \eqref{31502} and \eqref{31507},
it is easy to deduce that
\begin{equation}\label{36a}
\begin{aligned}
Q(t)
&\le C_3 \underset{0\le \tau \le t}{\sup} \left\{\|\nabla u(\tau)\|_{\mathcal{H}^{1,\infty}}^2+P(\Lambda_m(\tau))
+P(\|\Delta p(\tau)\|_{\mathcal{H}^1}^2)\right\}\\
&\le CC_{m+2}\left\{P(N_m(0))+P(N_m(t))\int_0^t P(N_m(\tau))d\tau \right\}.
\end{aligned}
\end{equation}
In order to close the a priori estimates, one still need to get the uniform
estimate for $\|\nabla \partial_t^{m-1} u\|_{L^2}^2$.
To this end, we combine \eqref{3a8}, \eqref{36a} and Lemma \ref{lemma3.11} to deduce that
\begin{equation}\label{36b}
\begin{aligned}
\|\nabla \partial_t^{m-1} u\|_{L^2}^2
&\le CC_{m+2}\left\{\|u(t)\|_{\mathcal{H}^m}^2+\|\eta(t)\|_{\mathcal{H}^{m-1}}^2
+\|\partial_t^{m-1} {\rm div}u\|_{L^2}^2\right\}\\
&\le CC_{m+2}\left\{P(\Lambda_{1m})
+\|\eta(t)\|_{\mathcal{H}^{m-1}}^2+P(Q(t))\right\}\\
&\le CC_{m+2}\left\{P(N_m(0))+P(N_m(t))\int_0^t P(N_m(\tau)) d\tau\right\}.
\end{aligned}
\end{equation}
Hence, the combination of \eqref{3a8}, \eqref{31801}, \eqref{36a}
and \eqref{36b} yields for $m \ge 6$ that
\begin{equation*}
\begin{aligned}
&N_m(t)
+\int_0^t(\|\nabla \partial_t^{m-1}p(\tau)\|_{L^2}^2
+\|\Delta p(\tau)\|_{\mathcal{H}^2}^2)d\tau
+\varepsilon \int_0^t \|\nabla u(\tau)\|_{\mathcal{H}^{m}}^2 d\tau\\
&\quad \quad
+\varepsilon \sum_{k=0}^{m-2}\int_0^t \|\nabla^2 \partial_t^k u(\tau)\|_{m-1-k}^2d\tau
+\varepsilon^2 \int_0^t \|\nabla^2 \partial_t^{m-1}u(\tau)\|_{L^2}^2 d\tau\\
&\quad \quad
+\int_0^t \|\Delta d(\tau)\|_{\mathcal{H}^{m}}^2d\tau
+\int_0^t\|\nabla \Delta d(\tau)\|_{\mathcal{H}^{m-1}}^2d\tau\\
&
\le \widetilde{C}_2C_{m+2}\left\{P(N_m(0))+P(N_m(t))\int_0^t P(N_m(\tau)) d\tau\right\},
\quad \forall t\in [0, T],
\end{aligned}
\end{equation*}
which completes the proof of \eqref{31c}.
Furthermore, the \eqref{eq1}$_3$ implies that
\begin{equation*}
|\rho(x, 0)|{\rm exp}\left(-\int_0^t \|{\rm div}u(\tau)\|_{L^\infty}d\tau\right)
\le \rho(x,t) \le |\rho(x, 0)|{\rm exp}\left(\int_0^t \|{\rm div}u(\tau)\|_{L^\infty}d\tau\right),
\end{equation*}
which proves \eqref{31a}.
Therefore, we complete the proof of Theorem \ref{Theoream3.1}.

\section{Proof of Theorem \ref{Theorem1.1} (Uniform Regularity)}\label{Proof1}
\quad
In this section, we will give the proof for the Theorem \ref{Theorem1.1}.
Indeed, we shall indicate how to combine the a priori estimates obtained
so far to prove the uniform existence result.
Fixing $m \ge 6$, we consider the initial data
$(p_0^\varepsilon, u_0^\varepsilon, d_0^\varepsilon)\in X_{n,m}^\varepsilon$ such that
\begin{equation}
\mathcal{I}_m(0)=\underset{0< \varepsilon \le 1}{\sup}
                 \|(p_0^\varepsilon, u_0^\varepsilon, d_0^\varepsilon)\|_{X_{n,m}^\varepsilon}
                 \le \widetilde{C}_0
\quad {\rm and} \quad 0<\widehat{C}_0^{-1}\le \rho_0^\varepsilon \le \widehat{C}_0.
\end{equation}
For such initial data, since we are not aware of a local existence
result for \eqref{eq1} and \eqref{bc1}(or \eqref{bc2}),
we first establish the local existence of solution for \eqref{eq1} and \eqref{bc1}
with initial data $(p_0^\varepsilon, u_0^\varepsilon, d_0^\varepsilon)\in X_{n,m}^\varepsilon$.
For such initial data $(p_0^\varepsilon, u_0^\varepsilon, d_0^\varepsilon)$,
it is easy to see that there exists a sequence of smooth approximate initial data
$(p_0^{\varepsilon,\delta},u_0^{\varepsilon, \delta}, d_0^{\varepsilon,\delta})
\in X^{\varepsilon, ap}_{n,m}$
($\delta$ being a regularity parameter), which have enough space regularity
so that the time derivatives at the initial time can be defined by the equations
\eqref{eq1} and the boundary compatibility conditions are satisfied.
Fixed $\varepsilon \in (0, 1]$, one constructs the approximate solutions as follows:\\
(1)Define $u^0=u_0^{\varepsilon, \delta}$
and $d^0=d_0^{\varepsilon,\delta}$.\\
(2)Assume that $(u^{k-1}, d^{k-1})$  has been defined for $k \ge 1$.
Let $(\rho^k, u^k, d^k)$ be the unique solution to the following linearized
initial data boundary value problem:
\begin{equation}\label{4eq1}
\left\{
\begin{aligned}
&\rho^k_t+{\rm div}(\rho^k u^{k-1})=0,
&{\rm in}~ \Omega \times (0, T),\\
&\rho^k u^k_t+\rho^k u^{k-1}\cdot \nabla u^k+\nabla p^k
=\varepsilon \mu \Delta u^k+\varepsilon (\mu+\lambda)\nabla {\rm div}u^k
-\nabla d^k \cdot \Delta d^k,
& {\rm in}~ \Omega \times (0, T),\\
&d^k_t-\Delta d^k=|\nabla d^{k-1}|^2 d^{k-1}-u^{k-1}\cdot \nabla d^{k-1},
& {\rm in}~ \Omega \times (0, T),\\
\end{aligned}
\right.
\end{equation}
with initial data
\begin{equation}\label{4id}
(\rho^k, u^k, d^k)|_{t=0}=(\rho_0^{\varepsilon,\delta},
u_0^{\varepsilon,\delta}, d_0^{\varepsilon,\delta}),
\end{equation}
and Navier-type and Neumann boundary condition
\begin{equation}\label{4bc}
u^k\cdot n=0,
\quad n\times (\nabla \times u^k)=[Bu^k]_\tau,
\quad {\rm and}\quad
\frac{\partial d^k}{\partial n}=0,
\quad {\rm on}~\partial \Omega.
\end{equation}
Since $\rho^k, u^k$ and $d^k$ are decoupled, the existence of global unique smooth
solution $(\rho^k, u^k, d^k)(t)$ of \eqref{4eq1}-\eqref{4bc} can be obtained by using
classical methods, for example, the similar argument in Cho et al. \cite{Cho-Choe-Kim}.
On the other hand, by virtue of
$(\rho_0^{\varepsilon, \delta}, u_0^{\varepsilon, \delta}, d_0^{\varepsilon, \delta})
\in  H^{4m}\times H^{4m}\times H^{4(m+1)}$,
one proves that there exists a positive time $\widetilde{T}_1=\widetilde{T}_1(\varepsilon)$
(depending on $\varepsilon$,
$\|(\rho_0^{\varepsilon, \delta}, u_0^{\varepsilon, \delta})\|_{H^{4m}}$
and
$\| d_0^{\varepsilon, \delta}\|_{H^{4(m+1)}}$) such that
\begin{equation}\label{41}
\|(\rho^k, u^k)(t)\|_{H^{4m}}^2
+\|d^k(t)\|_{H^{4(m+1)}}^2 \le \widetilde{C}_1
\quad {\rm and}\quad
\frac{\widehat{C}_0}{2}\le \rho^k(t)\le 2\widehat{C}_0
\quad {\rm for}~0\le t \le \widetilde{T}_1,
\end{equation}
where the constant $\widetilde{C}_1$ depends on $\widetilde{C}_0, \varepsilon^{-1}$,
$\|(\rho_0^{\varepsilon, \delta}, u_0^{\varepsilon,\delta})\|_{H^{4m}}$
and
$\|d_0^{\varepsilon, \delta}\|_{H^{4(m+1)}}$.
Based on the above uniform time $\widehat{T}_1(\le \widetilde{T}_1)$(independent of $k$)
such that $(\rho^k, u^k, d^k)$ converges to a limit
$(\rho^{\varepsilon, \delta}, u^{\varepsilon, \delta}, d^{\varepsilon, \delta})$
as $k \rightarrow +\infty$ in the following strong sense:
$$
(\rho^k, u^k) \rightarrow (\rho^{\varepsilon, \delta}, u^{\varepsilon, \delta})
~{\rm in}~L^\infty(0, \widehat{T}_1; L^2)
\quad {\rm and}\quad \nabla u^k \rightarrow \nabla u^{\varepsilon, \delta}
~{\rm in}~L^2(0, \widehat{T}_1; L^2),
$$
and
$$
d^k \rightarrow d^{\varepsilon, \delta} ~{\rm in}~L^\infty(0, \widehat{T}_1; H^1)
\quad {\rm and}\quad \Delta d^k \rightarrow \Delta d^{\varepsilon, \delta}
~{\rm in}~L^2(0, \widehat{T}_1; L^2).
$$
It is easy to check that
$(\rho^{\varepsilon, \delta}, u^{\varepsilon, \delta}, d^{\varepsilon, \delta})$
is a classical solution to the problem \eqref{eq1} and \eqref{bc1}
with initial data
$(\rho_0^{\varepsilon, \delta}, u_0^{\varepsilon, \delta}, d_0^{\varepsilon, \delta})$.
In view of the lower semicontinuity  of norms, one can deduce from
the uniform bounds  \eqref{41} that
\begin{equation}\label{41}
\|(\rho^{\varepsilon, \delta}, u^{\varepsilon, \delta})(t)\|_{H^{4m}}^2
+
\|d^{\varepsilon, \delta}(t)\|_{H^{4(m+1)}}^2
\le \widetilde{C}_1 \quad {\rm and}\quad
\frac{\widehat{C}_0}{2}\le \rho^k(t) \le 2C_0
\quad{\rm for}~0\le t \le \widetilde{T}_1.
\end{equation}
Applying the a priori estimates given in Theorem \ref{Theoream3.1} to
the solution
$(\rho^{\varepsilon, \delta}, u^{\varepsilon, \delta}, d^{\varepsilon, \delta})$,
one can obtain a uniform
time $T_0$ and constant $C_3$(independent of $\varepsilon$ and $\delta$) such that
\begin{equation}\label{42}
\begin{aligned}
&
N_m(p^{\varepsilon,\delta}, u^{\varepsilon,\delta}, d^{\varepsilon,\delta})(t)
+\int_0^t(\|\nabla \partial_t^{m-1}p^{\varepsilon,\delta}(\tau)\|_{L^2}^2
+\|\Delta p^{\varepsilon,\delta}(\tau)\|_{\mathcal{H}^2}^2)d\tau\\
&
+\varepsilon \int_0^t \|\nabla u^{\varepsilon,\delta}(\tau)\|_{\mathcal{H}^{m}}^2 d\tau
+\varepsilon \sum_{k=0}^{m-2}\int_0^t
\|\nabla^2 \partial_t^k u^{\varepsilon,\delta}(\tau)\|_{m-1-k}^2d\tau
\\
&
+\varepsilon^2 \int_0^t \|\nabla^2 \partial_t^{m-1}u^{\varepsilon,\delta}(\tau)\|_{L^2}^2 d\tau
+\int_0^t \|\Delta d^{\varepsilon,\delta}(\tau)\|_{\mathcal{H}^{m}}^2d\tau
+\int_0^t\|\nabla \Delta d^{\varepsilon,\delta}(\tau)\|_{\mathcal{H}^{m-1}}^2d\tau\\
&\le \widetilde{C}_3,~\forall t \in [0, \min\{T_0, \widehat{T}_1\}],
\end{aligned}
\end{equation}
and
\begin{equation}\label{42-1}
\frac{1}{2\widehat{C}_0}\le \rho^{\varepsilon, \delta}(t)
\le 2\widehat{C}_0, \quad t\in [0, {\rm min}\{T_0, \widehat{T}_1\}].
\end{equation}
where $T_0$ and $\widetilde{C}_3$ depend only on $\widehat{C}_0$
and $\mathcal{I}_m(0)$. Based on the uniform estimate \eqref{42}
and \eqref{42-1} for
$(\rho^{\varepsilon, \delta}, u^{\varepsilon, \delta}, d^{\varepsilon, \delta})$,
one can pass the limit $\delta \rightarrow 0$ to get a strong solution
$(\rho^\varepsilon, u^{\varepsilon}, d^{\varepsilon})$ of \eqref{eq1} and \eqref{bc1}
with initial data $(\rho_0^\varepsilon, u_0^\varepsilon, d_0^\varepsilon)$ satisfying
\eqref{4eq1} by using a strong compactness arguments(see \cite{Simon}).
Indeed, it follows from
\eqref{42} that $(p^{\varepsilon,\delta}, u^{\varepsilon, \delta},
\nabla d^{\varepsilon, \delta})$
is bounded uniformly in $L^\infty(0, \widetilde{T}_2; H_{co}^m)$,
where $\widetilde{T}_2=\min\{T_0, \widetilde{T}_1\}$, while
$(\nabla p^{\varepsilon, \delta}, \nabla u^{\varepsilon, \delta},
\Delta d^{\varepsilon, \delta})$
is bounded uniformly in $L^\infty(0, \widetilde{T}_2; H_{co}^{m-1})$,
and $(\partial_t p^{\varepsilon, \delta}, \partial_t u^{\varepsilon, \delta},
\partial_t \nabla d^{\varepsilon, \delta})$
is bounded uniformly in $L^\infty(0, \widetilde{T}_2; H_{co}^{m-1})$.
Then, the strong compactness argument implies that
$(p^{\varepsilon, \delta}, u^{\varepsilon, \delta}, \nabla d^{\varepsilon, \delta})$
is compact in $\mathcal{C}([0, \widetilde{T}_2]; H_{co}^{m-1})$.
In particular, there exists a sequence $\delta_n \rightarrow 0^+$
and $(p^\varepsilon, u^\varepsilon, \nabla d^\varepsilon)
\in \mathcal{C}([0, \widetilde{T}_2]; H_{co}^{m-1})$
such that
\begin{equation*}
(p^{\varepsilon, \delta_n}, u^{\varepsilon, \delta_n}, \nabla d^{\varepsilon, \delta_n})
\rightarrow  (p^\varepsilon, u^{\varepsilon}, \nabla d^{\varepsilon})
~{\rm in}~\mathcal{C}([0, \widetilde{T}_2]; H_{co}^{m-1})~{\rm as}~
\delta_n \rightarrow 0^+.
\end{equation*}
Moreover, applying the lower semicontinuity of norms to the bounds
\eqref{42}, one obtains the bounds \eqref{42} and \eqref{42-1}
for $(p^\varepsilon, u^\varepsilon, d^\varepsilon)$.
It follows from the bounds of \eqref{42} and \eqref{42-1} for
$(p^\varepsilon, u^\varepsilon, d^\varepsilon)$,
and the anisotropic Sobolev inequality \eqref{24} that
$$
\begin{aligned}
&\underset{0\le t\le \widetilde{T}_2}{\sup}
\|(\rho^{\varepsilon,\delta_n}-\rho^{\varepsilon},
u^{\varepsilon,\delta_n}-u^{\varepsilon}, d^{\varepsilon,\delta_n}-d^{\varepsilon})(t)\|_{L^\infty}^2\\
&\le C\underset{0\le t\le \widetilde{T}_2}{\sup}
  \|\nabla(\rho^{\varepsilon,\delta_n}-\rho^{\varepsilon},
  u^{\varepsilon,\delta_n}-u^{\varepsilon}, d^{\varepsilon,\delta_n}-d^{\varepsilon})\|_{H^1_{co}}
  \|(\rho^{\varepsilon,\delta_n}-\rho^{\varepsilon},
     u^{\varepsilon,\delta_n}-u^{\varepsilon},
     d^{\varepsilon,\delta_n}-d^{\varepsilon})\|_{H^2_{co}}
  \rightarrow 0,
\end{aligned}
$$
and
$$
\underset{0\le t\le \widetilde{T}_2}{\sup}
\|\nabla(d^{\varepsilon,\delta_n}-d^{\varepsilon})\|_{L^\infty}^2
\le C\underset{0\le t\le \widetilde{T}_2}{\sup}
  \|\Delta(d^{\varepsilon,\delta_n}-d^{\varepsilon})\|_{H^1_{co}}
  \|\nabla(d^{\varepsilon,\delta_n}-d^{\varepsilon})\|_{H^2_{co}}
  \rightarrow 0,
$$
Hence, it is easy to check that $(\rho^\varepsilon, u^\varepsilon, d^\varepsilon)$
is a weak solution of the nematic liquid crystal flows \eqref{eq1}.
The uniqueness of the solution $(\rho^\varepsilon, u^\varepsilon, d^\varepsilon)$
comes directly from the Lipschitz regularity of solution.
Thus, the whole family $(\rho^{\varepsilon, \delta},
u^{\varepsilon, \delta}, d^{\varepsilon, \delta})$
converge to $(\rho^{\varepsilon}, u^{\varepsilon}, d^{\varepsilon})$.
Therefore,
we have established the local solution of equation \eqref{eq1}
and \eqref{bc1} with initial data
$(p_0^\varepsilon, u_0^{\varepsilon}, d_0^{\varepsilon})
\in X_{n,m}^\varepsilon,~t\in [0, T_2]$.

We shall use the local existence results to prove Theorem \ref{Theorem1.1}.
If $T_0 \le \widetilde{T}$, then Theorem \ref{Theorem1.1} follows
from \eqref{42} and \eqref{42-1} with $\widetilde{C}_1=\widetilde{C}_3$.
On the other hand, for the case $\widetilde{T} \le  T_0$,
based on the uniform estimate \eqref{42} and \eqref{42-1}, we can use the local existence
results established above to extend our solution step by step to the
uniform time interval $t\in [0, T_0]$.
Therefore, we complete the proof of Theorem \ref{Theorem1.1}.

\section{Proof of Theorem \ref{Theorem1.2} (Inviscid Limit)}\label{Proof2}

\quad In this section, we study the vanishing viscosity of solutions
for the equation \eqref{eq1} to the solution for the equation \eqref{eq2}
with a rate of convergence. It is easy to see that the solution
$(\rho, u, d)\in H^3 \times H^3 \times H^4$ of equation \eqref{eq1} and \eqref{bc1}
with initial data $(\rho_0, u_0, d_0)\in H^3 \times H^3 \times H^4$ satisfies
$$
\sum_{k=0}^3\|(\rho, u)\|_{C^k([0, T_1]; H^{3-k})}
+\sum_{k=0}^2\|d\|_{C^k([0, T_1]; H^{4-2k})}
\le \widetilde{C}_4
$$
where $\widetilde{C}_4$ depends only on $\|(\rho_0, u_0, d_0)\|_{H^3 \times H^3 \times H^4}$.
On the other hand, it follows from the Theorem \ref{Theorem1.1}
that the solution $(\rho^\varepsilon, u^\varepsilon, d^\varepsilon)$ of equation
\eqref{eq1} and \eqref{bc1} with initial data $(\rho_0, u_0, d_0)$
satisfies
$$
\|(p(\rho^\varepsilon), u^\varepsilon, d^\varepsilon)\|_{X_m^\varepsilon}
\le \widetilde{C}_1,
\quad
\frac{1}{2\widehat{C}_0}\le \rho^\varepsilon(t) \le 2\widehat{C}_0
\quad \forall t\in [0, T_0],
$$
where $T_0$ and $\widetilde{C}_1$ are defined in Theorem \ref{Theorem1.1}.
In particular, this uniform regularity implies the bound
$$
\|(\rho^\varepsilon, u^\varepsilon)\|_{W^{1,\infty}}
+\|d^\varepsilon\|_{W^{2,\infty}}
+\|\partial_t(\rho^\varepsilon, u^\varepsilon)\|_{L^\infty}
+\|d_t^\varepsilon\|_{W^{1,\infty}}\le \widetilde{C}_1,
$$
which plays an important role in the proof of Theorem \ref{Theorem1.2}.

Let us define
\begin{equation*}
\phi^\varepsilon=\rho^\varepsilon-\rho,\quad
v^\varepsilon=u^\varepsilon-u, \quad \varphi^\varepsilon=d^\varepsilon-d.
\end{equation*}
It then follows from \eqref{eq1} that
\begin{equation}\label{eq3}
\left\{
\begin{aligned}
&\partial_t \phi^\varepsilon+\rho{\rm div}v^\varepsilon
+u\cdot \nabla \phi^\varepsilon=R_1^\varepsilon,\\
&\rho \partial_t v^\varepsilon+\rho u\cdot \nabla v^\varepsilon
+\nabla(p^\varepsilon-p)
+\Phi^\varepsilon
=-\mu \varepsilon \nabla \times (\nabla \times v^\varepsilon)
+(2\mu+\lambda)\varepsilon \nabla {\rm div}v^\varepsilon
+R_2^\varepsilon+R_3^\varepsilon,\\
&\partial_t \varphi^\varepsilon
-\Delta \varphi^\varepsilon=R_4^\varepsilon,\\
\end{aligned}
\right.
\end{equation}
where
\begin{equation*}
\begin{aligned}
&R_1^\varepsilon=-\phi^\varepsilon {\rm div}v^\varepsilon
-v^\varepsilon \cdot \nabla \phi^\varepsilon
-\phi^\varepsilon{\rm div}u-\nabla \rho \cdot v^\varepsilon,\\
&R_2^\varepsilon=-\phi^\varepsilon v^\varepsilon_t-\phi^\varepsilon u_t
+\mu \varepsilon \Delta u+(\mu+\lambda)\varepsilon \nabla {\rm div}u,\\
&R_3^\varepsilon=-\nabla d^\varepsilon \cdot \Delta \varphi^\varepsilon
-\nabla \varphi^\varepsilon \cdot \Delta d,\\
&R_4^\varepsilon=-u\cdot \nabla \varphi^\varepsilon
                 -v^\varepsilon\cdot \nabla d^\varepsilon
                 +(\nabla \varphi^\varepsilon : \nabla (d^\varepsilon+d))d^\varepsilon
                 +|\nabla d|^2 \varphi^\varepsilon,\\
&\Phi^\varepsilon=(\rho^\varepsilon u^\varepsilon-\rho u)\cdot \nabla u^\varepsilon.
\end{aligned}
\end{equation*}
The boundary conditions to \eqref{eq3} are given as follows
\begin{equation}\label{bc3}
\left\{
\begin{aligned}
& v^\varepsilon \cdot n=0,
\quad n\times (\nabla \times v^\varepsilon)
=[B v^\varepsilon]_\tau+[Bu]_\tau-n\times w, \quad x\in \partial \Omega,\\
&\frac{\partial \varphi^\varepsilon}{\partial n}=0, \quad x\in \partial \Omega.
\end{aligned}
\right.
\end{equation}

\begin{lemm}\label{lemma5.1}
For $t \in [0, \min\{T_0, T_1\}]$, it holds that
\begin{equation}\label{511}
\underset{0\le \tau \le t}{\sup}(\|(\phi^\varepsilon,v^\varepsilon)(\tau)\|_{L^2}^2
+\|\varphi^\varepsilon(\tau)\|_{H^1}^2)
+\mu \varepsilon \int_0^t \|v^\varepsilon\|_{H^1}^2 d\tau
+\int_0^t (\|\nabla \varphi^\varepsilon\|_{L^2}^2
+\|\Delta \varphi^\varepsilon\|_{L^2}^2) d\tau
\le C \varepsilon^{\frac{3}{2}}.
\end{equation}
where $C>0$ depend only on $\widetilde{C}_0, \widetilde{C}_1$ and $\widetilde{C}_4$.
\end{lemm}
\begin{proof}
Multiplying \eqref{eq3}$_1$ by $v^\varepsilon$, it is easy to deduce that
\begin{equation}\label{512}
\begin{aligned}
&\frac{d}{dt}\frac{1}{2}\int \rho |v^\varepsilon|^2 dx
+\int \Phi^\varepsilon \cdot v^\varepsilon dx
+\int \nabla (p^\varepsilon-p)\cdot v^\varepsilon dx\\
&=-\mu \varepsilon \int \nabla \times (\nabla \times v^\varepsilon)\cdot v^\varepsilon dx
+(2\mu+\lambda)\varepsilon \int \nabla {\rm div}v^\varepsilon \cdot v^\varepsilon dx
+\int R_2^\varepsilon \cdot v^\varepsilon dx
+\int R_3^\varepsilon \cdot v^\varepsilon dx.
\end{aligned}
\end{equation}
It is easy to check that
\begin{equation}\label{513}
\int \Phi^\varepsilon \cdot v^\varepsilon dx
\le C \|(\rho, u^\varepsilon, \nabla u^\varepsilon)\|_{L^\infty}
(\|\phi^\varepsilon\|_{L^2}^2+\|v^\varepsilon\|_{L^2}^2).
\end{equation}
Integrating by part and applying the equation \eqref{eq1}$_1$,
we find
\begin{equation}\label{514}
\begin{aligned}
&\int \nabla (p^\varepsilon-p)\cdot v^\varepsilon dx
=-\int (p^\varepsilon-p){\rm div}v^\varepsilon dx\\
&\ge \int \frac{p'(\rho)}{\rho}\phi^\varepsilon
(\phi^\varepsilon_t+u\cdot \nabla \phi^\varepsilon-R_1^\varepsilon)dx
-C(1+\|\nabla u^\varepsilon\|_{L^\infty})\|\phi^\varepsilon\|_{L^2}^2\\
&\ge \frac{d}{dt}\int \frac{p'(\rho)}{2\rho}|\phi^\varepsilon|^2dx
-C(1+\|(\rho, u, \rho^\varepsilon, u^\varepsilon)\|_{W^{1,\infty}})
(\|\phi^\varepsilon\|_{L^2}^2+\|v^\varepsilon\|_{L^2}^2)\\
&\ge \frac{d}{dt}\int \frac{p'(\rho)}{2\rho}|\phi^\varepsilon|^2dx
-C(\|\phi^\varepsilon\|_{L^2}^2+\|v^\varepsilon\|_{L^2}^2).
\end{aligned}
\end{equation}
Integrating by part and applying the boundary condition \eqref{bc3},
one arrives at directly
\begin{equation}\label{515}
\begin{aligned}
&-\mu \varepsilon \int \nabla \times (\nabla \times v^\varepsilon)\cdot v^\varepsilon dx\\
&=-\mu \varepsilon\int_{\partial \Omega} n\times (\nabla \times v^\varepsilon)\cdot v^\varepsilon dx-\mu \varepsilon\int |\nabla \times v^\varepsilon|^2 dx\\
&=-\mu \varepsilon\int_{\partial \Omega}
 ([B v^\varepsilon]_\tau+[B u]_\tau-n\times w)\cdot v^\varepsilon d\sigma
  -\mu \varepsilon\int |\nabla \times v^\varepsilon|^2 dx\\
&\le -\mu \varepsilon \|\nabla \times v^\varepsilon\|_{L^2}^2
     +C\varepsilon(|v^\varepsilon|_{L^2(\partial \Omega)}^2
     +|v^\varepsilon|_{L^2(\partial \Omega)}),
\end{aligned}
\end{equation}
and
\begin{equation}\label{516}
(2\mu+\lambda)\varepsilon\int \nabla {\rm div}v^\varepsilon \cdot v^\varepsilon dx
=(2\mu+\lambda)\varepsilon \int |{\rm div}v^\varepsilon|^2 dx.
\end{equation}
On the other hand, by virtue of the H\"{o}lder and Young inequalities, one attains
\begin{equation}\label{517}
\int R_2^\varepsilon \cdot v^\varepsilon dx
\le C\|(\phi^\varepsilon, v^\varepsilon)\|_{L^2}^2+C\varepsilon^2,
\end{equation}
and
\begin{equation}\label{518}
\int R_3^\varepsilon \cdot v^\varepsilon dx
\le \delta \|\Delta \varphi^\varepsilon\|_{L^2}^2
+C_\delta(\|v^\varepsilon\|_{L^2}^2+\|\nabla \varphi^\varepsilon\|_{L^2}^2).
\end{equation}
Substituting \eqref{513}-\eqref{518} into \eqref{512}, we obtain
\begin{equation}\label{519}
\begin{aligned}
&\frac{d}{dt}\int \left(\frac{p'(\rho)}{\rho}|\phi^\varepsilon|^2
+\frac{\rho}{2}|v^\varepsilon|^2\right)dx
+\mu \varepsilon \|\nabla \times v^\varepsilon\|_{L^2}^2
+(2\mu+\lambda)\varepsilon \|{\rm div}v^\varepsilon\|_{L^2}^2\\
&\le C_\delta\|(\phi^\varepsilon, v^\varepsilon, \nabla \varphi^\varepsilon)\|_{L^2}^2
+C\varepsilon(|v^\varepsilon|_{L^2(\partial \Omega)}^2
+|v^\varepsilon|_{L^2(\partial \Omega)})+C\varepsilon^2
+\delta \|\Delta\varphi^\varepsilon\|_{L^2}^2.
\end{aligned}
\end{equation}
The application of Proposition \ref{prop2.1} gives directly
\begin{equation}\label{5110}
\|\nabla v^\varepsilon\|_{H^1}^2
\le C(\|\nabla \times v^\varepsilon\|_{L^2}^2
+\|{\rm div}v^\varepsilon\|_{L^2}^2+\|v^\varepsilon\|_{L^2}^2).
\end{equation}
By virtue of the trace theorem in Proposition \ref{prop2.3}
and Cauchy inequality, one deduces that
\begin{equation}\label{5111}
|v^\varepsilon|^2_{L^2(\partial \Omega)}
\le \delta \|\nabla v^\varepsilon\|_{L^2}^2+C_\delta \|v^\varepsilon\|_{L^2}^2,
\end{equation}
and
\begin{equation}\label{5112}
\begin{aligned}
\varepsilon |v^\varepsilon|_{L^2(\partial \Omega)}
&\le \varepsilon \|v^\varepsilon\|_{L^2}^{\frac{1}{2}}
          \|\nabla v^\varepsilon\|_{L^2}^{\frac{1}{2}}
\le \delta \varepsilon \|\nabla v^\varepsilon\|_{L^2}^2
    +C_\delta \varepsilon \|v^\varepsilon\|_{L^2}^{\frac{2}{3}}\\
&\le \delta \varepsilon \|\nabla v^\varepsilon\|_{L^2}^2
    +C_\delta  \|v^\varepsilon\|_{L^2}^2+\varepsilon^{\frac{3}{2}}.
\end{aligned}
\end{equation}
Then, the combination of \eqref{519}-\eqref{5112} yields that
\begin{equation}\label{5113}
\frac{d}{dt}\int \left(\frac{p'(\rho)}{\rho}|\phi^\varepsilon|^2
+\frac{\rho}{2}|v^\varepsilon|^2\right)dx
+\mu \varepsilon \|v^\varepsilon\|_{H^1}^2
\le C\|(\phi^\varepsilon, v^\varepsilon, \nabla \varphi^\varepsilon)\|_{L^2}^2
+C\varepsilon^{\frac{3}{2}}+\delta \|\Delta\varphi^\varepsilon\|_{L^2}^2.
\end{equation}
Multiplying \eqref{eq3} by $-\Delta \varphi^\varepsilon$ and integrating over $\Omega$, we find
\begin{equation}\label{5114}
\begin{aligned}
-\int \partial_t \varphi^\varepsilon \cdot \Delta \varphi^\varepsilon dx
+\int |\Delta \varphi^\varepsilon|^2 dx
=
-\int R_3^\varepsilon \cdot \Delta \varphi^\varepsilon dx.
\end{aligned}
\end{equation}
Integrating by part and applying the boundary condition \eqref{bc3}, it holds that
\begin{equation}\label{5115}
\begin{aligned}
-\int \partial_t \varphi^\varepsilon \cdot \Delta \varphi^\varepsilon dx
&=-\int_{\partial \Omega} \partial_t \varphi^\varepsilon \cdot (n\cdot \nabla \varphi^\varepsilon) d\sigma
+\frac{1}{2}\frac{d}{dt}\int |\nabla \varphi^\varepsilon|^2 dx
=\frac{1}{2}\frac{d}{dt}\int |\nabla \varphi^\varepsilon|^2 dx.
\end{aligned}
\end{equation}
Applying the Cauchy inequality, it is easy to deduce that
\begin{equation}\label{5116}
\begin{aligned}
-\int R_2^\varepsilon \cdot \Delta \varphi^\varepsilon dx
&\le \delta \|\Delta \varphi^\varepsilon\|_{L^2}^2
     +C_\delta \|u\|_{L^\infty}^2\|\nabla \varphi^\varepsilon\|_{L^2}^2
     +C_\delta \|\nabla d^\varepsilon\|_{L^\infty}^2
     (\|\nabla \varphi^\varepsilon\|_{L^2}^2+\|v^\varepsilon\|_{L^2}^2)\\
&\quad   +C_\delta \|\nabla d\|_{L^\infty}^2
     (\|\nabla \varphi^\varepsilon\|_{L^2}^2+\|\varphi^\varepsilon\|_{L^2}^2)\\
&\le \delta \|\Delta \varphi^\varepsilon\|_{L^2}^2
     +C_\delta(\|v^\varepsilon\|_{L^2}^2+
          \|\varphi^\varepsilon\|_{L^2}^2+\|\nabla \varphi^\varepsilon\|_{L^2}^2).
\end{aligned}
\end{equation}
Substituting \eqref{5115}-\eqref{5116} into \eqref{5114}, we obtain
\begin{equation}\label{5117}
\begin{aligned}
\frac{1}{2}\frac{d}{dt}\int |\nabla \varphi^\varepsilon|^2 dx
+\frac{3}{4}\int |\Delta \varphi^\varepsilon|^2 dx
\le
C(\|v^\varepsilon\|_{L^2}^2+
          \|\varphi^\varepsilon\|_{L^2}^2+\|\nabla \varphi^\varepsilon\|_{L^2}^2).
\end{aligned}
\end{equation}
In order to control the term $\int |\varphi^\varepsilon|^2 dx$ on the right hand side of
\eqref{5117}, we multiply the equation \eqref{eq3}$_3$ by
$\varphi^\varepsilon$ and integrating by part to get that
\begin{equation}\label{5118}
\frac{1}{2}\frac{d}{dt}\int |\varphi^\varepsilon|^2 dx
+\int |\nabla \varphi^\varepsilon|^2 dx
=\int R_4^\varepsilon \cdot \varphi^\varepsilon dx.
\end{equation}
In view of the H\"{o}lder inequality, one arrives at
\begin{equation*}
\begin{aligned}
\int R_4^\varepsilon \cdot \varphi^\varepsilon dx
&\le \|u\|_{L^\infty}\|\varphi^\varepsilon\|_{L^2}\|\nabla \varphi^\varepsilon\|_{L^2}
+\|\nabla d^\varepsilon\|_{L^\infty}
     (\|v^\varepsilon\|_{L^2}\|\varphi^\varepsilon\|_{L^2}
      +\|\varphi^\varepsilon\|_{L^2}\|\nabla \varphi^\varepsilon\|_{L^2})\\
&\quad  + \|\nabla d\|_{L^\infty}
     \|\nabla \varphi^\varepsilon\|_{L^2}\|\varphi^\varepsilon\|_{L^2}
      +\|\nabla d\|_{L^\infty}^2\|\varphi^\varepsilon\|_{L^2}^2\\
&\le
C(\|v^\varepsilon\|_{L^2}^2+
          \|\varphi^\varepsilon\|_{L^2}^2+\|\nabla \varphi^\varepsilon\|_{L^2}^2),
\end{aligned}
\end{equation*}
which, together with \eqref{5118}, yields directly
\begin{equation}\label{5119}
\frac{1}{2}\frac{d}{dt}\int |\varphi^\varepsilon|^2 dx
+\int |\nabla \varphi^\varepsilon|^2 dx
\le C(\|v^\varepsilon\|_{L^2}^2
+\|\varphi^\varepsilon\|_{L^2}^2
+\|\nabla \varphi^\varepsilon\|_{L^2}^2).
\end{equation}
Then the combination of \eqref{5113}, \eqref{5117} and \eqref{5119} yields immediately
\begin{equation*}
\begin{aligned}
&\frac{d}{dt}\!\!\int \left(\frac{p'(\rho)}{\rho}|\phi^\varepsilon|^2
+\frac{\rho}{2}|v^\varepsilon|^2+\frac{1}{2}|\varphi^\varepsilon|^2
+\frac{1}{2}|\nabla  \varphi^\varepsilon|^2\right)dx
+\mu \varepsilon \|v^\varepsilon\|_{H^1}^2
+\frac{3}{4}\int (|\nabla \varphi^\varepsilon|^2+|\Delta \varphi^\varepsilon|^2) dx\\
&\le C(\|\phi^\varepsilon\|_{L^2}^2+\|v^\varepsilon\|_{L^2}^2
+\|\varphi^\varepsilon\|_{L^2}^2+\|\nabla \varphi^\varepsilon\|_{L^2}^2)
+C\varepsilon^{\frac{3}{2}}.
\end{aligned}
\end{equation*}
which, together with the Gr\"{o}nwall inequality,
completes the proof of Lemma \ref{lemma5.1}.
\end{proof}

\begin{lemm}\label{lemma5.2}
For $t \in [0, \min\{T_0, T_1\}]$, it holds that
\begin{equation}\label{521}
\underset{0\le \tau \le t}{\sup}\|\Delta \varphi^\varepsilon(\tau)\|_{L^2}^2
+\int_0^t \|\nabla \Delta \varphi^\varepsilon\|_{L^2}^2d\tau
\le C\varepsilon^{\frac{1}{2}}.
\end{equation}
\end{lemm}
\begin{proof}
Taking $\nabla$ operator to \eqref{eq3}$_3$, we find
\begin{equation*}
\nabla \varphi^\varepsilon-\nabla \Delta \varphi^\varepsilon
=\nabla R_4^\varepsilon,
\end{equation*}
which, multiplying by $-\nabla \varphi^\varepsilon$, reads
\begin{equation}\label{522}
-\int \partial_t \nabla \varphi^\varepsilon\cdot \nabla \Delta \varphi^\varepsilon dx
+\int |\nabla \Delta \varphi^\varepsilon|^2 dx
=-\int \nabla R_4^\varepsilon \cdot \nabla \Delta \varphi^\varepsilon dx.
\end{equation}
Integrating by part and applying the boundary condition \eqref{bc3}, 
it is easy to deduce
\begin{equation}\label{523}
-\int \partial_t \nabla \varphi^\varepsilon\cdot \nabla \Delta \varphi^\varepsilon dx
=-\int_{\partial \Omega}
n\cdot \nabla \varphi^\varepsilon\cdot \nabla \Delta \varphi^\varepsilon d\sigma
+\frac{1}{2}\frac{d}{dt}\int |\Delta \varphi^\varepsilon|^2 dx
=\frac{1}{2}\frac{d}{dt}\int |\Delta \varphi^\varepsilon|^2 dx.
\end{equation}
On the other hand, it is easy to check that
\begin{equation}\label{524}
\begin{aligned}
\| \nabla R_4^\varepsilon\|_{L^2}^2
&\le C(\|v^\varepsilon\|_{L^2}^2+\|\varphi^\varepsilon\|_{H^1}^2)
    +C(\|\nabla^2 \varphi^\varepsilon\|_{L^2}^2+\|\nabla v^\varepsilon\|_{L^2}^2)\\
&\le C(\|v^\varepsilon\|_{L^2}^2+\|\varphi^\varepsilon\|_{H^1}^2)
    +C(\|\Delta \varphi^\varepsilon\|_{L^2}^2+\|\nabla v^\varepsilon\|_{L^2}^2),
\end{aligned}
\end{equation}
where we have used the standard elliptic estimates in the last inequality.
Hence, by virtue of the Cauchy inequality, \eqref{523} and \eqref{524}, we obtain
\begin{equation}\label{525}
\begin{aligned}
&\frac{1}{2}\frac{d}{dt}\int |\Delta \varphi^\varepsilon|^2 dx
+\int |\nabla \Delta \varphi^\varepsilon|^2 dx\\
&\le \delta \|\nabla \Delta \varphi^\varepsilon\|_{L^2}^2
   +C(\|v^\varepsilon\|_{L^2}^2+\|\varphi^\varepsilon\|_{H^1}^2)
    +C(\|\Delta \varphi^\varepsilon\|_{L^2}^2+\|\nabla v^\varepsilon\|_{L^2}^2).
\end{aligned}
\end{equation}
Choosing $\delta$ small enough in \eqref{525} and integrating over $[0, t]$, one attains
\begin{equation*}
\begin{aligned}
&\int |\Delta \varphi^\varepsilon(t)|^2 dx
+\int_0^t \|\nabla \Delta \varphi^\varepsilon\|_{L^2}^2 d\tau\\
&\le C(\|v^\varepsilon\|_{L^2}^2+\|\varphi^\varepsilon\|_{H^1}^2)
    +C\int_0^t(\|\Delta \varphi^\varepsilon\|_{L^2}^2+\|\nabla v^\varepsilon\|_{L^2}^2)d\tau
\le C \varepsilon^{\frac{1}{2}},
\end{aligned}
\end{equation*}
where we have used the estimate \eqref{511} in the last inequality.
Therefore, we complete the proof of Lemma \ref{lemma5.2}.
\end{proof}

\begin{lemm}\label{lemma5.3}
For $t \in [0, \min\{T_0, T_1\}]$, it holds that
\begin{equation}\label{531}
\begin{aligned}
&\|({\rm div}v^\varepsilon, \nabla (p^\varepsilon-p))\|_{L^2}^2
+(2\mu+\lambda)\varepsilon\int_0^t \|\nabla {\rm div}v^\varepsilon(\tau)\|_{L^2}^2 d\tau\\
&\le \delta \int_0^t \|v_t^\varepsilon\|_{L^2}^2 d\tau
+C_\delta \int_0^t \|(\varphi^\varepsilon, v^\varepsilon)\|_{H^1}^2 d\tau
+C_\delta \varepsilon^{\frac{1}{2}}.
\end{aligned}
\end{equation}
\end{lemm}
\begin{proof}
Multiplying \eqref{eq3}$_2$ by $\nabla {\rm div}v^\varepsilon$, it is easy to deduce that
\begin{equation}\label{532}
\begin{aligned}
&\underset{\Rmnum{8}_1}{\underbrace{\int(\rho v_t^\varepsilon+\rho u\cdot \nabla v^\varepsilon)dx}}
+\underset{\Rmnum{8}_2}{\underbrace{\int \nabla (p^\varepsilon-p)\cdot \nabla {\rm div}v^\varepsilon dx}}
+\underset{\Rmnum{8}_3}{\underbrace{\int \Phi^\varepsilon\cdot \nabla {\rm div}v^\varepsilon dx}}\\
&=\underset{\Rmnum{8}_4}{\underbrace{-\mu \varepsilon \int \nabla \times (\nabla \times v^\varepsilon)\cdot
  \nabla {\rm div}v^\varepsilon dx}}
   +\underset{\Rmnum{8}_5}{\underbrace{(2\mu+\lambda)\varepsilon\int|\nabla {\rm div}v^\varepsilon|^2 dx}}\\
&\quad +\underset{\Rmnum{8}_6}{\underbrace{\int R_2^\varepsilon
    \cdot \nabla {\rm div}v^\varepsilon dx}}
   +\underset{\Rmnum{8}_7}{\underbrace{\int R_3^\varepsilon \cdot \nabla {\rm div}v^\varepsilon dx}}.
\end{aligned}
\end{equation}
Following the same argument as Lemma 6.2 of \cite{Wang-Xin-Yong}, it is easy to obtain
the following estimates
\begin{equation}\label{533}
\begin{aligned}
&\Rmnum{8}_1\le -\frac{d}{dt}\int \frac{\rho}{2}|{\rm div}v^\varepsilon|^2 dx
+\delta \|v_t^\varepsilon\|_{L^2}^2
+C_\delta \|\nabla v^\varepsilon\|_{L^2}^2+C|v^\varepsilon|_{L^2(\partial \Omega)},\\
&\Rmnum{8}_2 \le -\frac{d}{dt}\int \frac{1}{2\gamma p^\varepsilon}|\nabla (p^\varepsilon-p)|^2dx
+C(1+\|(u^\varepsilon, p^\varepsilon)\|_{W^{1,\infty}})
\|(p^\varepsilon-p, v^\varepsilon)\|_{H^1}^2,\\
&\Rmnum{8}_3 \le C(1+\|(\rho^\varepsilon, u^\varepsilon)\|_{W^{1,\infty}})
(\|(\varphi^\varepsilon, v^\varepsilon)\|_{H^1}^2
+|(\varphi^\varepsilon, v^\varepsilon)|_{L^2(\partial \Omega)}),\\
&\Rmnum{8}_4\le \delta \varepsilon \|\nabla {\rm div}v^\varepsilon\|_{L^2}^2
+C_\delta \varepsilon (1+\|v^\varepsilon\|_{H^1}^2),\\
&\Rmnum{8}_6\le \frac{(2\mu+\lambda)\varepsilon}{8}\|\nabla {\rm div}v^\varepsilon\|_{L^2}^2
+\delta \|v_t^\varepsilon\|_{L^2}^2
+C_\delta(\|(\varphi^\varepsilon, v^\varepsilon)\|_{H^1}^2+\varepsilon^{\frac{3}{2}}).
\end{aligned}
\end{equation}
On the other hand, integrating by parts and applying the Cauchy inequality, one
arrives at directly
\begin{equation}\label{534}
\begin{aligned}
\Rmnum{8}_7
&=-\int_{\partial \Omega} n\cdot (\nabla d^\varepsilon \cdot \Delta \varphi^\varepsilon
+\nabla \varphi^\varepsilon \cdot \Delta d){\rm div}v^\varepsilon d\sigma\\
&\quad +\int {\rm div}(\nabla d^\varepsilon \cdot \Delta \varphi^\varepsilon
+\nabla \varphi^\varepsilon \cdot \Delta d){\rm div}v^\varepsilon dx\\
&=\int {\rm div}(\nabla d^\varepsilon \cdot \Delta \varphi^\varepsilon
+\nabla \varphi^\varepsilon \cdot \Delta d){\rm div}v^\varepsilon dx\\
&\le C(1+\|(\nabla d^\varepsilon, \Delta d^\varepsilon)\|_{L^\infty})
(\|\nabla(\varphi^\varepsilon, v^\varepsilon)\|_{L^2}^2
+\|\Delta\varphi^\varepsilon\|_{L^2}^2
+\|\nabla \Delta \varphi^\varepsilon\|_{L^2}^2).
\end{aligned}
\end{equation}
Substituting \eqref{533} and \eqref{534} into \eqref{532}, we find
\begin{equation}\label{535}
\begin{aligned}
&\frac{d}{dt}\int \left(\frac{\rho}{2}|{\rm div}v^\varepsilon|^2+
\frac{1}{2\gamma p^\varepsilon}|\nabla (p^\varepsilon-p)|^2\right)
+(2\mu+\lambda)\varepsilon \int |\nabla {\rm div}v^\varepsilon|^2 dx\\
&\le \delta \|v_t^\varepsilon\|_{L^2}^2
+C\|(\varphi^\varepsilon, v^\varepsilon)\|_{H^1}^2
+C|(\varphi^\varepsilon, v^\varepsilon)|_{L^2(\partial \Omega)}
+C\varepsilon^{\frac{3}{2}}\\
&\quad +C(\|\nabla \varphi^\varepsilon\|_{L^2}^2
+\|\Delta\varphi^\varepsilon\|_{L^2}^2
+\|\nabla \Delta \varphi^\varepsilon\|_{L^2}^2).
\end{aligned}
\end{equation}
By virtue of the trace theorem in Proposition \ref{prop2.3},  we obtain
\begin{equation}\label{536}
|(\varphi^\varepsilon, v^\varepsilon)|_{L^2}
\le C(\|(\varphi^\varepsilon, v^\varepsilon)\|_{H^1}
+\|(\varphi^\varepsilon, v^\varepsilon)\|_{L^2}^{\frac{2}{3}})
\le (\|(\varphi^\varepsilon, v^\varepsilon)\|_{H^1}+\varepsilon^{\frac{1}{2}}).
\end{equation}
Integrating \eqref{535} over $[0, t]$ and substituting \eqref{536}
into the resulting inequality, we find
\begin{equation*}
\begin{aligned}
&\|({\rm div}v^\varepsilon, \nabla (p^\varepsilon-p))\|_{L^2}^2
+(2\mu+\lambda)\varepsilon\int_0^t \|\nabla {\rm div}v^\varepsilon(\tau)\|_{L^2}^2 d\tau\\
&\le \delta \int_0^t \|v_t^\varepsilon\|_{L^2}^2 d\tau
+C_\delta \int_0^t \|(\phi^\varepsilon, v^\varepsilon)\|_{H^1}^2 d\tau
+C_\delta \varepsilon^{\frac{1}{2}}.
\end{aligned}
\end{equation*}
Therefore, we complete the proof of Lemma \ref{lemma5.3}.
\end{proof}

\begin{lemm}\label{lemma5.4}
For $t \in [0, \min\{T_0, T_1\}]$, it holds that
\begin{equation}\label{541}
\begin{aligned}
&\|\nabla \times v^\varepsilon\|_{L^2}^2
+\varepsilon\int_0^t \|(\nabla \times v^\varepsilon)(\tau)\|_{H^1}^2 d\tau\\
&\le \delta \|\nabla (\phi^\varepsilon, v^\varepsilon)\|_{L^2}^2
+C\delta \int_0^t(\|v_t^\varepsilon\|_{L^2}^2
+\varepsilon\|\nabla^2 v^\varepsilon\|_{L^2}^2)d\tau
+C_\delta \int_0^t \|(\phi^\varepsilon, v^\varepsilon)\|_{H^1}^2d\tau
+C_\delta \varepsilon^{\frac{1}{6}}.
\end{aligned}
\end{equation}
\end{lemm}
\begin{proof}
Multiplying by \eqref{eq3}$_2$ by $\nabla \times (\nabla \times v^\varepsilon)$
yields immediately
\begin{equation}\label{542}
\begin{aligned}
&\underset{\Rmnum{9}_1}{\underbrace{\int \rho^\varepsilon v_t^\varepsilon\cdot \nabla \times (\nabla \times v^\varepsilon)dx}}
+\underset{\Rmnum{9}_2}{\underbrace{\int \nabla (p^\varepsilon-p)\cdot \nabla \times (\nabla \times v^\varepsilon)dx}}\\
&=-\mu \varepsilon \|\nabla \times (\nabla \times v^\varepsilon)\|_{L^2}^2
+(2\mu+\lambda)\varepsilon \int \nabla {\rm div}v^\varepsilon
\cdot \nabla \times (\nabla \times v^\varepsilon) dx\\
&\quad \underset{\Rmnum{9}_3}{\underbrace{-\int \widetilde{\Phi}^\varepsilon
\cdot \nabla \times (\nabla \times v^\varepsilon) dx}}
+\underset{\Rmnum{9}_4}{\underbrace{\int \widetilde{R}_2^\varepsilon \cdot \nabla \times (\nabla \times v^\varepsilon) dx}}
+\underset{\Rmnum{9}_5}{\underbrace{\int {R}_3^\varepsilon \cdot \nabla \times (\nabla \times v^\varepsilon) dx}},
\end{aligned}
\end{equation}
where
\begin{equation*}
\begin{aligned}
&\widetilde{\Phi}^\varepsilon
=\rho^\varepsilon u^\varepsilon \cdot \nabla v^\varepsilon
+(\rho^\varepsilon u^\varepsilon-\rho u)\cdot \nabla u,\\
&\widetilde{R}_2^\varepsilon=
-\phi^\varepsilon u_t+\mu\varepsilon \Delta u
+(\mu+\lambda)\varepsilon \nabla {\rm div}u.
\end{aligned}
\end{equation*}
Following the same argument as Lemma 6.3  of \cite{Wang-Xin-Yong}, it is easy to obtain
the following estimates
\begin{equation}\label{543}
\begin{aligned}
&\Rmnum{9}_1 \ge \frac{d}{dt}\left\{\int\frac{\rho^\varepsilon}{2}
|\nabla \times v^\varepsilon|^2dx+\int_{\partial \Omega}
\left(\frac{\rho^\varepsilon}{2}v^\varepsilon (B v^\varepsilon)
+\rho^\varepsilon v^\varepsilon \cdot(Bu-n\times w)\right)d\sigma\right\}\\
&\quad \quad -\delta \|v_t^\varepsilon\|_{L^2}^2
-C_\delta(\|v^\varepsilon\|_{H^1}^2+|v^\varepsilon|_{L^2}),\\
&|\Rmnum{9}_2|\le C(\|p^\varepsilon-p\|_{H^1}^2+\|v^\varepsilon\|_{H^1}^2
+|p^\varepsilon-p|_{L^2}),\\
&|\Rmnum{9}_3|\le C(\|(\phi^\varepsilon, v^\varepsilon)\|_{H^1}^2
+|(\phi^\varepsilon, v^\varepsilon)|_{L^2}),\\
&|\Rmnum{9}_4|\le \delta \varepsilon \|\nabla \times (\nabla \times v^\varepsilon)\|_{L^2}^2
+C_\delta(\|(\phi^\varepsilon, v^\varepsilon)\|_{H^1}^2
+|(\phi^\varepsilon, v^\varepsilon)|_{L^2}+\varepsilon^{\frac{3}{2}}).
\end{aligned}
\end{equation}
On the other hand, integrating by part and applying the boundary condition \eqref{bc3},
we find
\begin{equation}\label{544}
\begin{aligned}
\Rmnum{9}_5
&=\int_{\partial \Omega}R_3^\varepsilon
\cdot (n\times(\nabla \times v^\varepsilon))d\sigma
+\int \nabla \times R_3^\varepsilon  \cdot \nabla \times v^\varepsilon dx\\
&=\int_{\partial \Omega}R_3^\varepsilon \cdot [B v^\varepsilon]_\tau d\sigma
+\int_{\partial \Omega}R_3^\varepsilon \cdot ([B u]_\tau-n \times w)d\sigma\\
&\quad +\int \nabla \times R_3^\varepsilon  \cdot \nabla \times v^\varepsilon dx\\
&=\Rmnum{9}_{51}+\Rmnum{9}_{52}+\Rmnum{9}_{53}.
\end{aligned}
\end{equation}
Integrating by part and applying the H\"{o}lder inequality, we obtain
\begin{equation}\label{545}
\begin{aligned}
\Rmnum{9}_{51}
&=\int_{\partial \Omega}(n\times R_3^\varepsilon)
\cdot (n\times [B v^\varepsilon]_\tau) d\sigma\\
&=\int_{\partial \Omega}(n\times R_3^\varepsilon)
\cdot (n\times (B v^\varepsilon)) d\sigma\\
&=\int (\nabla \times R_3^\varepsilon)
\cdot (n\times (B v^\varepsilon)) dx
+\int R_3^\varepsilon \cdot \nabla \times(n\times (B v^\varepsilon)) d\sigma\\
&\le C(\|v^\varepsilon\|_{H^1}^2+\|\nabla \varphi^\varepsilon\|_{L^2}^2
+\|\nabla \Delta \varphi^\varepsilon\|_{L^2}^2).
\end{aligned}
\end{equation}
It is easy to check that
\begin{equation}\label{546}
|\Rmnum{9}_{52}|\le C(|\nabla \varphi^\varepsilon|_{L^2(\partial \Omega)}
+|\Delta \varphi^\varepsilon|_{L^2(\partial \Omega)}),
\end{equation}
and
\begin{equation}\label{547}
|\Rmnum{9}_{53}|\le C(\|\nabla \times v^\varepsilon\|_{L^2}^2
+\|\nabla \varphi^\varepsilon\|_{L^2}^2
+\|\nabla \Delta \varphi^\varepsilon\|_{L^2}^2).
\end{equation}
Then, substituting \eqref{545}-\eqref{547} into \eqref{544} yields
\begin{equation}\label{548}
\begin{aligned}
\Rmnum{9}_{5}
\le C(\|v^\varepsilon\|_{H^1}^2+\|\nabla \varphi^\varepsilon\|_{L^2}^2
+\|\nabla \Delta \varphi^\varepsilon\|_{L^2}^2)
+C(|\nabla \varphi^\varepsilon|_{L^2(\partial \Omega)}
+|\Delta \varphi^\varepsilon|_{L^2(\partial \Omega)}).
\end{aligned}
\end{equation}
The application of the trace theorem in Proposition \ref{prop2.3} yields that
\begin{equation}\label{549}
\begin{aligned}
&|(\phi^\varepsilon, v^\varepsilon)|_{L^2}
\le C\|(\phi^\varepsilon, v^\varepsilon)\|_{H^1}^{\frac{1}{2}}
\|(\phi^\varepsilon, v^\varepsilon)\|_{L^2}^{\frac{1}{2}}
\le C\|(\phi^\varepsilon, v^\varepsilon)\|_{H^1}^2+C\varepsilon^{\frac{1}{2}},\\
&|\nabla \varphi^\varepsilon|_{L^2(\partial \Omega)}
\le C\|\nabla \varphi^\varepsilon\|_{H^1}\le C\varepsilon^{\frac{1}{2}},\\
&|\Delta \varphi^\varepsilon|_{L^2(\partial \Omega)}
\le \|\Delta \varphi^\varepsilon\|_{H^1}^{\frac{1}{2}}
\|\Delta \varphi^\varepsilon\|_{L^2}^{\frac{1}{2}}
\le C\|\nabla \Delta \varphi^\varepsilon\|_{L^2}^2
+C\varepsilon^{\frac{1}{6}}.
\end{aligned}
\end{equation}
Substituting \eqref{543}, \eqref{548} and \eqref{549} into \eqref{542} reads immediately
\begin{equation}\label{5410}
\begin{aligned}
&\frac{d}{dt}\left\{\int\frac{\rho^\varepsilon}{2}
|\nabla \times v^\varepsilon|^2dx+\int_{\partial \Omega}
\left(\frac{\rho^\varepsilon}{2}v^\varepsilon (B v^\varepsilon)
+\rho^\varepsilon v^\varepsilon \cdot(Bu-n\times w)\right)d\sigma\right\}\\
&\quad+\frac{\mu \varepsilon}{2}\int |\nabla \times (\nabla \times v^\varepsilon)|^2 dx\\
&\le C\delta \|v_t^\varepsilon\|_{L^2}^2
+C\delta \varepsilon \|\nabla^2 v^\varepsilon\|_{L^2}^2
+C_\delta(\|(\phi^\varepsilon, v^\varepsilon)\|_{H^1}^2+\varepsilon^{\frac{1}{6}}).
\end{aligned}
\end{equation}
In view of the Proposition \ref{prop2.1}, one arrives at
\begin{equation}\label{5411}
\begin{aligned}
\|\nabla \times v^\varepsilon\|_{H^1}^2
&\le C_1(\|\nabla \times (\nabla \times v^\varepsilon)\|_{L^2}^2
+\|{\rm div}(\nabla \times v^\varepsilon)\|_{L^2}^2
+\|\nabla \times v^\varepsilon\|_{L^2}^2
+|n\times (\nabla \times v^\varepsilon)|_{H^{\frac{1}{2}}}^2)\\
&\le C_1(\|\nabla \times (\nabla \times v^\varepsilon)\|_{L^2}^2
+\|\nabla \times v^\varepsilon\|_{L^2}^2
+|B v^\varepsilon|_{H^{\frac{1}{2}}}^2
+|(Bu)_\tau-n\times w|_{H^{\frac{1}{2}}}^2).
\end{aligned}
\end{equation}
By virtue of the trace inequality in Proposition \ref{prop2.3}, we have
\begin{equation}\label{5412}
\|(\phi^\varepsilon, v^\varepsilon)\|_{L^2}^2
\le C\|(\phi^\varepsilon, v^\varepsilon)\|_{H^1}
\|(\phi^\varepsilon, v^\varepsilon)\|_{L^2}
\le \delta \|\nabla (\phi^\varepsilon, v^\varepsilon)\|_{L^2}^2
+C_\delta \varepsilon^{\frac{3}{2}}.
\end{equation}
Substituting \eqref{5411} and \eqref{5412} into \eqref{5410}
and integrating the resulting inequality over $[0, t]$
yield the estimate \eqref{541}.
Therefore, we complete the proof of Lemma \ref{lemma5.4}.
\end{proof}

\emph{\bf{Proof for Theorem \ref{Theorem1.2}:}}
By virtue of Proposition \ref{prop2.1}, we have
\begin{equation}\label{5a}
\begin{aligned}
\|v^\varepsilon\|_{H^1}^2
&\le C(\|\nabla \times v^\varepsilon\|_{L^2}^2
+\|{\rm div}v^\varepsilon\|_{L^2}^2
+\|v^\varepsilon\|_{L^2}^2+\|v^\varepsilon\cdot n\|_{H^{\frac{1}{2}}}^2)\\
&\le C(\|\nabla \times v^\varepsilon\|_{L^2}^2
+\|{\rm div}v^\varepsilon\|_{L^2}^2
+\|v^\varepsilon\|_{L^2}^2),
\end{aligned}
\end{equation}
and
\begin{equation}\label{5b}
\begin{aligned}
\|v^\varepsilon\|_{H^2}^2
&\le C(\|\nabla \times v^\varepsilon\|_{L^2}^2
+\|{\rm div}v^\varepsilon\|_{H^1}^2
+\|v^\varepsilon\|_{H^1}^2+\|v^\varepsilon\cdot n\|_{H^{\frac{3}{2}}}^2)\\
&\le C(\|\nabla \times v^\varepsilon\|_{H^2}^2
+\|{\rm div}v^\varepsilon\|_{H^1}^2
+\|v^\varepsilon\|_{H^1}^2).
\end{aligned}
\end{equation}
On the other hand, it follows from the equation \eqref{eq3}$_2$ that
\begin{equation}\label{5c}
\|v_t^\varepsilon\|_{L^2}^2
\le C(\|(\phi^\varepsilon, v^\varepsilon)\|_{H^1}^2
+\varepsilon^2 \|\nabla^2 v^\varepsilon\|_{L^2}^2+\varepsilon^{\frac{1}{2}}).
\end{equation}
The combination  of \eqref{531}, \eqref{541}, \eqref{5a}-\eqref{5c} and
choosing $\delta$ small enough, one obtains that
\begin{equation*}
\|\nabla(v^\varepsilon, p^\varepsilon-p)\|_{L^2}^2
+\varepsilon\int_0^t \|v^\varepsilon(\tau)\|_{H^2}^2 d\tau
\le C\int_0^t \|\nabla(v^\varepsilon, p^\varepsilon-p)\|_{L^2}^2 d\tau
+C\varepsilon^{\frac{1}{6}},
\end{equation*}
which, together with the Gr\"{o}nwall inequality, gives
\begin{equation}\label{5d}
\|\nabla(v^\varepsilon, p^\varepsilon-p)\|_{L^2}^2
+\varepsilon\int_0^t \|v^\varepsilon(\tau)\|_{H^2}^2 d\tau
\le C\varepsilon^{\frac{1}{6}}.
\end{equation}
On the other hand, by virtue of Sobolev inequality, uniform estimate \eqref{111}
and convergence rate \eqref{511}, it is easy to deduce
\begin{equation}\label{5e}
\|(\rho^\varepsilon-\rho, u^\varepsilon-u)\|_{L^\infty(0, T_2; L^\infty(\Omega))}
\le C\|(\rho^\varepsilon-\rho, u^\varepsilon-u)\|_{L^2}^{\frac{2}{5}}
\|(\rho^\varepsilon-\rho, u^\varepsilon-u)\|_{W^{1,\infty}}^{\frac{3}{5}}\le C\varepsilon^{\frac{3}{10}},
\end{equation}
and
\begin{equation}\label{5f}
\|d^\varepsilon-d\|_{L^\infty(0,T^2; W^{1,\infty}(\Omega))}
\le C\|d^\varepsilon-d\|_{H^1}^{\frac{2}{5}}
\|d^\varepsilon-d\|_{W^{2,\infty}}^{\frac{3}{5}}\le C\varepsilon^{\frac{3}{10}},
\end{equation}
The combination of \eqref{511}, \eqref{521} and \eqref{5d}-\eqref{5f}
completes the proof of Theorem \ref{Theorem1.2}

\section*{Acknowledgements}
The author Jincheng Gao would like to thank Yong Wang for fruitful discussion
and suggestion.


\begin{thebibliography}{99}

\bibitem{Huang-Wang-Wen1}
T. Huang, C.Y. Wang, H.Y. Wen,
Strong solutions of the compressible nematic liquid crystal flow,
J. Differential Equations 252 (2012) 2222-2265.

\bibitem{Huang-Wang-Wen2}
T. Huang, C.Y. Wang, H.Y. Wen,
Blow up criterion for compressible nematic liquid crystal flows in dimension three,
Arch. Ration. Mech. Anal. 204 (2012) 285-311.


\bibitem{Hu-Wu}
X.P. Hu, H. Wu
Global solution to the three-dimensional compressible flow of liquid crystals,
SIAM J. Math. Anal. 45 (2013) 2678-2699.

\bibitem{Ding-Huang-Wen-Zi}
S.J. Ding, J.R. Huang, H.Y. Wen, R.Z. Zi,
Incompressible limit of the compressible nematic liquid crystal flow,
J. Funct. Anal. 264 (2013) 1711-1756.


\bibitem{Jiang-Jiang-Wang}
F. Jiang, S. Jiang, D.H. Wang,
Global weak solutions to the equations of compressible flow of nematic liquid
crystals in two dimensions,
Arch. Ration. Mech. Anal. 214 (2014) 403-451.

\bibitem{Lin-Lai-Wawng}
J.Y. Lin, B.S. Lai, C.Y. Wang,
Global finite energy weak solutions to the compressible nematic liquid crystal
flow in dimension three,
SIAM J. Math. Anal. 47 (2015) 2952-2983.

\bibitem{Schade-Shibata}
K. Schade, Y. Shibata,
On strong dynamics of compressible nematic liquid crystals,
SIAM J. Math. Anal. 47 (2015) 3963-3992.

\bibitem{Gao-Tao-Yao}
J.C. Gao, Q. Tao, Z.A. Yao,
Long-time behavior of solution for the compressible nematic liquid
crystal flows in $\mathbb{R}^3$,
J. Differential Equations  261(2016)  2334-2383.

\bibitem{Gie-Kelliher}
G.M. Gie, J.P. Kelliher,
Boundary layer analysis of the Navier-Stokes equations with
generalized Navier boundary conditions,
J. Differential Equations 253 (2012) 1862-1892.

\bibitem{Xiao-Xin1}
Y.L. Xiao, Z.P. Xin,
On 3D Lagrangian Navier-Stokes $\alpha$ model with a class of vorticityslip
boundary conditions,
J. Math. Fluid Mech. 15 (2013) 215-247.


\bibitem{Xiao-Xin}
Y.L. Xiao, Z.P. Xin,
On the vanishing viscosity limit for the 3D Navier-Stokes equations
with a slip boundary condition,
Commun. Pure Appl. Math. 60 (2007) 1027-1055.


\bibitem{Wang-Xin-Yong}
Y. Wang, Z.P. Xin, Y. Yong,
Uniform Regularity and Vanishing Viscosity Limit for the Compressible Navier--Stokes with General Navier-Slip Boundary Conditions in Three-Dimensional Domains,
SIAM J. Math. Anal. 47 (2015) 4123-4191.

\bibitem{Constantin}
P. Constantin,
Note on loss of regularity for solutions of the 3-D incompressible Euler and
related equations,
Comm. Math. Phys. 104 (1986) 311-326.

\bibitem{Constantin-Foias}
P. Constantin, C. Foias,
Navier-Stokes Equation, University of Chicago Press, Chicago, 1988.

\bibitem{Kato}
T. Kato,
Nonstationary flows of viscous and ideal fluids in $\mathbb{R}^3$,
J. Funct. Anal. 9 (1972)  296-305.

\bibitem{Masmoudi}
N. Masmoudi,
Remarks about the inviscid limit of the Navier-Stokes system,
Comm. Math. Phys. 270 (2007) 777-788.

\bibitem{Beira1}
H. Beirao da Veiga,
Vorticity and regularity for flows under the Navier boundary condition,
Commun. Pure Appl. Anal. 5 (2006) 907-918.

\bibitem{Beira2}
H. Beirao da Veiga, F. Crispo,
Concerning the $W^{k,p}$-inviscid limit for 3-d flows
under a slip boundary condition,
J. Math. Fluid Mech. 13 (2011) 117-135.


\bibitem{Masmoudi-Rousset}
N. Masmoudi, F. Rousset,
Uniform Regularity for the Navier-Stokes equation with Navier
boundary condition,
Arch. Ration. Mech. Anal. 203 (2012) 529-575.


\bibitem{Xiao-Xin2}
Y.L. Xiao, Z. P. Xin,
On the inviscid limit of the 3D Navier-Stokes equations with
generalized Navier-slip boundary conditions,
Comm. Math. Stat. 1 (2013) 259-279.


\bibitem{Gao-Guo-Xi}
J.C. Gao, B.L. Guo, X.Y. Xi,
Uniform Regularity and Vanishing Viscosity Limit for the Nematic
Liquid Crystal Flows in Three Dimensional Domain,
arXiv preprint arXiv:1606.03914, 2016.


\bibitem{Paddick}
M. Paddick,
The strong inviscid limit of the isentropic compressible Navier¨CStokes equations
with Navier boundary conditions,
preprint, arXiv:1410.2811v1.



\bibitem{Xin-Yanagisawa}
Z.P. Xin, T. Yanagisawa, Zero-viscosity limit of the linearized Navier¨CStokes equations
for a compressible viscous fluid in the half-plane,
Comm. Pure Appl. Math. 52 (1999) 479-541.

\bibitem{Wang-Williams}
Y.G. Wang, M. Williams,
The inviscid limit and stability of characteristic boundary layers
for the compressible Navier¨CStokes equations with Navier-friction
boundary conditions,
Ann. Inst. Fourier  62 (2013) 2257-2314.


\bibitem{Temam}
R. Temam,
Navier-Stokes Equations: Theory and Numerical Analysis,
AMS, Providence, RI, 1984.

\bibitem{Gues}
O. Gues,
Probleme mixte hyperbolique quasi-lineaire caracteristique,
Comm. Partial Differential Equations 15 (1990) 595-645.

\bibitem{Cho-Choe-Kim}
Y.G. Cho, H.J. Choe, H. Kim,
Unique solvability of the initial boundary value problems
for compressible viscous fluids,
J. Math. Pures Appl. 83 (2004) 243-275.


\bibitem{Simon}
J. Simon,
Compact sets in the space $L^p(0, T;B)$,
Ann. Mat. Pura Appl. 146 (1987) 65-96.

\end{thebibliography}
\end{document}